\documentclass{amsart}

\usepackage{amsfonts,amsthm,amssymb}
\usepackage[pdftex]{graphicx}
\usepackage{graphics}
\usepackage[matrix,arrow]{xy}
\usepackage[active]{srcltx}
\usepackage[colorlinks=true]{hyperref}

\newtheorem{thm}{Theorem}
\numberwithin{thm}{section}
\newtheorem{theorem}[thm]{Theorem}
\newtheorem*{theorem*}{Theorem}

\newtheorem{corollary}[thm]{Corollary}
\newtheorem*{corollary*}{Corollary}
\newtheorem{lemma}[thm]{Lemma}

\newtheorem{prop}[thm]{Proposition}
\newtheorem{proposition}[thm]{Proposition}

\newtheorem*{conjecture*}{Conjecture}

\newtheorem*{question*}{Question}

\theoremstyle{definition}
\newtheorem{defn}[thm]{Definition}
\newtheorem{definition}[thm]{Definition}

\newtheorem*{definitions*}{Definitions}

\theoremstyle{remark}

\newtheorem*{rem*}{Remark}
\newtheorem{remark}[thm]{Remark}

\newtheorem*{remark*}{Remark}

\newtheorem*{remarks*}{Remarks}
\newtheorem*{example*}{Example}

\newtheorem*{examples*}{Examples}

\newcommand{\acknowledgments}{{\noindent \em Acknowledgments.} }

\newcommand{\R}{\mathbb{R}}
\newcommand{\Z}{\mathbb{Z}}
\newcommand{\Q}{\mathbb{Q}}
\newcommand{\C}{\mathbb{C}}
\newcommand{\N}{\mathbb{N}}

\newcommand{\ep}{\epsilon}
\newcommand{\ga}{\gamma}

\def\lf{\lfloor}
\def\rf{\rfloor}
\def\A{{\mathcal A}}
\def\cz{{\mu_\text{CZ}}}
\renewcommand{\H}{\mathcal{H}}
\newcommand{\J}{\mathcal{J}}

\newcommand{\V}{\mathcal{V}}
\newcommand{\U}{\mathcal{U}}
\newcommand{\jtil}{\widetilde J}
\newcommand{\jbar}{\bar J}

\newcommand{\util}{\widetilde u}
\newcommand{\vtil}{\widetilde v}
\newcommand{\F}{\mathcal{F}}
\newcommand{\M}{\mathcal{M}}
\renewcommand{\P}{\mathcal{P}}
\newcommand{\D}{\mathbb{D}}
\newcommand{\pt}{{\rm pt}}
\newcommand{\mchi}{\hat\chi}

\begin{document}

\title{Local contact homology and applications}

\author[U. Hryniewicz]{Umberto L. Hryniewicz}
\address{Universidade Federal do Rio de Janeiro, Instituto de Matem\'atica,
Cidade Universit\'aria, CEP 21941-909 - Rio de Janeiro - Brazil}
\email{umberto@labma.ufrj.br}

\author[L. Macarini]{Leonardo Macarini}
\address{Universidade Federal do Rio de Janeiro, Instituto de Matem\'atica,
Cidade Universit\'aria, CEP 21941-909 - Rio de Janeiro - Brazil}
\email{leonardo@impa.br}

\setcounter{tocdepth}{1}

\date{\today}

\begin{abstract} 
We introduce a local version of contact homology for an isolated periodic orbit of the Reeb flow and prove that its rank is uniformly bounded for isolated iterations. Several applications are obtained, including a generalization of Gromoll-Meyer's theorem on the existence of infinitely many simple periodic orbits, resonance relations and conditions for the existence of non-hyperbolic periodic orbits.
\end{abstract}

\maketitle

\tableofcontents

\section{Introduction}

Since the seminal works of Floer~\cite{Fl1,Fl2,Fl3} several Morse theoretic tools based on elliptic PDEs have been developed in order to study global properties of Hamiltonian systems. It is hinted at in~\cite{Fl1} that Floer homology should be intuitively thought of as the homology of a suitable Conley index for the ``unregularized'' gradient flow of the Hamiltonian action functional. Loosely speaking, the most immediate isolated invariant set is the set of all bounded trajectories, {\it i.e.}, solutions of Floer's equation with finite energy. It becomes clear from Floer's work that one must look for smaller isolated invariant sets in order to obtain tools for computing, estimating and applying Floer homology to prove existence/multiplicity results for periodic orbits.

Local Floer homology of isolated orbits is a successful instance of this program: an isolated periodic orbit of a Hamiltonian system carries such a homological Conley index, which is defined by counting bounded trajectories connecting orbits that bifurcate after a small perturbation of the system. Many applications of local Floer homology have been obtained, recently the most notorious one being Ginzburg's proof of the Conley conjecture~\cite{Gi}.

The same program should be carried out for contact homology, which was introduced in~\cite{EGH} inside the bigger framework of Symplectic Field Theory (SFT). The purpose of this work is to define local contact homology of isolated closed Reeb orbits and provide a number of applications.

\subsection{Main result}

Let $(N^{2n-1},\xi)$ be a closed co-oriented contact manifold, {\it i.e.}, $\xi\subset TN$ is the hyperplane distribution given by $\xi = \ker\alpha$, for some $1$-form $\alpha$ on $N$ such that $\alpha \wedge (d\alpha)^{n-1}$ is a volume form. We call $\alpha$ a defining contact form for $\xi$ when it induces the given co-orientation. The associated Reeb vector field $X$ is determined by $i_Xd\alpha=0$, $i_X\alpha=1$. By a periodic, or closed, Reeb orbit we mean a pair $\gamma = (x,T)$ where $T>0$ and $x$ is a $T$-periodic trajectory of $X$. Each such $\gamma=(x,T)$ can be identified with a point of $C^\infty(S^1,N)/S^1$ given by the loop $t \in S^1 \simeq \R/\Z \mapsto x(Tt) \in N$. One calls $\gamma$ isolated if it is an isolated point of the set of closed Reeb orbits in $C^\infty(S^1,N)/S^1$.

In this work we associate to $\gamma$ its local contact homology $HC_*(\alpha,\gamma)$ by making a small nondegenerate perturbation of $\alpha$ and counting rigid finite-energy holomorphic cylinders in the symplectization of an isolating neighborhood $K$ of $\gamma$, asymptotic to good periodic orbits which are homotopic to $\gamma$ in $K$. See Section~\ref{lch_section} for details. 

Our first main result establishes a uniform bound for the rank of local contact homology under iterations of a periodic orbit.

\begin{theorem}\label{main}
Let $\gamma$ be a periodic orbit of the Reeb flow such that $\gamma^j$ is isolated for every $j \in \N$. Then there exists a constant $B>0$ satisfying $\dim HC_l(\alpha,\gamma^j) < B$ for every $j \in \N$ and $l\in\Z$.
\end{theorem}

Theorem~\ref{main} has numerous applications concerning global properties of Reeb dynamics, out of which we emphasize Theorem~\ref{inf_orbits} and its Corollary~\ref{corGM} generalizing a celebrated theorem of Gromoll and Meyer~\cite{GM}. Other applications are Theorem~\ref{grate} which provides a criterion to get infinitely many closed orbits in terms of growth rates, Theorem~\ref{thm:mec} generalizing resonance relations obtained by Ginzburg and Kerman~\cite{GK} to degenerate contact forms, Theorem~\ref{thm: hyperbolic1}, its Corollary~\ref{cor:spheres} and Theorem~\ref{thm: hyperbolic2} which exploit local contact homology to obtain criteria for existence of non-hyperbolic orbits.

It is the result of joint work of Viktor Ginzburg, Doris Hein and the authors~\cite{ghhm} that the arguments used in the proof of the Conley conjecture~\cite{Gi} can be brought to the realm of Reeb flows through local contact homology to deduce the existence of two periodic orbits for Reeb flows on the tight $3$-sphere, a result independently obtained by Cristofaro-Gardiner and Hutchings~\cite{CGH}.

Theorem~\ref{main} is based on two main building blocks. The first is Proposition~\ref{est_lch_prop} establishing that the rank of local contact homology of an isolated periodic orbit is less than or equal to the rank of the local Floer homology of the associated local return map. The second is the main result of~\cite{GG} giving a uniform bound for the rank of local Floer homology of admissible iterations. It should be mentioned that we prove a significantly stronger statement than Theorem~\ref{main}: if the isolated orbit $\gamma$ has multiplicity $m$ then $\Z_m := \Z/m\Z$ acts by chain maps on the chain complex of local Floer homology of the local return map, and local contact homology turns out to be isomorphic to the homology of the subcomplex of $\Z_m$-invariant chains.

Theorem~\ref{main} and its applications rely on the interplay between local and linearized contact homologies. The definitions of both homologies encounter well-known transversality problems. Hence, Theorem~\ref{main} and its dynamical applications, and also the results from~\cite{ghhm}, are conditional on completion of foundational work by Hofer, Wysocki and Zehnder~\cite{HWZ1,HWZ2,HWZ3} setting up the analytical stage where the relations between local and linearized contact homologies can be studied.

\subsection{Applications}

Our applications are very much in the spirit of Symplectic Dynamics, as explained by Bramham and Hofer~\cite{BH}. Modern methods in symplectic geometry, like pseudo-holomorphic curve theory, have proved to be successful in studying global properties of Hamiltonian systems. 

Here we will explore these ideas, more precisely, we use contact homology as introduced by Eliashberg, Givental and Hofer~\cite{EGH}. There are different versions of contact homology. We consider cylindrical and linearized contact homology. The former is an invariant of contact structures defined by contact forms satisfying restrictive dynamical assumptions on the Reeb flow. The definition of the latter, in turn, does not impose dynamical restrictions but requires extra geometric data and a more elaborate construction. 

If $\alpha$ is a nondegenerate defining contact form for $\xi$, co-orientations considered, then, with the help of a suitable regular cylindrical almost complex structure $J$ on the symplectization of $\xi$, one can define a differential graded algebra (DGA) $(\A(\alpha),d)$ associated to the pair $(\alpha,J)$. Here $\A(\alpha)$ is the graded supercommutative unital algebra freely generated by the good closed $\alpha$-Reeb orbits with coefficients in a suitable ring. The grading is given by the reduced Conley-Zehnder index. See Section~\ref{preliminaries} for the precise definitions of good/bad closed Reeb orbits, their Conley-Zehnder indices, and for the precise geometric properties of $J$. The differential $d$ is a degree $-1$ derivation on $\A(\alpha)$ defined by the algebraic count of rigid finite-energy holomorphic spheres with one positive puncture in the symplectization, see~\cite{EGH}. Throughout this work, we will {\it always assume that the first Chern class of the contact structure vanishes}, so that one can use $\Q$ as the coefficient ring. The homology of this DGA depends only on the contact structure and is referred to as the full contact homology of $\xi$ with rational coefficients, see~\cite{Bo} for a survey.

An augmentation on $(\A(\alpha),d)$ is an algebra homomorphism $\ep : \A(\alpha) \to \Q$ satisfying 
\begin{equation}
\begin{array}{cccc} \ep(1)=1, & \ep \circ d = 0 & \text{and} & \ep(\gamma)=0 \ \ \text{if} \ \ |\gamma| \neq 0. \end{array}
\end{equation}
It can be used to linearize $d$ as follows. Let $C(\alpha)$ be the $\Q$-vector space freely generated by the good closed $\alpha$-Reeb orbits, graded by $|\cdot|$. The algebra $\A(\alpha)$ can be decomposed according to word length $\A(\alpha) = \A_0 \oplus \A_1 \oplus \A_2 \oplus \dots$ where $\A_0=\Q$. If $\pi_1$ is the projection onto $\A_1$ and $S^{\epsilon} : \A(\alpha) \to \A(\alpha)$ is the algebra homomorphism determined by $S^{\epsilon}(1)=1$ and $S^{\epsilon}(\gamma) = \gamma + \epsilon(\gamma)$ on generators, then the linearized differential is defined by
\begin{equation*}
\begin{array}{ccc}
d_{\epsilon} : C_*(\alpha) \to C_{*-1}(\alpha) & & d_{\epsilon} = \pi_1 \circ S^{\epsilon} \circ d.
\end{array}
\end{equation*}
Then $(C(\alpha),d_{\epsilon})$ is a chain complex and its homology is the so-called linearized contact homology, denoted by $HC^\ep_*(\alpha)$. The dependence on $J$ is not explicit in order to keep the notation simpler. Two augmentations $\epsilon,\epsilon'$ for $(\mathcal A(\alpha),d)$ are said to be homotopic if there exists a degree $+1$ derivation $K$ such that $\epsilon'=\epsilon\circ \Phi$ where $\Phi = e^{dK+Kd}$. Then the linearization $\Phi_\epsilon = \pi_1 \circ S^\epsilon \circ \Phi$ of $\Phi$ is a chain map and an isomorphism between the complexes $(C(\alpha),d_\epsilon)$ and $(C(\alpha),d_{\epsilon'})$. In the following we denote by ${\rm Aug}(\A(\alpha),d)$ the set of all augmentations of $(\A(\alpha),d)$. The above notion of homotopic augmentations is a relation that generates an equivalence relation on ${\rm Aug}(\A(\alpha),d)$, and we denote by $\overline{\rm Aug}(\A(\alpha),d)$ the set of equivalence classes. An element of $\overline{\rm Aug}(\A(\alpha),d)$ will be referred to as a homotopy class of augmentations. As a consequence, linearized homology depends only on the homotopy class of the augmentation, up to isomorphism.

Let us briefly discuss invariance properties of linearized contact homology. Augmentations are defined for the DGA associated to a nondegenerate defining contact form for $\xi$ and a regular cylindrical almost complex structure on the symplectization, but one may consider homotopy classes of augmentations for the {\it co-oriented contact structure} as the following discussion shows; we refer to~\cite{Bo} for some details.

Consider two pairs $(\alpha_0,J_0)$ and $(\alpha_1,J_1)$ as above. There are positive constants $c_0>c_1$ such that $c_0\alpha_0=fc_1\alpha_1$ for some $f:N\to(1,+\infty)$ smooth. If we choose a smooth function $\phi:[-1,1]\to\R$ satisfying $\phi(-1)=0$, $\phi(1)=1$ and $\phi'>0$, then $\Omega=d((1-\phi+\phi f)c_1\alpha_1)$ is an exact symplectic form on $[-1,1]\times N$ satisfying $\Omega|_{T(\{1\}\times N)}=c_0d\alpha_0$, $\Omega|_{T(\{-1\}\times N)}=c_1d\alpha_1$, where we identified $\{\pm 1\}\times N \simeq N$. It is not hard to construct almost complex structures $\jbar$ on $\R\times N$ which are $\Omega$-tamed on $[-1,1]\times N$ and satisfy $\jbar|_{[1,+\infty)\times N} = J_0$, $\jbar|_{(-\infty,-1]\times N} = J_1$. The data $(\Omega,\jbar)$ induces a chain map $\Psi: (\A(\alpha_0),d_0) \to (\A(\alpha_1),d_1)$ between the corresponding DGAs defined by counting rigid $\jbar$-holomorphic finite-energy punctured spheres with one positive puncture. An augmentation $\ep_1$ of $(\A(\alpha_1),d_1)$ can be pulled-back to an augmentation $\Psi^*\ep_1 = \ep_1 \circ \Psi$ of $(\A(\alpha_0),d_0)$, and the linearization $\Psi_{\epsilon_1} = \pi_1 \circ S^{\epsilon_1} \circ \Psi$ of $\Psi$ by $\ep_1$ induces an isomorphism $HC^{\Psi^*\ep_1}(\alpha_0) \simeq HC^{\ep_1}(\alpha_1)$. It turns out that, in fact, the homotopy class of $\Psi^*\ep_1$ depends on the cobordism data only up to a generic homotopy, 
see~\cite[Theorem 3.2]{CH}. Therefore, we end up with a way to identify homotopy classes of augmentations of both DGAs: the homotopy class of the augmentation $\epsilon_0$ of $(\A(\alpha_0),d_0)$ is identified with the homotopy class of the augmentation $\epsilon_1$ of $(\A(\alpha_1),d_1)$ if $\epsilon_0$ and $\Psi^*\epsilon_1$ are homotopic, where $\Psi$ is obtained from a cobordism data as above. 
This relation generates an equivalence relation on $\coprod_{(\alpha,J)} \overline{\rm Aug}(\A(\alpha),d)$, where the disjoint union is taken over all pairs $(\alpha,J)$ satisfying the conditions necessary to get $(\A(\alpha),d)$ well-defined. The resulting equivalence classes will be referred to as homotopy classes of augmentations for the {\it co-oriented contact structure} $\xi$. We will abuse a bit the notation and denote such equivalence classes by $[\ep]$ and its corresponding linearized contact homology by $HC_*^{[\ep]}(\xi)$.

Let $b_*^{[\ep]}(\xi)$ be the rank of $HC_*^{[\ep]}(\xi)$. We say that $\xi$ is homologically unbounded if there is a homotopy class of augmentations $[\ep]$ for $\xi$ and a sequence of integers $l_i$ such that $|l_i| \to \infty$ and $b_{l_i}^{[\ep]}(\xi) \to \infty$ as $i\to\infty$. The following theorem follows easily from Theorem \ref{main} and the Morse inequalities in Proposition~\ref{mi}. Its proof is given in~\ref{inf_orbs} below.

\begin{theorem}
\label{inf_orbits}
Suppose that $\xi$ is homologically unbounded. Then the Reeb flow of every defining contact form for $\xi$ has infinitely many geometrically distinct periodic orbits.
\end{theorem}

The main point of the above theorem is that there are no genericity assumptions on $\alpha$ like being nondegenerate or Morse-Bott. Examples of homologically unbounded contact structures can be obtained by cosphere bundles. More precisely, given a closed oriented manifold $M$ of dimension $n$, it is proved in \cite{CL} that
\begin{equation}
\label{CL}
HC_{*+(n-3)}^{[\ep_0]}(S^*M, \xi_0) \simeq H_*(\Lambda M/S^1, M;\Q),
\end{equation}
where $\xi_0$ is the standard contact structure of the unit cotangent bundle $S^*M$, $\ep_0$ is given by the obvious filling of $S^*M$, $\Lambda M$ is the free loop space on $M$ and $M \subset \Lambda M$ indicates the subset of constant loops \cite[Theorem 4.4]{Bo}.\footnote{The trivializations of the contact structure over periodic orbits used in \cite{CL} send the vertical distribution in the cotangent bundle to a fixed Lagrangian subspace in $\R^{2n-2}$. This fixes the grading in the isomorphism \eqref{CL}.} A result due to Vigu\'e-Poirrier and Sullivan \cite{VS} establishes that if $M$ is simply connected then the rank of $H_*(\Lambda M/S^1, M;\Q)$ is asymptotically unbounded if the cohomology algebra of $M$ is not generated by a single class. Consequently, we have the following generalization of a celebrated result due to Gromoll and Meyer \cite{GM}. It is also proved by Mark McLean~\cite{McL2} for more general coefficients using symplectic homology.

\begin{corollary}\label{corGM}
Let $M$ be a closed oriented manifold and assume that the rank of $H_*(\Lambda M/S^1, M;\Q)$ is asymptotically unbounded. Then every hypersurface in $T^*M$ which is fiberwise starshaped with respect to the zero section has infinitely many, geometrically distinct, periodic orbits. In particular, the result holds if $M$ is simply connected and its cohomology algebra over $\Q$ is not generated by a single class. 
\end{corollary}

Another source of examples is given by connected sums. Given two contact manifolds $(N_1,\xi_1)$ and $(N_2,\xi_2)$ it is well known that its connected sum $N_1 \# N_2$ carries a contact structure $\xi_1 \# \xi_2$, see \cite{Ge}. Moreover, homotopy classes of augmentations $[\ep_1]$ and $[\ep_2]$ for $\xi_1$ and $\xi_2$, respectively, induce a homotopy class of augmentations $[\ep_1 \# \ep_2]$ for $\xi_1 \# \xi_2$. A result due to Bourgeois and van Koert \cite{Bo, BvK} gives the long exact sequence
\begin{align*}
\cdots & \rightarrow HC_{*-1}(S^{2n-3},\xi_0) \rightarrow HC^{[\ep_1 \# \ep_2]}_*(N_1 \# N_2,\xi_1 \# \xi_2) \\
& \rightarrow  HC^{[\ep_1]}_*(N_1,\xi_1) \oplus HC^{[\ep_2]}_*(N_2,\xi_2) \rightarrow HC_{*-2}(S^{2n-3},\xi_0) \rightarrow \cdots
\end{align*}
where $HC_{*}(S^{2n-3},\xi_0)$ is the cylindrical contact homology of the standard contact structure on $S^{2n-3}$. Since the rank of $HC_{*}(S^{2n-3},\xi_0)$ is at most one in any degree, we conclude that the connected sum of any contact manifold (admitting an augmentation) with a homologically unbounded contact manifold is homologically unbounded.

Now, fix a nondegenerate contact form $\alpha$ for $\xi$ with an augmentation $\ep$. There is a natural filtration in contact homology given by the action. Given $a \in \R$ denote the truncated homology by $HC_*^{a,\ep}(\alpha)$ given by the homology of the subcomplex generated by good closed Reeb orbits of action $< a$. In general, it depends on the contact form, a regular cylindrical almost complex structure and the augmentation. Following \cite{McL,Se}, we define the growth rate of $HC^{\ep}(\alpha)$ as
$$ \Gamma^{\ep}(\alpha) = \limsup_{a\to\infty}\frac{1}{\log a}\log \dim \iota(HC^{a,\ep}(\alpha)), $$
where $\iota: HC^{a,\ep}(\alpha) \to HC^{[\ep]}(\xi)$ is the map induced by the inclusion. The argument in \cite[Section 4a]{Se} shows that the set $\{\Gamma^{\ep}(\alpha);\ \ep\text{ is an augmentation for }\alpha\}$ is an invariant of the contact structure. Since our context is different from the one in \cite{Se} (in particular, we have to deal with augmentations) we will give a proof of this fact in Section~\ref{inv:grate}. The following theorem will be proved in Section~\ref{proof:grate}.

\begin{theorem}
\label{grate}
If there exists a nondegenerate contact form $\alpha$ for $\xi$ with an augmentation $\ep$ such that $\Gamma^{\ep}(\alpha) > 1$ then every contact form defining $\xi$ has infinitely many geometrically distinct periodic orbits.
\end{theorem}

There are several examples of contact manifolds satisfying the previous hypothesis, see \cite{MP}.

Our next application is a generalization of a result due to Ginzburg and Kerman \cite{GK} on resonance relations. Assume that there exist integers $l_-$ and $l_+$ such that $HC_l^{[\ep]}(\xi)$ has finite rank for every $l\leq l_-$ and $l\geq l_+$. Under this assumption the positive/negative mean Euler characteristic is defined as
$$
\chi_\pm^{[\ep]}(\xi) = \lim_{m\to\infty} \frac{1}{m} \sum_{l=|l_\pm|}^m (-1)^lb_{\pm l}^{[\ep]}(\xi)
$$
provided that the limits exist. Given an isolated periodic orbit $\gamma$, its positive/negative local Euler characteristic is defined as
$$ \chi_\pm(\alpha,\gamma) = \sum_{\pm i\geq 0} (-1)^i\dim HC_i(\alpha,\gamma). $$
The sum above is finite. Now, assume that $\gamma^j$ is isolated for every $j \in \N$. The positive/negative local {\it mean} Euler characteristic of $\gamma$ is defined as
$$ \mchi_\pm(\alpha,\gamma) = \lim_{m\to\infty} \frac{1}{m} \sum_{j=1}^m \chi_\pm(\alpha,\gamma^j) $$ provided that the limits exist.

\begin{theorem}
\label{thm:mec}
Let $\alpha$ be a contact form for $\xi$ with finitely many simple closed orbits. Given any homotopy class of augmentations $[\ep]$ for $\xi$, the positive/negative mean Euler characteristic satisfies
\[
\chi_\pm^{[\ep]}(\xi) = \sum_{\{ \gamma \text{ such that } \pm\Delta(\gamma) > 0 \}} \frac{\mchi_\pm(\alpha,\gamma)}{\Delta(\gamma)},
\]
where $\Delta(\gamma)$ is the mean index of $\gamma$ and the sum runs over the set of simple periodic orbits $\gamma$ such that $\pm\Delta(\gamma) > 0$, provided that $\chi_\pm^{[\ep]}(\xi)$ and $\mchi_\pm(\alpha,\gamma)$ are defined.
\end{theorem}

Notice that in the previous theorem we do not assume $\alpha$ to be nondegenerate. When $\alpha$ is nondegenerate the local mean Euler characteristic of a periodic orbit is easily computed and we obtain

\begin{corollary}[Theorem 1.7 and Remark 1.10 in \cite{GK}]
\label{cor:mec}
If $\alpha$ is nondegenerate and has finitely many simple periodic orbits then
$$ \chi^{[\ep]}_\pm(\xi) = \frac{1}{2}\sum_{\gamma\in B^\pm(\alpha)} \frac{(-1)^{|\gamma|}}{\Delta(\gamma)} + \sum_{\gamma\in G^\pm(\alpha)} \frac{(-1)^{|\gamma|}}{\Delta(\gamma)}, $$
where $B^\pm(\alpha)$ (resp. $G^\pm(\alpha)$) is the set of simple periodic orbits with positive/negative mean index whose even iterates are bad (resp. good). 
\end{corollary}

An application of the previous theorem is the following result. Recall that a periodic orbit is hyperbolic if its linearized return map has no eigenvalue in the unit circle.

\begin{theorem}
\label{thm: hyperbolic1}
Suppose that there is a homotopy class of augmentations $[\ep]$ for $\xi$ such that $HC_{n-3}^{[\ep]}(\xi)$ has finite rank and that there exists an integer $C>0$ such that
$$ (-1)^{n}\sum_{i=0}^{mC} (-1)^{\pm i+n-3}b_{\pm i+n-3}^{[\ep]}(\xi) < (-1)^{n}mC\chi_\pm^{[\ep]}(\xi) $$
for every $m \in \N$. If a contact form $\alpha$ for $\xi$ has finitely many geometrically distinct closed orbits then there is a non-hyperbolic one.
\end{theorem}

Examples satisfying these hypotheses can be obtained using Yau's computation of the contact homology of subcritical Stein fillable contact manifolds \cite{Y}. More precisely, it is proved in \cite{Y} that given a subcritical Stein domain $(V^{2n},J)$ such that $\partial V = N$ then the cylindrical contact homology is given by
$$ HC_*(\xi) \simeq \oplus_{m\geq0} H_{2(n+m-1)-*}(V), $$
where $\xi$ is the maximal complex subbundle of $TN$. One can check from this that if $n$ is even and $V$ has trivial homology in every odd degree then $N$ satisfies the hypotheses of Theorem~\ref{thm: hyperbolic1} for the positive Euler characteristic. In particular, we get a result obtained by Viterbo in~\cite{Vit}.

\begin{corollary}\label{cor:spheres}
Assume that the Reeb flow associated to a contact form on $S^{4k-1}$ defining its standard contact structure admits finitely many periodic orbits. Then there must be at least one nonhyperbolic periodic orbit.
\end{corollary}

As a byproduct of the proof of Theorem~\ref{thm: hyperbolic1} we also obtain

\begin{theorem}\label{thm: hyperbolic2}
Suppose that there is a homotopy class of augmentations $[\ep]$ for $\xi$ such that
\[
0 < \dim HC_{n-3}^{[\ep]}(\xi) < \infty.
\]
If a contact form $\alpha$ for $\xi$ has finitely many geometrically distinct closed orbits then there is a non-hyperbolic one.
\end{theorem}

A homology computation in~\cite{AM} shows that there is a family of inequivalent contact structures on $S^2 \times S^3$ meeting this assumption. The isomorphism \eqref{CL} implies that the unit cotangent bundle of a closed oriented manifold with non-trivial fundamental group and compact universal covering also satisfies this condition. By the aforementioned long exact sequence, the connected sum of any manifold meeting the assumption in Theorem~\ref{thm: hyperbolic2} with any contact manifold $(N,\xi)$ such that $\dim HC_{n-3}^{[\ep]}(\xi) < \nolinebreak \infty$ has a contact structure satisfying this assumption as well.

\vskip .2cm

\noindent {\bf Organization of the paper.} Section~\ref{preliminaries} furnishes the basic material on pseudo-holomorphic curves necessary for this work. In Section~\ref{isolating_nbds_section} we define isolating neighborhoods of isolated closed Reeb orbits. This notion is crucial in the construction of local contact homology accomplished in Section~\ref{lch_section}. The computation of local contact homology is focused in Sections \ref{lch:isolated} and \ref{lch:iterated} where we deal with prime and iterated periodic orbits respectively. Morse inequalities are achieved in Section~\ref{section:mi}. They are a cornerstone in the proof of our applications presented in Section~\ref{proof_appls}. Finally, the appendices provide technical details about holomorphic curves needed for defining local contact homology and for proving Theorem~\ref{main}. Appendix~\ref{exist_orbits} handles basic compactness issues related to the degeneration of Morse-Bott data to degenerate data; this is necessary since the SFT-compactness theorem from~\cite{sftcomp} is not available in such limiting situations. 

\vskip .4cm

\noindent {\bf Note 1.} After we started to write the present paper, we were aware that Mark McLean obtained similar results using symplectic homology~\cite{McL2}. The results can be related to ours using the Bourgeois-Oancea long exact sequence~\cite{BO}. However, although our techniques lead to similar results, we think they are complementary to McLean's since they furnish an equivariant version of the local homology of closed Reeb orbits.

\vskip .2cm

\noindent {\bf Note 2.} In~\cite{fabert} Fabert introduces a local version of Symplectic Field Theory, in the same spirit as the local Gromov-Witten theory from~\cite{BP}. Thus, our use of the word ``local'' may generate a conflict in terminology, since we used it in analogy to local Floer homology.

\vskip .2cm

\acknowledgments We are grateful to Alberto Abbondandolo, Urs Frauenfelder, Viktor Ginzburg, Nancy Hingston, Michael Hutchings, Alexandru Oancea, Gabriel Paternain and Otto van Koert for useful discussions regarding this paper. Special thanks to Alberto Abbondandolo for pointing out to us, in the beginning of this work, Ginzburg-Gurel's result in \cite{GG} and its relationship with Gromoll-Meyer's theorem, to Viktor Ginzburg for pointing out a mistake in a first draft of this paper, and to Michael Hutchings for identifying an issue in our attempt to define local contact homology using domain-dependent almost complex structures in a previous version of this paper.

We thank the IAS for its hospitality and for supporting the preparation of this work.

These results were first presented by the first author at the Workshop on Conservative Dynamics and Symplectic Geometry, IMPA, Rio de Janeiro, Brazil, August 1--5, 2011. He thanks the organizers for the opportunity to participate in such a great event.

This material is based upon work supported by the National Science Foundation under agreement No. DMS-0635607. Any opinions, findings and conclusions or recommendations expressed in this material are those of the authors and do not necessarily reflect the views of the National Science Foundation.

Both authors were also supported by CNPq.

\section{Preliminaries}\label{preliminaries}

\subsection{Stable Hamiltonian structures}

We start by recalling the concept of a stable Hamiltonian structure from~\cite{sftcomp}.

\begin{definition}\label{shs_1}
A stable Hamiltonian structure on a $(2n-1)$-manifold $N$ is a triple $\H = (\xi,X,\omega)$ where $\xi \subset TN$ is a hyperplane distribution, $X$ is a vector field everywhere transverse to $\xi$ such that its flow preserves $\xi$, $\omega$ is a closed $2$-form such that $\omega|_\xi$ turns $\xi$ into a symplectic vector bundle and satisfies $i_X\omega=0$.
\end{definition}

We refer to $X$ as the Hamiltonian vector field of $\H$. Any contact form $\alpha$ on $N$ induces a stable Hamiltonian structure $(\xi,R,Cd\alpha)$, where $R$ is the associated Reeb vector field, $\xi = \ker\alpha$ is the contact structure and $C>0$.

\begin{remark}\label{shs_2}
Given such $\H = (\xi,X,\omega)$ one can define a $1$-form $\lambda$ on $N$ by
\begin{equation}\label{1form}
\begin{array}{cc} i_X\lambda=1, & \ker \lambda = \xi. \end{array}
\end{equation}
Then $\lambda \wedge \omega^n$ is a volume form and $\ker d\lambda \supseteq \ker \omega = \R X$. In particular $\mathcal L_X\lambda = 0$ and $\mathcal L_X\omega=0$. The stable Hamiltonian structure could be alternatively defined as the pair $(\lambda,\omega)$. In this case, $X$ and $\xi$ would be uniquely determined by $i_X\lambda=1$, $i_X\omega=0$ and $\xi = \ker \lambda$.
\end{remark}

Throughout Section~\ref{preliminaries} we fix a compact $(2n-1)$-manifold $N$ equipped with a stable Hamiltonian structure $\H = (\xi,X,\omega)$, and assume that $c_1(\xi)$ vanishes.

\subsubsection{Periodic orbits}\label{s:perorbs}

If $x:\R\to N$ is a periodic trajectory of $X$ with period $T>0$ then $x_T : S^1 = \R/\Z \to N$, $t \mapsto x(Tt)$, defines an element of $C^\infty(S^1,N)$. A periodic orbit $\gamma$ of $X$ is the element of $C^\infty(S^1,N)/S^1$ induced by some $x_T$ as above. We write $\gamma = (x,T)$ and $\gamma^k = (x,kT)$ for any $k\in\Z^+$. The set $x(\R) \subset N$ is called the geometric image of $\gamma$. If $T$ is the minimal positive period of $x$ then we call $\gamma$ simply covered, or prime. The set of periodic orbits of $X$ will be denoted by $\P(\H)$. When $\H$ is induced by some contact form $\alpha$ we may write $\P(\alpha)$, or simply $\P$ when the context is clear. For any given $K\subset N$ we denote by $\P(\H,K)$, or $\P(\alpha,K)$, the subset of orbits with geometric image contained in $K$.

The flow $\{\phi_t\}_{t\in\R}$ of $X$ induces a $\omega$-symplectic linear flow $d\phi_t : \xi \to \xi$. If $\gamma=(x,T) \in \P$ and $1$ is not in the spectrum of $d\phi_T : \xi_{x(0)} \to \xi_{x(0)=x(T)}$ then $\gamma$ is called nondegenerate. When every $\gamma \in \P(\H)$ ($\gamma \in \P(\H,K))$ is nondegenerate we call $\H$ nondegenerate (on $K$). The notation $\gamma = (x,T)$ is ambiguous since the choice of $x$ is not determined, so from now on we choose a special point ${\rm pt}_\gamma$ in the geometric image of every closed orbit $\gamma$ of $\{\phi_t\}$ and assume that $x(0) = \pt_\gamma$. Orbits with the same geometric image share the same special point, by convention.

\subsubsection{Conley-Zehnder indices} \label{trivializations}

Assume at first, for simplicity, that $H_1(N,\Z)$ is torsion free, and choose a set of generators $\{C_i\}$, $i=1,\dots,l$. The $C_i$ can be represented by 1-dimensional submanifolds $(\dim N\geq3)$ still denoted $C_i$, and we choose $\omega$-symplectic trivializations of $\xi|_{C_i}$. Any $\gamma = (x,T) \in \P(\H)$ can be seen as a singular 1-cycle, which induces a homology class $[\gamma] \in H_1(N,\Z)$. There are unique $n_i \in \Z$ satisfying $[\gamma] = \sum_i n_i C_i$. A 2-chain realizing a homology between $\gamma$ and $\sum_i n_iC_i$ can be used to single out a homotopy class of $\omega$-symplectic trivializations of $(x_T)^*\xi$. A trivialization in this class represents the linearized dynamics of $X$ along $\gamma$ restricted to $\xi$ as a path in $Sp(2n-2)$ starting at the identity. If $\gamma$ is nondegenerate then this path ends in the complement of the Maslov cycle, and has a well-defined Conley-Zehnder index as in~\cite{salamon}, denoted by $\cz(\gamma)\in\Z$. This is independent of the choice of the 2-chain since we assumed that $c_1(\xi)$ vanishes. The degree of $\gamma$ is defined by
\begin{equation}\label{degree}
|\gamma| := \cz(\gamma) + n-3.
\end{equation}

In the case $N$ is a unit cotangent bundle over an orientable base it is not necessary to assume $H_1(N,\Z)$ is torsion free. In fact, we can (symplectically) trivialize $\xi$ along closed Reeb orbits in such a way that the vertical Lagrangian distribution is sent to a fixed Lagrangian subspace in $(\R^{2n-2},\omega_0)$. The Conley-Zehnder index is then defined with respect to these trivializations, and the degree is as in~\eqref{degree}. The crucial properties we need are the following. Firstly, the obtained trivializations along iterated orbits are homotopic to iterated trivializations. Secondly, Fredholm indices of holomorphic curves are given by the usual formulas.


\subsubsection{Good orbits}

Let $\gamma = (x,T) \in \P(\H)$ be simply covered. According to~\cite{EGH}, if the number of eigenvalues of $d\phi_T:\xi_{x(0)}\to\xi_{x(T)=x(0)}$ in $(-1,0)$ is odd (counted with multiplicities) then the even multiples $\gamma^{2k}$ are called bad orbits. An orbit is called good if it is not bad, and we define
\begin{equation*}\label{}
\begin{aligned}
&\P_0(\H) := \{ \gamma \in \P(\H) : \gamma \text{ is good} \} \\
&\P_0(K,\H) = \P(K,\H) \cap \P_0(\H).
\end{aligned}
\end{equation*}
In the case $\H$ is induced by a contact form $\alpha$ we write $\P_0(\alpha)$ and $\P_0(K,\alpha)$ accordingly.

\subsection{Pseudo-holomorphic curves}

We take a moment to review the basic definitions from pseudo-holomorphic curve theory.

\subsubsection{Cylindrical almost complex structures}\label{alm_cpx_strs}

Let $V$ be a compact $(2n-1)$-manifold. In $\R\times V$ there is a natural $\R$-action induced by the maps
\begin{equation}\label{translation}
\tau_c : (a,p) \mapsto (a+c,p), \ c\in \R.
\end{equation}
In the language of~\cite{sftcomp}, an almost complex structure $J$ on $\R\times V$ is cylindrical if $\tau_c^*J = J, \ \forall c\in\R$, and the vector field $R := J \partial_a$ is horizontal, {\it i.e.}, it is tangent to $\{a\} \times V$, $\forall a\in\R$. Since $J$ is $\R$-invariant, the formula $\Xi := TV \cap J(TV)$ defines a $(2n-2)$-dimensional $J$-invariant distribution in $V$ and $R$, seen as a vector field on $V$, is everywhere transverse to $\Xi$. Note that $J$ is a complex structure on $\Xi$. $J$ is called symmetric if the $1$-form $\lambda$ on $V$ defined by $i_R\lambda=1$, $\Xi = \ker \lambda$ satisfies $\mathcal L_R\lambda = 0$.

Let $\Omega$ be a closed $2$-form on $V$ of maximal rank, that is, $\dim \ker \Omega = 1$. Then a cylindrical $J$ as above is said to be adjusted to $\Omega$ if the restriction $\Omega|_\Xi$ turns $\Xi$ into a symplectic vector bundle, $J|_\Xi$ is $\Omega|_\Xi$-compatible, {\it i.e.}, $\Omega(\cdot,J\cdot)$ defines a metric on $\Xi$, and $i_R\Omega=0$. Note that if $J$ is cylindrical, symmetric and adjusted to $\Omega$ then $\H = (\Xi,R,\Omega)$ is a stable Hamiltonian structure.

Conversely, a stable Hamiltonian structure $\H = (\Xi,R,\Omega)$ induces symmetric cylindrical almost complex structures on $\R\times N$ adjusted to $\Omega$. In fact, choose some $\Omega$-compatible complex structure $\widehat J$ on $\Xi$. Then we have a unique $\R$-invariant almost complex structure $J$ on $\R\times N$ defined by requiring $J \partial_a = R$ and $J|_\Xi = \widehat J$, which is the desired almost complex structure.

The set of such $J$ will be denoted by $\J(\H)$. When $\H$ is induced by a contact form $\alpha$ we may write $\J(\alpha)$ instead of $\J(\H)$.


\subsubsection{Almost complex structures in non-cylindrical cobordisms}\label{acs_cobordisms}

Consider stable Hamiltonian structures $\H^\pm$ on $N$, and $J^\pm \in \J(\H^\pm)$. For a given $L>0$ we denote by $\J_L(J^-,J^+)$ the space of almost complex structures $\bar J$ on $\R\times N$ such that $\bar J|_{[L,+\infty)} \equiv J^+$ and $\bar J|_{(-\infty,-L]} \equiv J^-$. We write $\J(J^-,J^+) = \cup_{L>0} \J_L(J^-,J^+)$. When $\H^\pm$, $J^\pm \in \J(\H^\pm)$ and $L>0$ are fixed we may consider smooth 1-parameter families $\{\bar J^\tau\}_{\tau\in[0,1]} \subset \J_{L}(J^{-},J^{+})$. We denote the spaces of such families by $\J_{\tau,L}(J^{-},J^{+})$ and $\J_\tau(J^{-},J^{+}) = \cup_{L>0} \J_{\tau,L}(J^{-},J^{+})$. 


We need to consider almost complex structures in splitting cobordisms. Fix $\H$, $J \in \J(\H)$ and numbers $0<L<R$. We denote by $\J_{L<R}(J)$ the set of almost complex structures which coincide with $J$ on $(\R\setminus [-L-R,L+R])\times N$, and are cylindrical on the neck $[L-R,R-L]\times N$. Clearly $\J_{L<R}(J) \subset \J_{L+R}(J,J)$. Also, we consider the set $\J_{\tau,L<R}(J)$ of smooth families $\{\jtil^\tau\}_{\tau\in[0,1]} \subset \J_{L<R}(J)$. Note that we {\bf do not assume} that the elements of $\J_{L<R}(J)$ are symmetric on the neck $[L-R,R-L]\times N$.


\subsubsection{Finite-energy curves}\label{fecurves}

Let $(S,j)$ be a closed Riemann surface and $Z \subset S$ be a finite set. Fix $\H = (\xi,X,\omega)$ and $J \in \J(\H)$. A smooth map $$ F = (a,f) : S \setminus Z \to \R\times N $$ is $J$-holomorphic, or pseudo-holomorphic, if it satisfies the Cauchy-Riemann equation $$\bar\partial_J (F) = \frac{1}{2} (dF + J(F) \circ dF \circ j) = 0. $$ Consider the set $\Lambda = \{\phi:\R\to[0,1] \mid \phi' \geq 0\}$. We can view an element of $\Lambda$ as a real function on $\R\times N$ depending only on the first coordinate. Similarly, we view the form $\lambda$ in~\eqref{1form} as an $\R$-invariant $1$-form on $\R\times N$. Following~\cite{sftcomp}, one defines the $\omega$-energy $E_\omega(F)$ and the energy $E(F)$ as
\begin{equation}\label{energy_cylindrical}
  E_\omega(F) = \int_{S\setminus Z} f^*\omega, \ \ \ \ \ \ E(F) = E_\omega(F) + \sup_{\phi \in \Lambda} \int_{S\setminus Z} F^*(d\phi \wedge\lambda).
\end{equation}
All these integrals have non-negative integrands. If $0< E(F)< \infty$ then $F$ is said to be a finite-energy curve. The elements of $Z$ are called punctures of $F$, and a puncture $z\in Z$ is called removable when $F$ is bounded near $z$. In this case, an application of Gromov's Removable Singularity Theorem shows that $F$ can be smoothly continued across $z$. 

Let $\H^\pm = (\xi^\pm,X^\pm,\omega^\pm)$ be stable Hamiltonian structures on $N$, fix $J^\pm \in \J(\H^\pm)$, $L>0$ and $\bar J \in \J_L(J^-,J^+)$. Assume also that there exists a symplectic form $\Omega$ on $[-L,L]\times N$ that agrees with $\omega^\pm$ on $T(\{\pm L\}\times N) \simeq TN$, up to positive constants, and tames $\bar J|_{[-L,L]\times N}$. Consider a smooth map $F : S \setminus Z \to \R\times N$ which is $\bar J$-holomorphic. Following~\cite{sftcomp} we define
\begin{equation}\label{energy_cobordism}
  \begin{aligned}
    E(F) = &\int_{F^{-1}([-L,L]\times N)} F^*\Omega + \sup_{\phi \in \Lambda} \int_{F^{-1}([L,+\infty)\times N)} F^*(d\phi\wedge\lambda^+ + \omega^+) \\
    &+ \sup_{\phi \in \Lambda} \int_{F^{-1}((-\infty,-L]\times N)} F^*(d\phi\wedge\lambda^- + \omega^-)
  \end{aligned}
\end{equation}
where $\lambda^\pm$ are the 1-forms associated to $\H^\pm$ as in~\eqref{1form}. All integrands above are non-negative. Moreover $F$ is constant if, and only if, $E(F)=0$. This definition of the energy differs slightly from that given in~\cite{sftcomp}, but yields the same finite-energy curves. $F$ is called a finite-energy curve when $0<E(F)<\infty$. As before, the points of $Z$ are called punctures, and a puncture is removable if, and only if, $F$ is bounded around it.

Finally, we need to define finite-energy $\jtil$-holomorphic curves for $\jtil \in \J_{L<R}(J)$, where $0<L<R$, $\H=(\xi,X,\omega)$ and $J \in \J(\H)$ are fixed. The correct taming conditions are as follows. On the neck  $[L-R,R-L]\times N$ we ask $\jtil$ to be adjusted\footnote{Note that $\jtil$ is cylindrical on the neck.} to some closed 2-form $\bar\omega\in\Omega^2(N)$ of maximal rank, as discussed in \S~\ref{alm_cpx_strs}, and assume also that $\jtil$ is tamed by symplectic forms $\Omega^\pm$ on $(\pm[R+L,R-L])\times N$ such that:
\begin{itemize}
\item $\Omega^+$ coincides with $\omega$ on $T(\{R+L\}\times N)$ and with $\bar\omega$ on $T(\{R-L\}\times N)$ up to positive constants,
\item $\Omega^-$ coincides with $\omega$ on $T(\{-R-L\}\times N)$ and with $\bar\omega$ on $T(\{-R+L\}\times N)$ up to positive constants.
\end{itemize}
Any $\jtil \in \J_{L<R}(J)$ induces a 1-form $\bar\lambda$ on $N$ by $i_{\bar R}\bar\lambda=1$, $\bar\xi = \ker\bar\lambda$, where $\bar\xi$ is the maximal complex subbundle of $TN$ induced by the cylindrical piece $\bar J|_{[L-R,R-L]\times N}$ and $\bar R = \bar J|_{[L-R,R-L]\times N} \cdot \partial_a$ (here $a$ is the $\R$-coordinate). Note that $(\bar\xi,\bar R,\bar\omega)$ need {\bf not} be a stable Hamiltonian structure since $\bar J$ is not necessarily symmetric on the neck $[L-R,R-L]\times N$. The energy $E(F)$ of the $\jtil$-holomorphic map $F = (a,f) : S \setminus Z \to \R\times N$ is
\begin{equation}\label{energy_split_cobordism}
\begin{aligned}
E(F) = &\sup_{\phi\in\Lambda} \int_{F^{-1}([L+R,+\infty)\times N)} F^*(d\phi\wedge \lambda +\omega) \\
&+ \sup_{\phi\in\Lambda} \int_{F^{-1}((-\infty,-L-R]\times N)} F^*(d\phi\wedge \lambda +\omega) \\
&+ \sup_{\phi\in\Lambda} \int_{F^{-1}([L-R,R-L]\times N)} F^*(d\phi\wedge \bar\lambda + \bar\omega) \\
&+ \int_{F^{-1}([-R-L,L-R]\times N)} F^*\Omega^- \\
&+ \int_{F^{-1}([R-L,R+L]\times N)} F^*\Omega^+.
\end{aligned}
\end{equation}
Punctures behave exactly as in the other cases.

%

\subsubsection{Asymptotic behavior}\label{asymp_behavior}

Let $F : (S\setminus Z,j) \to \R\times N$ be a finite-energy $J$-holomorphic curve, where $J \in \J(\H)$. The behavior of $F$ near a non-removable puncture $z\in Z$ is studied in~\cite{props1}, see also the Appendix of~\cite{sftcomp}.  

\begin{proposition}\label{partial_asymptotic_behavior}
Let $\psi : (B_1(0),0) \to (V,z)$ be a holomorphic chart of $(S,j)$ and write $F(s,t) = (a(s,t),f(s,t)) = F\circ \psi(e^{-2\pi(s+it)})$ for $(s,t) \in \R^+\times \R/\Z$. When $X$ is nondegenerate one finds $\gamma = (x,T) \in \P(\H)$, $\epsilon = \pm1$, $c,d \in \R$ such that the loops $t\mapsto f(s,t)$ converge to $t\mapsto x(\epsilon Tt+c)$ in $C^\infty(S^1,N)$, and all partial derivatives of $a(s,t) - \epsilon Ts - d$ tend to $0$ uniformly in $t$ as $s\to+\infty$.
\end{proposition}

A non-removable puncture $z$ is positive if $\epsilon = +1$, and negative if $\epsilon = -1$. In any case one says that $F$ is asymptotic to $\gamma$ at $z$, and $\gamma$ is the asymptotic limit of $F$ at $z$. These definitions are independent of the choice of $\psi$. If $S=S^2$ and $\# Z=2$ then $(S\setminus Z,j)\simeq (\R\times\R/\Z,i)$, and we loosely refer to the ``ends'' $\{\pm\infty\}\times S^1$ as punctures.

Under the assumption that the Hamiltonian vector fields are nondegenerate, the asymptotic behavior of finite-energy curves in non-cylindrical cobordisms is analogous and we will not describe it here. The reader can easily guess the precise statements.


\section{Isolating neighborhoods of isolated orbits}\label{isolating_nbds_section}

In this section we discuss the necessary geometric set-up for defining local contact homology of an isolated orbit. Let $\H = (\xi,X,\omega)$ be a stable Hamiltonian structure on the $(2n-1)$-manifold $N$, and $x:\R \to N$ be a $T$-periodic trajectory of the vector field $X$, where $T>0$. Assume that $\gamma=(x,T)$ is isolated, {\it i.e.}, $t\mapsto x(Tt)$ defines an isolated point of the set of closed $X$-orbits in $C^\infty(S^1,N)/S^1$.

\begin{defn}\label{isol_nbd_def}
An isolating neighborhood for $\gamma$ is a compact connected neighborhood $K$ of $x(\R)$ with smooth boundary satisfying:
\begin{itemize}
\item $\gamma$ is the only closed $X$-orbit in $K$ in the free homotopy class of $\gamma$ (of loops in $K$),
\item $H^2(K,\R) = 0$, $H^1(K,\R) = \R$ and there is a nontrivial class $[\theta] \in H^1(K,\R)$ represented by a closed 1-form $\theta$ satisfying $\inf_K i_X\theta > 0$.
\end{itemize}
\end{defn}

It follows that $K$ contains no closed $X$-orbits which are contractible in $K$.

\begin{lemma}
Every isolated $\gamma$ has an isolating neighborhood.
\end{lemma}

\begin{proof}
Let $T_0$ be the minimal positive period of $x$. Take a neighborhood $K \simeq S^1 \times \overline B$ of $x(\R)$, where $B\subset \R^{2n-2}$ is an open ball around the origin, equipped with coordinates $(t,z)$, such that $x(t) \simeq (t/T_0,0)$ and $\xi_{x(t)} \simeq 0\times \R^{2n-2}$. Thus $\lambda|_{S^1\times 0} = T_0dt$ where $\lambda$ is the $1$-form~\eqref{1form}. After shrinking $B$, $X$ becomes transverse to $\{t\}\times \overline B \ \forall t$. Hence $dt$ is a closed 1-form generating $H^1(K,\R)\simeq \R$ such that $i_X dt > 0$. Suppose $\gamma_k = (x_k,T_k) \neq \gamma$ are closed $X$-orbits homotopic to $\gamma$ in $K$ such that $\limsup_{k\to \infty} \sup_{s \in \R} |\pi_{\R^{2n-2}} \circ x_k(s)| = 0$, where $\pi_{\R^{2n-2}}$ denotes projection onto the second coordinate. We can assume $x_k(0) \in \{0\}\times \overline B$, so that $x_k \to x$ in $C^\infty_{\rm loc}$. If $T_k \to \infty$ then $\int_\gamma dt = \int_{\gamma_k}dt \to +\infty$, an absurd. Thus we get a bound for the periods $T_k = \int_{\gamma_k} \lambda$ so that $\gamma_k \to \gamma$, contradicting the hypothesis that $\gamma$ is isolated. This contradiction shows that, possibly after further shrinking $\overline B$, we get an isolating neighborhood $K$ as in Definition~\ref{isol_nbd_def} since $H^2(K,\R) = 0$.
\end{proof}

\begin{lemma}\label{lemma_nearby_orbits}
Let $K$ be an isolating neighborhood for the isolated orbit $\gamma$. For every neighborhood $\V$ of $\gamma$ in $C^\infty(S^1,N)/S^1$ there exists a $C^\infty$-neighborhood $\mathcal O$ of $\H$ in the space of stable Hamiltonian structures such that the following holds: if $\H'= (\xi',X',\omega') \in \mathcal O$ then all closed $X'$-orbits in $K$ homotopic to $\gamma$ through loops in $K$ lie in $\V$.
\end{lemma}

The statement above is easy and left with no proof.

\subsection{Finite-energy cylinders in isolating neighborhoods}

\begin{lemma}\label{lemma_nearby_cylinders}
Let $\H=(\xi,X,\omega)$, $J \in \J(\H)$, $L>0$ and an isolated closed $X$-orbit $\gamma$ be given. Let also $K$ be an isolating neighborhood of $\gamma$, and $\Omega$ be a symplectic form on $[-L,L]\times K$ which tames $J$ and agrees with $\omega$ on $T(\{\pm L\}\times K)$ up to positive constants. Then there are $C^\infty$-neighborhoods $\mathcal O$ of $\H$ and $\mathcal D$ of $\Omega$, a neighborhood $\U$ of $J$ in the $C^\infty$-strong topology, and a constant $C>0$ such that the following holds. If 
\begin{itemize}
\item $\H^\pm = (\xi^\pm,X^\pm,\omega^\pm) \in \mathcal O$ and all closed $X^\pm$-orbits in $K$ homotopic to $\gamma$ in $K$ are nondegenerate,
\item $J^\pm \in \J(\H^\pm) \cap \U$, $\jtil \in \J_L(J^-,J^+) \cap \U$, 
\item $\widetilde \Omega \in \mathcal D$ is a symplectic form coinciding with $\omega^\pm$ on $T(\{\pm L\}\times K)$ up to positive constants
\end{itemize}
then $\widetilde\Omega$ tames $\jtil$ on $[-L,L]\times K$, and every finite-energy $\jtil$-holomorphic map $F = (a,f) : \R \times S^1 \to \R\times N$ satisfying
\begin{itemize}
\item[i)] $f(\R\times S^1) \subset K$ and the loops $t \mapsto f(s,t)$ are homotopic to $\gamma$ in $K$;
\item[ii)] $\{+\infty\} \times S^1$ is a positive puncture, $\{-\infty\} \times S^1$ is a negative puncture;
\end{itemize}
must also satisfy $E(F)\leq C$ and $\overline {f(\R\times S^1)} \subset {\rm int}(K)$. Here we used $\widetilde \Omega$ to define $E(F)$ according to the discussion in \S~\ref{fecurves}.
\end{lemma}

The proof below makes use of the results from Appendix~\ref{exist_orbits}.

\begin{proof}
First we address the issue of obtaining the desired energy bounds. We claim the existence of $C>0$ with the following property. Consider arbitrary sequences
\begin{itemize}
\item $\H_n^\pm = (\xi^\pm_n,X^\pm_n,\omega^\pm_n) \to \H$ in $C^\infty$,
\item $J_n^\pm \in \J(\H_n^\pm)$, $\jtil_n \in \J_L(J_n^-,J_n^+)$ such that $J_n^+,J_n^-,\jtil_n \to J \ \text{ in } \ C^\infty$-strong,
\item $\Omega_n$ symplectic forms on $[-L,L]\times K$ coinciding with $\omega^\pm_n$ on $T(\{\pm L\}\times K)$ up to positive constants such that $\Omega_n \to \Omega$ in $C^\infty$,
\item $F_n = (a_n,f_n)$ finite-energy $\jtil_n$-holomorphic cylinders satisfying i) and~ii).
\end{itemize}
Then $E(F_n)\leq C$ if $n$ is large enough.

Obviously, $\Omega_n$ tames $\jtil_n$ on $[-L,L]\times K$ if $n$ is large enough since $\Omega$ tames $J$. By condition ii) the $F_n$ are non-constant maps.

To prove the existence of $C$ we proceed indirectly and assume, by contradiction, the existence of $\H_n^\pm$, $J_n^\pm$, $\jtil_n$, $\Omega_n$ and $F_n$ meeting the above requirements and
\begin{equation}\label{energy_contradiction}
E(F_n)\to+\infty.
\end{equation}

Since $H_2(K,\R)$ vanishes, there exists a primitive $\alpha$ for $\Omega$ on $[-L,L]\times K$ and, using the Mayer-Vietoris principle, we find primitives $\alpha_n$ of $\Omega_n$ on $[-L,L]\times K$ such that $\alpha_n \to \alpha$ in $C^\infty$. Consider inclusions $i_a:K \simeq \{a\}\times K \hookrightarrow \R\times K$ and 1-forms defined by $\alpha^\pm_n = (i_{\pm L})^*\alpha_n, \ \alpha^\pm = (i_{\pm L})^*\alpha$. Then, by the properties of $\Omega_n,\Omega,\alpha_n,\alpha$ we have $\alpha^\pm_n \to \alpha^\pm$, $d\alpha^\pm_n = c_n^\pm\omega^\pm_n$, $d\alpha^\pm = c^\pm\omega$, where $c^\pm_n,c^\pm>0$ satisfy $c^\pm_n\to c^\pm$. Let $\theta$ be the closed 1-form as in Definition~\ref{isol_nbd_def}. By adding $A\theta$ to $\alpha_n$ and $\alpha$, with $A\gg 1$, we may assume, without loss of generality, that $\liminf_n \inf_K \alpha^\pm_n(X^\pm_n) >0$.

For each $n$ let $x_n^\pm$ be regular values of $a_n(s,t)$ satisfying $-L<x_n^- < x_n^+ < L$ and $x_n^\pm \to \pm L$ as $n\to\infty$. We compute
\[
\begin{aligned}
\mathcal I_n := & \int_{a_n\geq L} f_n^*\omega^+_n + \int_{-L\leq a_n\leq L} F_n^*\Omega_n + \int_{a_n\leq -L} f_n^*\omega^-_n \\
&\leq c' \left( \int_{a_n\geq L}  f_n^*(c_n^+\omega^+_n) + \int_{-L\leq a_n\leq L} F_n^*\Omega_n + \int_{a_n\leq -L} f_n^*(c_n^-\omega^-_n) \right) \\
&= c' \left( \int_{a_n \geq x_n^+} f_n^*(c_n^+\omega^+_n) + \int_{x_n^-\leq a_n \leq x_n^+} F_n^*\Omega_n + \int_{a_n\leq x_n^-} f_n^*(c_n^-\omega^-_n) \right) \\
&+ c' \left( \int_{x_n^+\leq a_n\leq L} F_n^*\Omega_n - f_n^*(c_n^+\omega^+_n) + \int_{-L\leq a_n \leq x_n^-} F_n^*\Omega_n -  f_n^*(c_n^-\omega^-_n) \right) \\
& = (T'_n) + (T''_n).
\end{aligned}
\]
Here $c'$ depends only on $c^\pm$, which in turn depend only on $\Omega$.

We extend $\Omega,\Omega_n$ smoothly and arbitrarily to $\R\times K$, and let $h_n^\pm:\R\times S^1 \to \R$ be defined by $F_n^*\Omega_n - c_n^\pm f_n^*\omega_n^\pm = h^\pm_n(s,t)ds\wedge dt$. Consider the compact measurable sets $E^\pm_n = \{(s,t)\mid a_n(s,t)=\pm L\}$. For this discussion $n$ is fixed. Since $\Omega_n - c_n^\pm\omega_n^\pm = 0$ on $T(\{\pm L\}\times N)$, every point in $A^\pm_n := \{(s,t)\in E^\pm_n \mid h^\pm_n(s,t)\neq 0\}$ is a point where $F_n$ hits $\{\pm L\}\times N$ transversely, in particular, the (Lebesgue) measure of $A^\pm_n$ is zero. Setting $B^\pm_n := \{(s,t)\in E^\pm_n \mid h^\pm_n(s,t)=0\}$ then $h^\pm_n$ integrates to zero (with respect to Lebesgue) in both $A^\pm_n$ and $B^\pm_n$. Hence
\[
\int_{\{a_n = \pm L\}} F_n^*\Omega_n - c_n^\pm f_n^*\omega_n^\pm = 0. 
\]
Consequently, by the dominated convergence theorem, we could have chosen $x_n^\pm$ close enough to $\pm L$ so that, say,
\[
(T''_n) \leq 1.
\]

Consider diffeomorphisms $\Psi_n$ of $\R\times K$ satisfying $\Psi_n\to id$ in $C^\infty$-strong, $\Psi_n([x_n^-,x_n^+]\times K) = [-L,L]\times K$ and $\Psi_n(x_n^\pm,p) = (\pm L,p) \ \forall p\in K$. Then $\Psi_n^*\Omega_n \to \Omega$ in $C^\infty$ on $[-L,L]\times K$, in particular, $\Psi_n^*\Omega_n$ tames $\jtil_n$ on $[-L,L]\times K$ when $n\gg1$. Note that $\Psi_n^*\alpha_n$ is a primitive of $\Psi_n^*\Omega_n$ on $[-L,L]\times K$, and $(i_{x_n^\pm})^*\Psi_n^*\alpha_n = (i_{\pm L})^*\alpha_n = \alpha_n^\pm$ since $\Psi_n\circ i_{x_n^\pm}=i_{\pm L}$. With Stokes theorem we can estimate
\[
\begin{aligned}
(T_1) &\leq 2c' \left( \int_{a_n \geq x_n^+} f_n^*(c_n^+\omega^+_n) + \int_{x_n^-\leq a_n \leq x_n^+} F_n^*\Psi_n^*\Omega_n + \int_{a_n\leq x_n^-} f_n^*(c_n^-\omega^-_n) \right) \\
&= 2c' \left( \int_{\gamma_n^+} \alpha_n^+ - \int_{\gamma_n^-} \alpha_n^- \right) \leq c''
\end{aligned}
\]
for some constant $c''>0$ independent of $n$. Here $\gamma^\pm_n$ are the asymptotic limits of $F_n$ at $\{\pm\infty\}\times S^1$ oriented by the Hamiltonian vector fields. By Lemma~\ref{lemma_nearby_orbits}, $\gamma^\pm_n \to \gamma$, so that we get the uniform bound for the last line. Moreover, we assume that $x_n^\pm \to \pm L$ and $\Psi_n\to id$ fast enough.

Thus we bounded $\mathcal I_n$ by some constant independent of $n$. The other terms of the energy defined in \S~\ref{fecurves} are estimated analogously. This is in contradiction to~\eqref{energy_contradiction}.

Let again $\H_n^\pm$, $J_n^\pm$, $\jtil_n$, $\Omega_n$ and $F_n$ be sequences as above. We have already proved that $\sup_nE(F_n)<\infty$. We proceed indirectly and suppose, by contradiction, that $\overline{f_n(\R\times S^1)} \cap \partial K \neq \emptyset$ for all $n$. By Lemma~\ref{lemma_nearby_orbits}, $\gamma^\pm_n \subset{\rm int}(K)$ and, consequently, we find $(s_n,t_n)$ such that $f_n(s_n,t_n) \in \partial K$. We claim that $dF_n$ is $C^0_{\rm loc}$-bounded. If not we use Lemma~\ref{bubb} to get a non-constant finite-energy $J$-holomorphic plane with image in $\R\times K$. Lemma~\ref{hofer} applied to the end of this plane gives us a periodic orbit of $X$ contractible inside $K$, a contradiction. Let $d_n = a_n(s_n,t_n)$ and define 
\[
\util_n(s,t) = (a_n(s+s_n,t+t_n)-d_n,f_n(s+s_n,t+t_n)).
\]
By the above discussion we have $C^1_{\rm loc}$-bounds for $\util_n$ and, consequently, also $C^\infty_{\rm loc}$-bounds by elliptic boot-strapping arguments. Up to a subsequence, we can assume there exists a finite-energy $J$-holomorphic cylinder $\util = (b,u)$ such that $\util_n \to \util$ in $C^\infty_{\rm loc}$. The map $\util$ is non-constant since for any given $s\in\R$ the loop $t\mapsto u(s,t)$ is homotopic to $\gamma$, which is not contractible in $K$ by Definition~\ref{isol_nbd_def}. A quick look at the proof of Lemma~\ref{hofer} will show that
\[
\lim_{s\to-\infty} \int_{S^1} u(s,\cdot)^*\alpha > 0 
\]
since, otherwise, $u(s,t)$ would converge to a point in $N$ as $s\to-\infty$, a contradiction since $\gamma$ is not contractible in $K$. Thus an application of Lemma~\ref{hofer} gives sequences $s_n^\pm \to \pm\infty$ such that $u(s_n^\pm,t) \to \gamma(t+t_0^\pm)$ when $n\to\infty$. Hence
\[
0 \leq \int_C u^*\omega  \leq \lim_j \int_{[s^-_j,s^+_j]\times S^1} u^*\omega = \int_\gamma \eta - \int_\gamma \eta = 0
\]
where $\eta$ is any primitive of $\omega$ on $K$. Thus $E_\omega(\util)=0$ and we can apply Lemma~\ref{cylinder} to $\util$ and conclude that $X$ has a periodic orbit in $K$, homotopic to $\gamma$ in $K$, touching $\partial K$ ($\util(0,0) \in \R\times \partial K$). This contradiction completes the proof that $\overline{f_n(\R\times S^1)} \subset {\rm int}(K)$ when $n$ is large enough.
\end{proof}

We need a version of the above statement for almost complex structures as we stretch the neck. The proof is similar and omitted.

\begin{lemma}\label{lemma_nearby_cyls_neck}
Suppose we have a stable Hamiltonian structure $\H = (\xi,X,\omega)$ on $N$, and let $\gamma$ be an isolated closed $X$-orbit. Consider an isolating neighborhood $K$ for $\gamma$, $J \in \J(\H)$, $L>0$, and symplectic form $\Omega$ on $[-L,L]\times K$ taming $J$, which coincides with $\omega$ on $T(\{\pm L\}\times K)$, up to positive constants. Then there are $C^\infty$-neighborhoods $\mathcal O$ of $\H$, $\mathcal D$ of $\Omega$, $\mathcal W$ of $\omega$, a $C^\infty$-strong neighborhood $\mathcal U$ of $J$ and a constant $C>0$ such that if 
\begin{itemize}
\item $\H' = (\xi',X',\omega')\in\mathcal O$ and all closed $X'$-orbits in $K$ homotopic to $\gamma$ in $K$ are nondegenerate,
\item $\bar\omega \in \mathcal W$, 
\item $R>L$, $J' \in\J(\H')\cap\mathcal U$, $\jtil \in \J_{L<R}(J')\cap \mathcal U$ 
\item $\Omega'_\pm \in \mathcal D$ are symplectic forms on $[-L,L]\times N$ such that $\Omega'_+$ coincides with $\omega'$ on $T(\{L\}\times N)$ and with $\bar\omega$ on $T(\{-L\}\times N)$, up to positive constants; $\Omega'_-$ coincides with $\bar\omega$ on $T(\{L\}\times N)$ and with $\omega'$ on $T(\{-L\}\times N)$, up to positive constants,
\item $\jtil$ is adjusted to $\bar\omega$ on the neck $[L-R,R-L]\times K$ (see \S~\ref{acs_cobordisms}), 
\end{itemize}
then $(\tau_{\mp R})^*\Omega'_\pm$ tames $\jtil$ on $(\pm[R-L,R+L])\times K$, and every $\jtil$-holomorphic finite-energy map $F = (a,f) : \R \times S^1 \to \R\times N$ satisfying
\begin{enumerate}
\item[i)] $f(\R\times S^1) \subset K$ and the loops $t \mapsto f(s,t)$ are homotopic to $\gamma$ in $K$;
\item[ii)] $\{+\infty\} \times S^1$ is a positive puncture; $\{-\infty\} \times S^1$ is a negative puncture;
\end{enumerate}
must also satisfy $E(F)\leq C$ and $\overline{f(\R\times S^1)} \subset {\rm int}(K)$. Above we used the symplectic forms $(\tau_{\mp R})^*\Omega'_\pm$ on $(\pm[R-L,R+L])\times K$ and $\bar\omega$ on the neck $[L-R,R-L]\times K$ to define $E(F)$ as described in \S~\ref{fecurves}.
\end{lemma}

\subsection{Special stable Hamiltonian structures}\label{sp_strs_section}

\begin{defn}\label{special_strs}
We call $\H = (\xi,X,\omega)$ special near $\gamma$ if there exists a small closed smooth tubular neighborhood $K$ of $\gamma$ such that one of the following mutually excluding conditions holds:
\begin{itemize}
\item[i)] $\H$ is induced by some contact form, {\it i.e.}, $\xi = \ker \alpha$, $X = X_\alpha$ and $\omega = Cd\alpha$ where $\alpha$ is a contact form defined near $K$ and $C>0$ is a constant.
\item[ii)] The 1-form $\lambda$ given as in~\eqref{1form} is closed on $K$.
\end{itemize}
If some $K$ is given for which i) or ii) applies then we say $\H$ is special for $\gamma$ and $K$. 
\end{defn}

In order to define local contact homology of the pair $(\H,\gamma)$ one needs to slightly perturb $\H$ near $K$, within the class of stable hamiltonian structures, to make closed orbits inside $K$ and homotopic to $\gamma$ (in $K$) nondegenerate. This may not be possible in general. However, when $\H$ is special near $\gamma$ this perturbation can always be performed. This is well-known in case i). Assume that $\H$ falls into case ii). Let $\alpha$ be a primitive for $\omega$ on $K$, which exists since $H^2(K,\R)$ vanishes when $K$ is small enough. Possibly after replacing $\alpha$ by $C\lambda + \alpha$, with $C\gg 1$, we can assume $\inf_K i_X\alpha > 0$ and $\alpha$ is a contact form on $K$. The Reeb vector field $X_\alpha$ is a pointwise positive multiple of $X$. Let $\alpha'$ be a $C^\infty$-small perturbation of $\alpha$ near $K$ so that all closed $\alpha'$-Reeb orbits inside $K$ homotopic to $\gamma$ (in $K$) are nondegenerate. Then $\omega'=d\alpha'$ is a small perturbation of $\omega$ and $\H' = (\xi = \ker \lambda,X',\omega' = d\alpha')$ is a stable Hamiltonian structure $C^\infty$-close to $\H$, where the vector field $X'$ is given by $i_{X'}\omega'=0$ and $i_{X'}\lambda=1$ (recall that $d\lambda=0$ by assumption). Moreover, since $\R X' = \R X_{\alpha'}$, closed $X'$-orbits inside $K$ which are homotopic to $\gamma$ (in $K$) are nondegenerate.

\begin{remark}
As an example, consider $K = S^1\times \overline B$ and a smooth Hamiltonian $H:K \to \R$ satisfying $dH_t(0)=0, \forall t$. Assume $0$ is an isolated 1-periodic orbit of the Hamiltonian vector field $X_{H_t}$ characterized by $dH_t = i_{X_{H_t}}\omega_0$. The typical example of special stable hamiltonian structure satisfying ii) in Defintion~\ref{special_strs} is $\H = (\ker dt,\widetilde X_H,\omega_H)$, where $\widetilde X_H = \partial_t + X_{H_t}$ and $\omega_H = dH_t \wedge dt + \omega_0$. Then $x_0(t) = (t,0)$ is an isolated 1-periodic orbit of $\widetilde X_H$. By the above discussion we can always perturb the Hamiltonian $H$ to obtain a nondegenerate perturbation of $\H$.
\end{remark}

\section{Local contact homology}\label{lch_section}

Contact Homology was originally introduced in~\cite{EGH} inside the bigger framework of Symplectic Field Theory. Following Floer~\cite{Fl1}, we define a suitable version of what we call the local contact homology of an isolated orbit, see~Definition~\ref{defn_lch}.

\subsection{Defining local contact homology}\label{local_cch}

\subsubsection{Local chain complexes}\label{local_complex}

Throughout Section~\ref{lch_section} we fix $\H = (\xi,X,\omega)$, an isolated closed $X$-orbit $\gamma=(x,T)$, and assume $\H$ is special for $\gamma$ as in Definition~\ref{special_strs}. Since $\gamma$ is isolated, we will fix an isolating neighborhood $K$ of $\gamma$. Moreover, as explained in \S~\ref{sp_strs_section}, perhaps after shrinking $K$, we can always perturb $\H$ to an arbitrarily $C^\infty$-close $\H'=(\xi',X',\omega')$ so that all closed $X'$-orbits in $K$ homotopic to $\gamma$ (in $K$) are nondegenerate. We refer to $\H'$ as a {\bf nondegenerate perturbation of $\H$}. This notion depends on $K$ and $\gamma$.

We denote by $\P(\H',K,\gamma)$ and $\P_0(\H',K,\gamma)$ the sets of closed $X'$-orbits and good closed $X'$-orbits in $K$ which are homotopic to $\gamma$ in $K$, respectively. By Lemma~\ref{lemma_nearby_orbits} every $\gamma' \in \P(\H',K,\gamma)$ can be assumed arbitrarily $C^\infty$-close to $\gamma$ if $\H'$ is close enough to $\H$, in particular, they all lie in the interior of $K$. Also, $\P(\H',K,\gamma)$ is finite if $\H'$ is close enough to $\H$ since there are automatic period bounds for these orbits.

We fix a homotopy class of $\omega$-symplectic trivializations of $(x_T)^*\xi \to \R/\Z$, where $x_T: \R/\Z \to N$ is the map $t \mapsto x(Tt)$. It distinguishes homotopy classes of $\omega'$-symplectic trivializations of $\xi'$ along every $\gamma' \in \P(\H',K,\gamma)$, which are used to compute Conley-Zehnder indices $\cz(\gamma')$. Let $C_*(\H',K,\gamma)$ be the vector space over $\Q$ freely generated by $\mathcal P_0(\H',K,\gamma)$ and graded by $|\gamma'| = \cz(\gamma') + n-3$.

We choose $J \in \J(\H)$ and assume we can find $J' \in \J(\H')$ arbitrarily close to~$J$, $C^\infty$-strong or equivalently $C^\infty$-weak, which is regular for the data $(\H',K,\gamma)$ in the following sense. Consider the collection $\mathcal Z(J',K,\gamma)$ of $J'$-holomorphic maps $F = (a,f) : \R\times S^1 \to \R\times N$ with finite energy, satisfying $a(s,t) \to \pm\infty$ as $s\to\pm\infty$, with image in $\R\times K$, and asymptotic to orbits in $\P(\H',K,\gamma)$. 
Then we call $J'$ regular for the data $(\H',K,\gamma)$ if the linearized Cauchy-Riemann equation at every $F\in \mathcal Z(J',K,\gamma)$ determines a surjective Fredholm operator in a standard functional analytical set-up. 

\begin{remark}
The existence of $J'$ which are regular for the data $(\H',K,\gamma)$ is not guaranteed and, consequently, becomes part of our hypotheses. 
\end{remark}

The definition of a differential on $C_*(\H',K,\gamma)$ follows a standard construction. Let $\gamma_+,\gamma_- \in \P(\H',K,\gamma)$ and consider the collection $\F_{K,J'}(\gamma_+;\gamma_-)$ of triples $$ (t_+,t_-,F) \in S^1\times S^1 \times \mathcal Z(J',K,\gamma) $$ such that $F$ is asymptotic to $\gamma_\pm$ at $s\to\pm\infty$, respectively, and if we write $F=(a,f)$ then $f(s,t_\pm) \to \pt_{\gamma_\pm}$ as $s\to\pm\infty$.

By our regularity assumption $\F_{K,J'}(\gamma_+;\gamma_-)$ is a smooth manifold. It is nontrivial, but standard, that the topology that $\F_{K,J'}(\gamma_+;\gamma_-)$ inherits as a subset of a certain Banach manifold of maps coincides with the $C^\infty_{\rm loc}$-topology. 
The reparametrization group $\R\times S^1$ acts on $\F_{K,J'}(\gamma_+;\gamma_-)$ by
\begin{equation}\label{rep_action}
(s_0,t_0) * (t_+,t_-,F) = (t_+-t_0,t_--t_0,F(s+s_0,t+t_0)).
\end{equation}
The moduli space of interest for us is
\begin{equation}\label{moduli_space_d}
\M_{K,J'}(\gamma_+;\gamma_-) := \F_{K,J'}(\gamma_+;\gamma_-) / \R\times S^1.
\end{equation}
Since the $\R\times S^1$-action is free, this moduli space is a smooth manifold. 
Its dimension is $|\gamma_+|-|\gamma_-|$. When $\gamma_+\neq\gamma_-$ the space~\eqref{moduli_space_d} is equipped with the free $\R$-action induced by~\eqref{translation}, hence in this case we get
\begin{equation}\label{est_dim_transv}
\M_{K,J'}(\gamma_+;\gamma_-) \neq \emptyset \Rightarrow |\gamma_+|-|\gamma_-|\geq 1.
\end{equation}

\begin{lemma}\label{lemma_finite_d}
If $|\gamma_+|-|\gamma_-|=1$ then $\M_{K,J'}(\gamma_+;\gamma_-)/\R$ is finite.
\end{lemma}

\begin{proof}
Consider a sequence $C_n \in \M_{K,J'}(\gamma_+;\gamma_-)$ representing distinct elements in $\M_{K,J'}(\gamma_+;\gamma_-)/\R$. Energy bounds are automatically guaranteed by Lemma~\ref{lemma_nearby_cylinders} since $(\H',J')$ is a small perturbation of $(\H,J)$. The limiting behavior of the sequence $C_n$ is described by the SFT-Compactness Theorem, in fact, $C_n$ converges in the SFT-sense to a so-called holomorphic building. Every curve in this building lies in $\R\times K$. No curve in this building is a holomorphic sphere since $J'$ is tamed by exact symplectic forms. Moreover, no curve in this building is a plane since there are no contractible periodic orbits of the vector field $X$ inside $K$. Hence every level consists of a single $J'$-holomorphic cylinder. Now, using the assumed transversality and~\eqref{est_dim_transv}, a simple calculation shows that the building consists of a single cylinder which is an element of $\M_{K,J'}(\gamma_+;\gamma_-)/\R$. Consequently this moduli space is $0$-dimensional and compact, hence finite.
\end{proof}

Under the above assumptions on $(\H',J')$ we follow~\cite{be} and consider a system of coherent orientations on the various spaces $\M_{K,J'}(\gamma_+;\gamma_-)$.
These orientations are well-defined even when $\gamma_+$ or $\gamma_-$ is a bad orbit. One associates to every $C\in \M_{K,J'}(\gamma_+;\gamma_-)$ a sign $\epsilon(C) = \pm1$ by comparing the coherent orientations with the orientation given by the infinitesimal $\R$-action. Following~\cite{EGH}, we set
\begin{equation}\label{signed_sum}
n(\gamma_+,\gamma_-) = \sum_{C \in \\ \M_{K,J'}(\gamma_+;\gamma_-)/\R} \epsilon(C)
\end{equation}
for every pair $\gamma_+,\gamma_- \in \P(\H,'K,\gamma)$ of orbits satisfying $|\gamma_+|-|\gamma_-|=1$. Here we still wrote $C$ to indicate a class in $\M_{K,J'}(\gamma_+;\gamma_-)/\R$. By Lemma~\ref{lemma_finite_d} the above sum is finite. Set $n(\gamma_+,\gamma_-)=0$ if $|\gamma_+|-|\gamma_-|\neq1$.

There are operators $\rho_\pm$ on $\M_{K,J'}(\gamma_+;\gamma_-)$ given by rotating asymptotic markers:
\begin{equation}
\begin{aligned} 
&\rho_+([t_+,t_-,F]) = [t_++1/m_+,t_-,F] \\
&\rho_-([t_+,t_-,F]) = [t_+,t_-+1/m_-,F]
\end{aligned}
\end{equation}
where $m_\pm$ are the multiplicities of $\gamma_\pm$. Assume that $\gamma_+,\gamma_- \in \P(\H',K,\gamma)$ satisfy $|\gamma_+|-|\gamma_-|=1$. Then
\begin{equation}
\begin{aligned}
& \gamma_+ \text{ is good} \Rightarrow \epsilon(\rho_+(C))=\epsilon(C) \\
& \gamma_- \text{ is bad} \Rightarrow \epsilon(\rho_+(C))=-\epsilon(C).
\end{aligned}
\end{equation}
For a proof see~\cite{be}. In particular, $n(\gamma_+,\gamma_-)=0$ when $\gamma_+$ or $\gamma_-$ is bad; note here that a bad orbit has even multiplicity.
Finally we define
\begin{equation}\label{local_diff}
d : C_*(\H',K,\gamma) \to C_{*-1}(\H',K,\gamma) \ \ \text{ by } \ \ d\gamma' = \sum_{\gamma'' \in \P_0(\H',K,\gamma)} \frac{n(\gamma',\gamma'')}{m_{\gamma''}} \gamma'' 
\end{equation}
on generators $\gamma' \in \P_0(\H',K,\gamma)$.


\begin{lemma}\label{d2=0}
If $\H'$ is a sufficiently $C^\infty$-small nondegenerate perturbation of $\H$, $J' \in \J(\H')$ is sufficiently $C^\infty$-close to $J$ and regular for the data $(\H',K,\gamma)$, then the square of the map~\eqref{local_diff} vanishes.
\end{lemma}


\begin{proof}[Proof of Lemma~\ref{d2=0}]
Assuming that $\gamma',\gamma''' \in \P_0(\H',K,\gamma)$ satisfy $|\gamma'|=|\gamma'''|+2$, we wish to show that $m_{\gamma'''}\left<d^2(\gamma'),\gamma'''\right>$ is the algebraic count of ends of the manifold $\M_{K,J'}(\gamma';\gamma''')/\R$ which is 1-dimensional. Already here we need to note that all cylinders representing elements of this space project compactly inside ${\rm int}(K)$. This is so in view of Lemma~\ref{lemma_nearby_cylinders} since we take $(\H',J')$ is a very small perturbation of $(\H,J)$. Hence $\M_{K,J'}(\gamma';\gamma''')/\R$ is indeed a manifold without a genuine boundary (boundaries only exist in the sense of SFT). Let us denote $\ell = |\gamma'|-1$.

As explained in the proof of Lemma~\ref{lemma_finite_d}, any sequence $C_n \in \M_{K,J'}(\gamma';\gamma''')/\R$ can only SFT-converge to a building having precisely one cylinder in each of its levels. In fact, all curves in such a limiting building lie in $\R\times K$ since $K$ is compact, moreover, no spheres appear since $J'$ is tamed by some exact symplectic form on $\R\times K$, and no planes appear since $X'$ has no contractible closed trajectories inside $K$. Thus the building is cylindrical. The regularity assumption now allows for two possibilities: either the building consists of a single cylinder $C\in \M_{K,J'}(\gamma';\gamma''')/\R$ and $C_n\to C$ in $\M_{K,J'}(\gamma';\gamma''')/\R$, or it has two levels, the top cylinder represents elements in $\M_{K,J'}(\gamma';\gamma'')$ and the bottom cylinder represents elements in $\M_{K,J'}(\gamma'';\gamma''')$ for some orbit $\gamma''\in\P(\H',K,\gamma)$ satisfying $|\gamma''|=\ell$. Here $\gamma''$ does not need to be good. Note the freedom in the positive asymptotic marker at the bottom cylinder, and in the negative asymptotic marker at the top cylinder.

Conversely, with an orbit $\gamma''\in\P(\H',K,\gamma)$ satisfying $|\gamma''|=\ell$ fixed, cylinders $C_+ \in \M_{K,J'}(\gamma';\gamma'')/\R$, $C_- \in \M_{K,J'}(\gamma'';\gamma''')/\R$ project compactly inside ${\rm int}(K)$ again by Lemma~\ref{lemma_nearby_cylinders} since $(\H',J')\sim(\H,J)$. Hence, using the assumed regularity, they can be glued to a $1$-parameter family of cylinders in $\M_{K,J'}(\gamma';\gamma''')/\R$ which will break to the two-level cylindrical building composed of $C_+$ and $C_-$.

We conclude that the SFT-boundary of $\M_{K,J'}(\gamma';\gamma''')/\R$ can be represented  precisely by broken cylinders with two levels, the top level being a cylinder in $\M_{K,J'}(\gamma';\gamma'')/\R$ and the bottom level being a cylinder in $\M_{K,J'}(\gamma'';\gamma''')/\R$, where the orbit $\gamma''\in\P(\H',K,\gamma)$ satisfying $|\gamma''|=\ell$ can be good or bad. However, the correspondence between ends of $\M_{K,J'}(\gamma';\gamma''')$ and such broken cylinders consisting of pairs $C_+\in\M_{K,J'}(\gamma';\gamma'')/\R$ and $C_-\in\M_{K,J'}(\gamma'';\gamma''')/\R$, for orbits $\gamma''$ as above, is not 1-1: each end is represented by $m_{\gamma''}$ pairs.

We proceed and compute
\begin{equation}\label{term_d2}
\begin{aligned}
& m_{\gamma'''} \left<d^2(\gamma'),\gamma'''\right> \\
& = \sum_{\gamma'' \text{ is good,}|\gamma''|=\ell} \frac{n(\gamma',\gamma'') n(\gamma'',\gamma''')}{m_{\gamma''}} = \sum_{|\gamma''|=\ell} \frac{n(\gamma',\gamma'') n(\gamma'',\gamma''')}{m_{\gamma''}} \\
& = \sum_{|\gamma''|=\ell} \sum_{\begin{array}{c} C_1 \in \M_{K,J'}(\gamma';\gamma'')/\R \\ C_2 \in \M_{K,J'}(\gamma'';\gamma''')/\R \end{array}} \frac{\epsilon(C_1) \epsilon(C_2)}{m_{\gamma''}} \\
\end{aligned}
\end{equation}
The factor $1/m_{\gamma''}$ takes into account the ambiguity in the representation of the ends of $\M_{K,J'}(\gamma';\gamma''')/\R$ by pairs $(C_1,C_2)$, so that the above sum can be seen as a signed count of ends, with each end counted once.

Now we prove that~\eqref{term_d2} vanishes. 
Consider (possibly bad) orbits $\gamma''_0,\gamma''_1\in \P(\H',K,\gamma)$ satisfying $|\gamma''_i|=\ell$ and let $E^i_+ \subset\M_{K,J'}(\gamma';\gamma_i'')$, $i=0,1$, be connected compact subsets, that is, compact intervals of $\R$-orbits. 
Similarly, consider also $E^i_- \subset\M_{K,J'}(\gamma_i'';\gamma''')$, $i=0,1$, compact and connected. Assuming $E^i_\pm$ are nonempty, the assumed regularity and the glueing construction give a local diffeomorphism 
\begin{equation}\label{glueing_map_d2=0}
\#_R : E^0_+ \times E^0_- \cup E^1_+ \times E^1_- \to \M_{K,J'}(\gamma';\gamma''')
\end{equation}
where $R>R_0\gg1$ is fixed large enough (depending on $E^0_+,E^1_+,E^0_-,E^1_-$); compare with~\cite[Theorem~3 in page 69]{schwarz}. Assume that there exists a connected component $Y\subset\M_{K,J'}(\gamma';\gamma''')$ such that $\#_R(E^0_+ \times E^0_- \cup E^1_+ \times E^1_-) \subset Y$, and suppose that pairs in $E^0_+\times E^0_-$ are not $\R$-translations of pairs in $E^1_+ \times E^1_-$. This means that a point in $E^0_+\times E^0_-$ and a point in $E^1_+\times E^1_-$ represent both boundary components of $Y/\R$ in the sense of SFT. Denote by $\chi^i_\pm$ the vector fields on $E^i_\pm$ given by infinitesimal $\R$-action $(i=0,1)$ and by $\left<\chi^i_\pm\right>$ the induced orientations. Let $\Pi : Y \to Y/\R$ be the quotient projection and fix an orientation preserving diffeomorphism $\phi:Y/\R\to(-1,1)$ where $Y/\R$ is equipped with the coherent orientation, 
and $(-1,1)$ is oriented by the vector field~$1$. If $(C^0_+,C^0_-) \in E^0_+\times E^0_-$, $(C^1_+,C^1_-) \in E^1_+\times E^1_-$ are chosen arbitrarily, it is crucial to observe that 
\[
d\#_R|_{(C^i_+,C^i_-)} \cdot (\chi^i_+,\chi^i_-) \in \ker d\Pi \ \ (i=0,1)
\]
and that if we let $R$ be large enough then the nonvanishing vectors 
\begin{equation}\label{nonvanish_vectors_special}
d(\phi \circ \Pi \circ \#_R)|_{(C^0_+,C^0_-)} \cdot (\chi^0_+,-\chi^0_-), \ d(\phi \circ \Pi \circ \#_R)|_{(C^1_+,C^1_-)} \cdot (\chi^1_+,-\chi^1_-)
\end{equation}
point in opposite directions as vectors in $T(-1,1)$. This is carefully explained by Schwarz~\cite{schwarz} in the finite-dimensional case, the proof carries over to our situation since glueing maps and orientations used here share analogous properties (see Lemma~4.3 in page 137 of~\cite{schwarz}). The picture is clear: let $\tau$ be a positive real number and fix a large glueing parameter $R$, then 
\[
\begin{array}{ccc}
\Pi \circ \#_R(\tau * C^0_+ , -\tau * C^0_-) & \text{and} & \Pi \circ \#_R(\tau * C^1_+ , -\tau * C^1_-)
\end{array}
\]
move towards different ends of $Y/\R$ as $\tau$ increases (here $*$ denotes the $\R$-action), implying that the vectors in~\eqref{nonvanish_vectors_special} point in opposite directions. Now we are position to complete the proof. The relevant signs in the terms of~\eqref{term_d2} corresponding to the SFT-boundary of $Y/\R$ are $\left<\chi^i_\pm\right> = \epsilon(C^i_\pm) \sigma_{C^i_\pm}$, $i=0,1$, where $\sigma$ denotes coherent orientation. Now, the differential of the glueing map $d\#_R$ induces at the level of orientations the same map as the linear glueing map. Thus, it respects coherent orientations: $d\#_R \cdot (\sigma_{C^i_+} \oplus \sigma_{C^i_-}) = \sigma_{\#_R(C^i_+,C^i_-)}$ $(i=0,1)$. 
Hence
\[
d\#_R \cdot (\left<\chi^i_+\right> \oplus \left<-\chi^i_-\right>) = -\epsilon(C^i_+)\epsilon(C^i_-) \sigma_{\#_R(C^i_+,C^i_-)} \ \ \ \ (i=0,1)
\]
Since the vectors~\eqref{nonvanish_vectors_special} point toward opposite ends of $(-1,1)$ we get
\[
\epsilon(C^0_+)\epsilon(C^0_-) + \epsilon(C^1_+)\epsilon(C^1_-) = 0.
\]
This holds for every noncompact connected component $Y/\R$ of $\M_{K,J'}(\gamma';\gamma''')/\R$, hence~\eqref{term_d2} vanishes.
\end{proof}

\begin{definition}[Local Contact Homology]\label{defn_lch}
Let $\gamma$ be an isolated closed orbit for the special stable Hamiltonian structure $\H = (\xi,X,\omega)$, and take a small isolating neighborhood $K$ for $\gamma$. Let $\H'$ be a small non-degenerate perturbation of $\H$, and $J' \in \J(\H')$ be a small perturbation of $J$ in the strong $C^\infty$-topology which is regular for the data $(\H',K,\gamma)$. The local contact homology $HC(\H,\gamma)$ of the pair $(\H,\gamma)$ is defined as the homology of the complex $(C_*(\H',K,\gamma),d)$.
\end{definition}

The remaining of Section~\ref{lch_section} is devoted to showing that this definition does not depend on the choice of $K$ and of the small perturbation $(\H',J')$ of $(\H,J)$.

\subsubsection{Chain maps}\label{chain_maps}


We consider $\H' = (\xi',X',\omega')$, $\H'' = (\xi'',X'',\omega'')$ nondegenerate $C^\infty$-small perturbations of $\H$ as explained in \S~\ref{local_complex}. Consider also $J' \in \J(\H')$ and $J'' \in \J(\H'')$ $C^\infty$-close to $J$ and regular for the data $(\H',K,\gamma)$ and $(\H'',K,\gamma)$, respectively. As before, regularity may not be achieved.

We assumed that $\H$ is special for $\gamma$. Consequently, according to Definition~\ref{special_strs}, $\H$ is either induced by some contact form $\alpha$, or the 1-form $\lambda$~\eqref{1form} is closed. In the first case consider $\Omega_0 = d(e^a\alpha)$, in the second case consider $\Omega_0 = d(Ae^a\lambda+\alpha)$ where $\alpha$ is some primitive of $\omega$ near $K$ and $A\gg1$. Here $a$ is the $\R$-coordinate. In both cases $J$ is $\Omega_0$-compatible. 

\begin{lemma}\label{lemma_symp_form_perturb}
Fix $L>0$. In both cases of Definition~\ref{special_strs} we can find a $C^\infty$-small exact perturbation $\Omega$ of $\Omega_0$ on $[-L,L]\times K$, which agrees with a positive multiple of $\omega'$ on $T(\{L\}\times K)$ and with a positive multiple of $\omega''$ on $T(\{-L\}\times K)$.
\end{lemma}
 
\begin{proof}
Let us prove this in the second case, the first case being trivial. Since $H_2(K;\R)$ vanishes and $\omega',\omega''$ are $C^\infty$-close to $\omega$, it follows from the Mayer-Vietoris principle that we can find primitives $\alpha',\alpha''$ of $\omega',\omega''$ near $K$, respectively, which are $C^\infty$-close to $\alpha$. Hence, there exists a smooth family $\{\alpha_s\}_{s\in\R}$ of $1$-forms defined near $K$ such that $\alpha_s=\alpha'$ for $s\geq L$, $\alpha_s=\alpha''$ for $s\leq -L$ and $\partial_s\alpha_s$ is uniformly $C^\infty$-small. Then take $\Omega = d(Ae^a\lambda + \alpha_a)$ on $[-L,L]\times N$ (note here that $\lambda$ is closed and $A$ is large).
\end{proof}

For any fixed $L>0$ we may find $\jbar \in \J_L(J'',J')$. Assuming that $J',J''$ are sufficiently $C^\infty$-close to $J$, we can also assume that $\jbar$ is arbitrarily close to $J$ in $C^\infty$-strong topology. Then $\jbar$ will be $\Omega$-tamed when $\Omega$ is the small perturbation of $\Omega_0$ given by Lemma~\ref{lemma_symp_form_perturb}. We can use $\Omega$ to define energies of $\jbar$-holomorphic maps.

Consider the space $\mathcal Z(\jbar,K,\gamma)$ of finite-energy $\jbar$-holomor\-phic maps $$ F = (a,f):\R\times S^1 \to \R\times K $$ such that the loops $t\mapsto f(s,t)$ are homotopic to $\gamma$ in $K$, with a positive (negative) puncture at $+\infty\times S^1$ ($-\infty\times S^1$). We suppose regularity can be achieved for such cylinders by arbitrarily small perturbations of $\jbar$ inside $\J_L(J'',J')$.
Thus, we could have assumed that $\jbar$ is regular and has the above listed geometric properties. This means that, after such perturbation, 
the linearized Cauchy-Riemann equation at every $F\in\mathcal Z(\jbar,K,\gamma)$ is a surjective Fredholm operator. In this case we call $\jbar$ regular for the data $((\H',J'),(\H'',J''),K,\gamma)$.


Similarly as in \S~\ref{local_complex}, given any $\gamma' \in \P(\H',K,\gamma)$ and $\gamma'' \in \P(\H'',K,\gamma)$ we consider the set $\F_{K,\jbar}(\gamma';\gamma'')$ of triples $(t_+,t_-,F) \in S^1 \times S^1\times \mathcal Z(\jbar,K,\gamma)$ such that $F$ satisfies $\partial_s F + \jbar(F) \partial_tF = 0$, is asymptotic to $\gamma',\gamma''$ as $s\to+\infty,-\infty$, respectively, and if we write $F=(a,f)$ then
\[
\begin{array}{ccc} f(s,t_+) \to \pt_{\gamma'} \ \text{as} \ s\to+\infty, & & f(s,t_-)\to \pt_{\gamma''} \ \text{as} \ s\to-\infty \end{array}
\]
and $a(s,t) \to \pm\infty$ as $s\to\pm\infty$. By the assumed regularity $\F_{K,\jbar}(\gamma';\gamma'')$ is a manifold of dimension $|\gamma'|-|\gamma''|+2$. It carries a free $\R\times S^1$-action defined as in~\eqref{rep_action}, and the orbit space of this action
\begin{equation}
\M_{K,\jbar}(\gamma';\gamma'') := \F_{K,\jbar}(\gamma';\gamma'')/\R\times S^1
\end{equation}
becomes a manifold of dimension $|\gamma'|-|\gamma''|$.

\begin{lemma}\label{lemma_finite_chain}
If $|\gamma'|=|\gamma''|$ then $\M_{K,\jbar}(\gamma';\gamma'')$ is finite.
\end{lemma}

The above statement is proved almost identically as Lemma~\ref{lemma_finite_d}, only note that the levels of certain relevant cylindrical buildings may satisfy the Cauchy-Riemann equation with respect to different almost complex structures. Again one relies on Lemma~\ref{lemma_nearby_cylinders} for this argument.

The moduli spaces $\M_{K,\jbar}(\gamma';\gamma'')$ for the various $\gamma',\gamma''$ as above are orientable and can be assigned a system of orientations which is coherent under glueing, see~\cite{be}. 
When $|\gamma'|=|\gamma''|$ one can associate signs $\epsilon(C)$ to elements $C\in \M_{K,\jbar}(\gamma';\gamma'')$: these moduli spaces are $0$-dimensional and one can compare the coherent orientations with the canonical orientation of the trivial vector space to obtain the above mentioned signs. There is no need to assume $\gamma'$ or $\gamma''$ is good.



Analogously to~\cite{EGH} we set
\begin{equation}\label{signed_sum_chain_map}
\bar n(\gamma',\gamma'') = \sum_{C\in \M_{K,\jbar}(\gamma';\gamma'')} \epsilon(C)
\end{equation}
if $|\gamma'|-|\gamma''|=0$, or $n(\gamma',\gamma'') = 0$ if $|\gamma'|-|\gamma''|\neq0$. As in \S~\ref{local_complex} rotating the asymptotic markers define operators on $\M_{K,\jbar}(\gamma';\gamma'')$ by
\begin{equation}
\begin{aligned}
& \bar\rho_+([t_+,t_-,F]) = [t_++1/m',t_-,F] \\
& \bar\rho_-([t_+,t_-,F]) = [t_+,t_-+1/m'',F]
\end{aligned}
\end{equation}
where $m',m''$ are the multiplicities of $\gamma',\gamma''$ respectively. Results from~\cite{be} tell us that $\rho_+$ preserves orientation if, and only if, $\gamma'$ is good. Similarly, $\rho_-$ preserves orientation if, and only if, $\gamma''$ is good. Since bad orbits have even multiplicity we conclude that $\bar n(\gamma',\gamma'')=0$ if $\gamma'$ or $\gamma''$ is bad. Finally one defines
\begin{equation}\label{continuation_map}
\Phi : C_*(\H',K,\gamma) \to C_*(\H'',K,\gamma) \ \ \text{by} \ \ \gamma' \mapsto  \sum_{\gamma'' \in \P_{0}(\H'',K,\gamma)} \frac{\bar n(\gamma',\gamma'')}{m_{\gamma''}}\gamma''.
\end{equation}

\begin{lemma}\label{cont_chain_map}
If $(\H',J')$, $(\H'',J'')$ as above are sufficiently close to $(\H,J)$ and $\jbar$ is regular and sufficiently $C^\infty$-strong close to $J$, then~\eqref{continuation_map} is a chain map with respect to the differentials defined in~\eqref{local_diff}.
\end{lemma}

\begin{proof}
The proof is analogous to that of Lemma~\ref{d2=0}. Let $\gamma_+ \in \P_0(\H',K,\gamma)$ and $\gamma_- \in \P_0(\H'',K,\gamma)$ satisfy $\ell := |\gamma_+|-1 = |\gamma_-|$. We wish to show that $\left< \Phi \circ d'(\gamma_+) - d'' \circ \Phi (\gamma_+),\gamma_- \right> = 0$, where $d'$ and $d''$ are the differentials defined using $J'$ and $J''$, respectively, in the local chain complexes generated by good closed orbits of $X'$ and $X''$ near $\gamma$. The signed sums~\eqref{signed_sum} associated to $J'$ and $J''$ will be denoted by $n'(\cdot,\cdot)$ and $n''(\cdot,\cdot)$, respectively. Abbreviate $\P'_0=\P_0(\H',K,\gamma)$, $\P'=\P(\H',K,\gamma)$.
We compute
\begin{equation}\label{first_terms_chain}
\begin{aligned}
& m_{\gamma_-} \left< \Phi \circ d'(\gamma_+),\gamma_- \right> = \sum_{\tilde\gamma \in \P_0,|\tilde\gamma|=\ell} \frac{n'(\gamma_+,\tilde\gamma) \bar n(\tilde\gamma,\gamma_-)}{m_{\tilde\gamma}} \\
& = \sum_{\tilde\gamma \in \P,|\tilde\gamma|=\ell} \frac{n'(\gamma_+,\tilde\gamma) \bar n(\tilde\gamma,\gamma_-)}{m_{\tilde\gamma}} \\
& = \sum_{\tilde\gamma \in \P,|\tilde\gamma|=\ell} \sum_{\begin{array}{c} C_1 \in \M_{K,J'}(\gamma_+;\tilde\gamma)/\R \\ C_2 \in \M_{K,\jbar}(\tilde\gamma;\gamma_-) \end{array}} \frac{\epsilon(C_1)\epsilon(C_2)}{m_{\tilde\gamma}} \\
\end{aligned}
\end{equation}
In the second equality we replaced a sum over the good orbits by a sum over all orbits since the corresponding coefficients for bad orbits vanish, as already observed before. 
Analogously
\begin{equation}\label{second_terms_chain}
\begin{aligned}
& m_{\gamma_-} \left< d'' \circ \Phi (\gamma_+),\gamma_- \right> \\ 
&= \sum_{\tilde\gamma\in\P,|\tilde\gamma|=\ell+1} \sum_{\begin{array}{c} C_1 \in \M_{K,\jbar}(\gamma_+;\tilde\gamma) \\ C_2 \in \M_{K,J''} (\tilde\gamma;\gamma_-)/\R \end{array}} \frac{\epsilon(C_1)\epsilon(C_2)}{m_{\tilde\gamma}}
\end{aligned}
\end{equation}
where here $\P''$ stands for $\P(\H'',K,\gamma)$. 

A glueing/compactness analysis, similar to the one described in the proof of Lemma~\ref{d2=0}, shows that to each term appearing in the sums~\eqref{first_terms_chain},\eqref{second_terms_chain} there corresponds an end of the 1-dimensional space $\M_{K,\jbar}(\gamma_+;\gamma_-)$. In fact, since $K$ has no finite-energy sphere of planes which are holomorphic with respect to $J'$, $J''$ or $\jbar$, sequences on $\M_{K,\jbar}(\gamma_+;\gamma_-)$ SFT-converge to cylindrical buildings. The assumed regularity implies that these have precisely two levels, one of which corresponds to a $\jbar$-holomorphic cylinder. Conversely, consider a two-level building consisting of a pair $(C_1,C_2) \in \M_{K,J'}(\gamma_+;\tilde\gamma)/\R \times \M_{K,\jbar}(\tilde\gamma;\gamma_-)$ or of a pair $(C_1,C_2) \in \M_{K,\jbar}(\gamma_+;\tilde\gamma) \times \M_{K,J''}(\tilde\gamma;\gamma_-)/\R$. The important observation is that by Lemma~\ref{lemma_nearby_cylinders} the maps representing $C_1,C_2$ project compactly inside ${\rm int}(K)$. Hence, using the assumed regularity, these can be glued to obtain $1$-parameter family of cylinders in $\M_{K,\jbar}(\gamma_+;\gamma_-)$. Hence, as usual, this glueing/compactness argument shows that these pairs $(C_1,C_2)$ form the SFT-boundary of $\M_{K,\jbar}(\gamma_+;\gamma_-)$.

However, the correspondence between ends of $\M_{K,\jbar}(\gamma_+;\gamma_-)$ and terms of~\eqref{first_terms_chain}, \eqref{second_terms_chain} is not 1-1. Each end associated to a 2-level broken cylinder corresponds precisely to $m_{\tilde\gamma}$ terms, where $\tilde\gamma$ is the asymptotic orbit between the levels. 
Hence, after dividing by the corresponding factors $m_{\tilde\gamma}$, adding~\eqref{first_terms_chain} with~\eqref{second_terms_chain} gives a signed count of ends of $\M_{K,\jbar}(\gamma_+;\gamma_-)$ with each end counted precisely once.

Let $Y$ be a connected component of $\M_{K,\jbar}(\gamma_+;\gamma_-)$. There are three possibilities:
\begin{itemize}
\item[I)] Both ends of $Y$ correspond to terms in~\eqref{first_terms_chain}.
\item[II)] Both ends of $Y$ correspond to terms in~\eqref{second_terms_chain}.
\item[III)] One end of $Y$ corresponds to a term in~\eqref{first_terms_chain} and the other end corresponds to a term in~\eqref{second_terms_chain}. \\
\end{itemize}

\noindent {\bf Cases I and II:} Let us handle I, case II is similar and left to the reader. There exists $\tilde\gamma_i \in \P(\H',K,\gamma)$ satisfying $|\tilde\gamma_i|=|\gamma_+|-1$, connected components $E^i_+$ of $\M_{K,J'}(\gamma_+;\tilde\gamma_i)$, and $C^i_-\in \M_{K,\jbar}(\tilde\gamma_i;\gamma_-)$ $(i=0,1)$ such that $(E^0_+/\R,C^0_-)$ and $(E^1_+/\R,C^1_-)$ correspond to the ends of $Y$ in the sense of SFT. 
Let $J^i \subset E^i_+$ be nontrivial compact connected subsets (compact intervals of $\R$-translations). Glueing yields a local diffeomorphism
\[
\#_R : J^0 \times \{C^0_-\} \cup J^1 \times \{C^1_-\} \to Y.
\]
Regularity is crucial to get this map well-defined; please see the end of~\cite[page 97]{schwarz} and note that the glueing map is a local embedding when the glueing parameter is ``frozen'' at a very large value. We insist on writing the domain of $\#_R$ as (a union of) products, similarly as in~\eqref{glueing_map_d2=0}, since we need to deal with orientations: the point-spaces $\{C^i_-\}$ are oriented by the coherent orientation (denoted by $\sigma$) and the signs $\epsilon(C^i_-)$ are given by\footnote{The tangent space of the point $\{C^i_-\}$ is the trivial vector space which has exterior algebra equal to $\R \ni 1$.} $\left< 1 \right> = \epsilon(C^i_-) \sigma_{C^i_-}$. For any given points $C^i_+ \in J^i$, the signs $\epsilon(C^i_+)$ are given by $\left<\chi^i_+\right> = \epsilon(C^i_+)\sigma_{C^i_+}$, where $\chi^i_+$ denotes the infinitesimal $\R$-action. Let us fix an orientation preserving diffeomorphism $\phi : Y \to (-1,1)$, where $(-1,1)$ is given its standard orientation and $Y$ is given the coherent orientation. Standard glueing analysis tells us that the vectors
\[
\begin{array}{ccc} d(\phi \circ \#_R)|_{(C^0_+,C^0_-)} \cdot (\chi^0_+,0) & \text{and} & d(\phi \circ \#_R)|_{(C^1_+,C^1_-)} \cdot (\chi^1_+,0) \end{array}
\]
point in opposite directions as vectors in $(-1,1)$, see Lemma~4.5 in page 141 of~\cite{schwarz} for the Morse theoretical case. At the level of orientations $d\#_R$ induces the same map as the linear-glueing described in~\cite{be}. Thus $d\#_R$ respects coherent orientations. We get
\[
d\#_R \cdot ( \left<\chi^i_+\right> \oplus \left< 1 \right>) = \epsilon(C^i_+)\epsilon(C^i_-)\sigma_{\#_R(C^i_+,C^i_-)} \ \ \ \ \ (i=0,1).
\]
Consequently $\epsilon(C^0_+)\epsilon(C^0_-) = -\epsilon(C^1_+)\epsilon(C^1_-)$, completing the proof that the contribution of all terms associated to the ends of $Y$ to 
$\eqref{first_terms_chain} - \eqref{second_terms_chain}$ is zero. Note here that the terms just analyzed correspond to terms only in~\eqref{first_terms_chain}. \\

\noindent {\bf Case III:} There exist $\tilde\gamma_0 \in \P(\H',K,\gamma)$, $\tilde\gamma_1 \in \P(\H'',K,\gamma)$ satisfying $|\tilde\gamma_1|=|\gamma_+|=|\tilde\gamma_0|+1$, a connected component $E^0_+$ of $\M_{K,J'}(\gamma_+;\tilde\gamma_0)$, a connected component $E^1_-$ of $\M_{K,J''}(\tilde\gamma_1;\gamma_-)$, $C^0_- \in \M_{K,\jbar}(\tilde\gamma_0;\gamma_-)$, $C^1_+ \in \M_{K,\jbar}(\gamma_+;\tilde\gamma_1)$, for $i=0,1$, such that $(E^0_+/\R,C^0_-)$ and $(C^1_+,E^1_-/\R)$ correspond to the ends of $Y$ in the sense of SFT. 
Let $J^0 \subset E^0_+$, $J^1 \subset E^1_-$ be nontrivial compact connected subsets (compact intervals of $\R$-translations). Glueing yields a local diffeomorphism
\[
\#_R : J^0 \times \{C^0_-\} \cup \{C^1_+\} \times J^1 \to Y
\]
where $R$ is a large glueing parameter. Regularity is crucial to get this map well-defined. As in cases I and II, the point-spaces $\{C^0_-\}$, $\{C_+^1\}$ are oriented by the coherent orientation (denoted by $\sigma$) and the signs $\epsilon(C^0_-)$, $\epsilon(C^1_+)$ are given by $\left< 1 \right> = \epsilon(C^0_-) \sigma_{C^0_-}$, $\left< 1 \right> = \epsilon(C^1_+) \sigma_{C^1_+}$. For any given points $C^0_+ \in J^0$ and $C^1_- \in J^1$, the signs $\epsilon(C^0_+)$, $\epsilon(C^1_-)$ are given by $\left<\chi^0_+\right> = \epsilon(C^0_+)\sigma_{C^0_+}$, $\left<\chi^1_-\right> = \epsilon(C^1_-)\sigma_{C^1_-}$ where $\chi^0_+$,$\chi^1_-$ denote the infinitesimal $\R$-actions. Let us fix an orientation preserving diffeomorphism $\phi : Y \to (-1,1)$, where $(-1,1)$ is given its standard orientation and $Y$ is given the coherent orientation. The drastic difference with cases I and II is that the vectors
\[
\begin{array}{ccc} d(\phi \circ \#_R)|_{(C^0_+,C^0_-)} \cdot (\chi^0_+,0) & \text{and} & d(\phi \circ \#_R)|_{(C^1_+,C^1_-)} \cdot (0,\chi^1_-) \end{array}
\]
point in the {\bf same direction} when seen as vectors in $(-1,1)$, see Lemma~4.5 in page 141 of~\cite{schwarz} for the Morse theoretical case. The reader acquainted with the glueing construction will notice that the geometrical picture is clear. Since $d\#_R$ induces the same map as the linear-glueing described in~\cite{be} at the level of orientations, $d\#_R$ respects coherent orientations. We get
\[
\begin{array}{c}
d\#_R \cdot ( \left<\chi^0_+\right> \oplus \left< 1 \right>) = \epsilon(C^0_+)\epsilon(C^0_-)\sigma_{\#_R(C^0_+,C^0_-)} \\
d\#_R \cdot (\left< 1 \right> \oplus \left<\chi^1_-\right>) = \epsilon(C^1_+)\epsilon(C^1_-)\sigma_{\#_R(C^1_+,C^1_-)}
\end{array}
\]
Differently from cases I and II, we arrive at $\epsilon(C^0_+)\epsilon(C^0_-) = \epsilon(C^1_+)\epsilon(C^1_-)$. Thus the terms associated to the ends of $Y$ contribute with zero to the number
\[
m_{\gamma_-} \left< \Phi \circ d'(\gamma_+) - d'' \circ \Phi(\gamma_+),\gamma_- \right> = \eqref{first_terms_chain} - \eqref{second_terms_chain}.
\]
\end{proof}

\subsubsection{Homotopies}\label{homotopies}

Two chain maps as in~\eqref{continuation_map} turn out to be chain homotopic. To see this, consider $\H',\H''$ small non-degenerate perturbations of $\H$. Consider also $J'\in\J(\H')$, $J''\in\J(\H'')$ small perturbations of $J$ which are regular for the data $(\H',K,\gamma)$ and $(\H'',K,\gamma)$, respectively, and to which the conclusions of Lemma~\ref{lemma_nearby_cylinders} apply.

Consider choices $\jbar_0,\jbar_1 \in \J_L(J'',J')$ which are $C^\infty$-strong close to $J$ and regular for the data $((\H',J'),(\H'',J''),K,\gamma)$, with the properties required in \S~\ref{chain_maps}. 
Note that $J',J''$ are allowed to be taken arbitrarily $C^\infty$-strong close to $J$. Then we may find $\jbar_0,\jbar_1$ and the homotopy $\{\jbar_\tau\} \in \J_{\tau,L}(J'',J')$ connecting $\jbar_0$ to $\jbar_1$ lying on a arbitrarily given small $C^\infty$-strong neighborhood of $J$ (uniformly in $\tau\in[0,1]$). Moreover, with the help of Lemma~\ref{lemma_symp_form_perturb} one finds an exact symplectic form $\Omega$ on $[-L,L]\times K$ $C^\infty$-close to $\Omega_0$ as described in \S~\ref{chain_maps}, which tames all $\jbar_\tau$, equals $\omega''$ on $T(\{-L\}\times N)$ up to a positive constant, and equals $\omega'$ on $T(\{L\}\times N)$ up to a positive constant. Such a symplectic form can be used to define the energy of $\jbar_\tau$-holomorphic maps, for all $\tau\in[0,1]$.

We need the path $\{\jbar_\tau\}$ to be regular for the data $((\H',J'),(\H'',J''),K,\gamma)$ in the following sense. Let $\mathcal Z(\{\jbar_\tau\},K,\gamma)$ denote the set of pairs $(z,F)$, where $z\in[0,1]$, and $F:\R\times S^1 \to \R\times K$ is a finite-energy $\jbar_z$-holomorphic cylinder with a positive/negative puncture at $+\infty\times S^1$/$-\infty\times S^1$, asymptotic to an orbit in $\P(\H',K,\gamma)$ at the positive puncture and to an orbit in $\P(\H'',K,\gamma)$ at the negative puncture. Regularity of $\{\jbar_\tau\}$ means that, in a standard functional analytical set-up, the linearization of the ($\tau$-dependent) Cauchy-Riemann equation at every $(z,F) \in \mathcal Z(\{\jbar_\tau\},K,\gamma)$ is surjective. We proceed assuming that any path in $\J_{\tau,L}(J'',J')$ can be slightly and uniformly (in $\tau$) $C^\infty$-strong perturbed to a regular path keeping the endpoints $\jbar_0,\jbar_1$ fixed.



For $\gamma' \in \P(\H',K,\gamma)$ and $\gamma'' \in \P(\H'',K,\gamma)$ we consider the set $\F_{K,\{\jbar_\tau\}}(\gamma';\gamma'')$ consisting of triples $(t_+,t_-,(z,F)) \in S^1\times S^1 \times \mathcal Z(\{\jbar_\tau\},K,\gamma)$ where $F:\R\times S^1\to \R\times K$ is a finite-energy solution of $\partial_s F+ \jbar_z(F) \partial_t F = 0$, asymptotic to the orbits $\gamma'$/$\gamma''$ at $+\infty\times S^1$/$-\infty\times S^1$. Moreover, writing $F=(a,f)$ then $a(s,t) \to \pm\infty$ as $s\to\pm\infty$, $f(s,t_+) \to \pt_{\gamma'}$ as $s\to+\infty$ and $f(s,t_-) \to \pt_{\gamma''}$ as $s\to-\infty$. The energy is defined using the taming symplectic forms mentioned above. In view of regularity $\F_{K,\{\jbar_{\tau,t}\}}(\gamma';\gamma'')$ is a smooth manifold of dimension $|\gamma'|-|\gamma''|+3$. As before, $\R\times S^1$ acts freely on $\F_{K,\{\jbar_\tau\}}(\gamma';\gamma'')$, so that the quotient space 
\[
\M_{K,\{\bar J_\tau\}}(\gamma';\gamma'') = \F_{K,\{\jbar_\tau\}}(\gamma';\gamma'') / \R\times S^1
\]
is a manifold of dimension $|\gamma'|-|\gamma''|+1$.


\begin{lemma}\label{lemma_finite_homotopy}
$\M_{K,\{\jbar_\tau\}}(\gamma';\gamma'')$ is finite when $|\gamma'|-|\gamma''|=-1$.
\end{lemma}

The proof is entirely analogous to that of Lemma~\ref{lemma_finite_d} and will be omitted. We need automatic energy bounds for elements in $\F(\{\jbar_\tau\},K,\gamma)$, which is guaranteed by Lemma~\ref{lemma_nearby_cylinders} in view of the special form of our small perturbations of the data $(\H,J)$ and by the properties of $\Omega$.

When $|\gamma'|-|\gamma''|+1=0$ there are signs $\epsilon([t_+,t_-,(z,F)])$ associated to each element $[t_+,t_-,(z,F)]$ of $\M_{K,\{\jbar_\tau\}}(\gamma';\gamma'')$ induced by a system of orientations which is coherent with the glueing operation; see~\cite{be}. 
We set
\begin{equation}\label{signed_sum_degree_+1}
\widetilde n(\gamma',\gamma'') = \sum_{C \in \M_{K,\{\jbar_\tau\}}(\gamma',\gamma'')} \epsilon(C)
\end{equation}
when $|\gamma'|-|\gamma''| = -1$, or $\widetilde n(\gamma',\gamma'') = 0$ otherwise. One defines a degree +1 map
\begin{equation}
Q : C_*(\H',K,\gamma) \to C_{*+1}(\H'',K,\gamma) \ \text{ by } \ Q\gamma' = \sum_{\gamma'' \in \P_0(K,\H'',\gamma)} \frac{\widetilde n(\gamma',\gamma'')}{m_{\gamma''}} \gamma''
\end{equation}
on generators.

\begin{lemma}\label{lemma_homotopies}
Let $\Phi_0,\Phi_1:C_*(\H',K,\gamma) \to C_*(\H'',K,\gamma)$ be the chain maps~\eqref{continuation_map} induced by $\jbar_0,\jbar_1$, respectively. Then $\Phi_1 - \Phi_0 = Q\circ d' - d''\circ Q$.
\end{lemma}

The argument is analogous to the ones given to prove Lemma~\ref{d2=0} and Lemma~\ref{cont_chain_map}, only note here that moduli spaces $\M_{K,\{\jbar_\tau\}}(\gamma_+;\gamma_-)$ with $\gamma_+ \in \P(\H',K,\gamma)$ and $\gamma_- \in \P(\H'',K,\gamma)$ satisfying $|\gamma_+|=|\gamma_-|$ do have ``genuine'' boundary points, corresponding to $\M_{K,\jbar_0}(\gamma_+;\gamma_-)$ and $\M_{K,\jbar_1}(\gamma_+;\gamma_-)$. As a final remark, these arguments strongly rely on Lemma~\ref{lemma_nearby_cylinders} which ensures that all relevant cylinders project compactly inside ${\rm int}(K)$. The proofs of lemmas~\ref{lemma_finite_homotopy} and~\ref{lemma_homotopies} will be omitted since they are very similar to the proofs of lemmas~\ref{lemma_finite_d},~\ref{d2=0}, and~\ref{cont_chain_map}.


\subsubsection{Stability of local contact homology}\label{stability_lch}

Here we study the chain maps~\eqref{continuation_map} more closely and show that they induce isomorphisms at the homology level. Consider, as before, pairs $(\H',J')$, $(\H'',J'')$, where $\H'=(\xi',X',\omega')$ and $\H''=(\xi'',X'',\omega'')$ are special nondegenerate small perturbations of $\H$, and $J'\in\J(\H')$, $J''\in\J(\H'')$ are small perturbations of $J$ which are regular for the data $(\H',K,\gamma)$ and $(\H'',K,\gamma)$, respectively. 

We can assume that the conclusions of Lemma~\ref{lemma_nearby_orbits} hold for $\H',\H''$, and those of Lemma~\ref{lemma_nearby_cylinders} hold for $J',J''$. Fix $L>0$ and choose $\jbar^+ \in \J_L(J'',J')$, $\jbar^- \in \J_L(J',J'')$ sufficiently $C^\infty$-strong small perturbations of $J$, to which the conclusions of Lemma~\ref{lemma_nearby_cylinders} also apply. This can be achieved since we have the freedom of choosing $J',J''$ as $C^\infty$-strong close to $J$ as we want. We also assume that $\jbar^\pm$ are regular, as explained in \S~\ref{chain_maps}. 
The energy of $\jbar^\pm$-holomorphic maps are defined using certain symplectic forms on $[-L,L]\times K$ as constructed in \S~\ref{chain_maps}. For any $R>L$ we consider $\jbar^R \in \J_{L<R}(J')$ defined by
\begin{equation}
\jbar^R = \left\{ \begin{aligned} & (\tau_{-R})^*\jbar^+ \text{ on } [0,+\infty) \times N \\ & (\tau_R)^*\jbar^- \text{ on } (-\infty,0]\times N\end{aligned} \right.
\end{equation}
which is smooth since $R>L$; see \S~\ref{acs_cobordisms} for the definition of $\J_{L<R}(J')$. 
Note that if $J',J''$, $\jbar^\pm$ are chosen sufficiently close to $J$ then $\jbar^R$ is forced to lie on any given arbitrarily small neighborhood of $J$ in the $C^\infty$-strong topology, uniformly for all $R>L$.

\begin{lemma}
The chain maps $\Phi_\pm$~\eqref{continuation_map} induced by $\jbar^\pm$ satisfy $\Phi_- \circ \Phi_+ = id$ in homology.
\end{lemma}

\begin{proof}
If $\gamma'_\pm \in \P(\H',K,\gamma)$ satisfy $|\gamma'_+|=|\gamma'_-|=:\ell$ then
\begin{equation}\label{sum_iso}
\begin{aligned}
& m_{\gamma'_-} \left< \Phi_- \circ \Phi_+ (\gamma'_+),\gamma'_- \right> \\
& = \sum_{\begin{array}{c} \gamma'' \in \P(\H'',K,\gamma) \\ |\gamma''|=\ell \end{array}} \sum_{\begin{array}{c} C_1 \in \M_{K,\jbar^+}(\gamma'_+;\gamma'') \\ C_2 \in \M_{K,\jbar^-}(\gamma'';\gamma'_-) \end{array}} \frac{\epsilon(C_1)\epsilon(C_2)}{m_{\gamma''}} \\
\end{aligned}
\end{equation}
Note that if $\gamma''$ is bad the corresponding inner-sum vanishes because rotating asymptotic markers at a puncture asymptotic to a bad orbit reverses orientations.


Now when $R$ is large enough, using regularity of all almost complex structures involved, we can glue a given pair
\begin{equation}\label{iso_pair_to_be_glued}
(C_1,C_2) \in \M_{K,\jbar^+}(\gamma'_+;\gamma'') \times \M_{K,\jbar^-}(\gamma'';\gamma'_-)
\end{equation}
to obtain an element of $\M_{K,\jbar^R}(\gamma'_+;\gamma'_-)$. In fact, using the assumed regularity glueing can be performed, but {\it a priori} it could not yield a cylinder with image in $\R\times K$. However, the important fact is that by Lemma~\ref{lemma_nearby_cylinders} the maps representing $C_1,C_2$ project compactly inside ${\rm int}(K)$. Thus the glued cylinder indeed represents an element of $\M_{K,\jbar^R}(\gamma'_+;\gamma'_-)$.


By a compactness argument one shows that every element of $\M_{K,\jbar^R}(\gamma'_+;\gamma'_-)$ arises this way when $R$ is fixed large enough. However, as in previous proofs in this section, the correspondence between terms of~\eqref{sum_iso} and elements of $\M_{K,\jbar^R}(\gamma'_+;\gamma'_-)$, for large and fixed $R$, is not 1-1. In fact, glueing can be performed in $m_{\gamma''}$ different ways and this is the reason for the coefficient $1/m_{\gamma''}$ in each term of~\eqref{sum_iso}. Hence~\eqref{sum_iso} is precisely a signed count of the elements of $\M_{K,\jbar^R}(\gamma'_+;\gamma'_-)$ with each element counted once. Consequently, in view of the coherence between glueing and the chosen orientations, we conclude that
\[
\eqref{sum_iso} = \sum_{C \in \M_{K,\jbar^R}(\gamma'_+;\gamma'_-)} \epsilon(C) = m_{\gamma'_-} \left< \Phi_R(\gamma'_+),\gamma'_- \right>
\]
where $\Phi_R$ denotes the map~\eqref{continuation_map} associated to $\jbar^R$. By standard glueing analysis, using the assumed regularity of all almost complex structures involved, it follows that $\jbar^R$ is also regular when $R$ is large enough. This shows that $$ \Phi_- \circ \Phi_+ = \Phi_R $$ when $R\gg1$.

%
%

Now we consider a regular homotopy between $\jbar^R$ and the $\R$-invariant $J'$. This last almost complex structure is taken regular, so that the associated chain map~\eqref{continuation_map} is necessarily the identity at the chain level. Thus, by Lemma~\ref{lemma_homotopies}, $\Phi_R$ induces the identity map at the level of homology.
\end{proof}

\begin{corollary}\label{local_independence}
Suppose that $\H',\H''$ are sufficiently small special nondegenerate $C^\infty$-perturbations of $\H$, and also that $J' \in \J(\H')$ and $J'' \in \J(\H'')$ are $C^\infty$-strong close to $J$ and regular for the data $(\H',K,\gamma)$ and $(\H'',K,\gamma)$, respectively. Then the homologies of the chain complexes $(C_*(\H',K,\gamma),d')$ and $(C_*(\H'',K,\gamma),d'')$ defined as in \S~\ref{local_complex} by the data $(\H',J')$ and $(\H'',J'')$, respectively, are isomorphic.
\end{corollary}

It follows from our discussion that there are well-defined graded vector spaces $HC_*(\H,K,\gamma,J)$ given by the homology of the chain complex $(C_*(\H',K,\gamma),d')$ where $K$ is a small tubular neighborhood of $\gamma$ and the data $(\H',J') \simeq (\H,J)$ is carefully chosen as in the above discussion. We still need to address the independence of $HC_*(\H,K,\gamma,J)$ on $J$ and $K$.

\subsection{Invariance of local contact homology}

\begin{lemma}\label{invariance_theorem}
Let $\{\H^s = (\xi^s,X^s,\omega^s)\}_{s\in[0,1]}$ be a smooth family of stable hamiltonian structures on a manifold $N$, and $J^s \in \J(\H^s)$ be a smooth 1-parameter family of $\R$-invariant almost complex structures. Let $\gamma$ be a closed $X^0$-orbit and let $K$ be a small compact tubular neighborhood of (the geometric image of) $\gamma$ such that for every $s\in[0,1]$ the following hold:
\begin{itemize}
\item[(a)] The vector field $X^s$ is a pointwise positive multiple of $X^0$ on the geometric image of $\gamma$.
\item[(b)] $\gamma$ is the only closed orbit of $X^s$ contained in $K$ in its free homotopy class (of loops in $K$).
\item[(c)] $X^s$ has no closed orbit contained in $K$ which is contractible in $K$.
\item[(d)] Either $\H^s$ is induced by some contact form on $K$, or the 1-form $\lambda^s$ associated to $\H^s$ as in~\eqref{1form} is closed on $K$ (see Definition~\ref{special_strs}).
\end{itemize}
Then $HC_*(\H^0,K,\gamma,J^0) \simeq HC_*(\H^1,K,\gamma,J^1)$.
\end{lemma}

In (b) above we abuse the notation and see $\gamma$ as a closed $X^s$-orbit. This is possible in view of (a).

\begin{proof}
It is an immediate consequence of Corollary~\ref{local_independence} that for every $s_0\in [0,1]$ there exists $\epsilon>0$ such that $HC_*(\H^s,K,\gamma,J^s) = HC_*(\H^{s_0},K,\gamma,J^{s_0})$ for all $s\in[0,1]$ satisfying $|s-s_0|<\epsilon$. In fact, if not, we find a sequence $s_n \to s_0$ such that $HC_*(\H^{s_n},K,\gamma,J^{s_n}) \neq HC_*(\H^{s_0},K,\gamma,J^{s_0})$, $\forall n$. There are very small $C^\infty$-perturbations $(\H'_n,J'_n)$ of $(\H^{s_n},J^{s_n})$ such that $\H'_n$ is nondegenerate, $J'_n$ is regular for the data $(\H'_n,K,\gamma)$, $(\H'_n,J'_n) \to (\H^{s_0},J^{s_0})$ in $C^\infty$ as $n\to\infty$, and the conclusions of Lemma~\ref{lemma_nearby_cylinders} hold for all $J'_n$. Moreover, we can assume that the homology of the chain complex $(C_*(\H'_n,K,\gamma),d)$, where $d$ is defined using $J'_n$, is $HC_*(\H^{s_n},K,\gamma,J^{s_n})$. However, Corollary~\ref{local_independence} says that these homology groups are also equal $HC_*(\H^{s_0},K,\gamma,J^{s_0})$ when $n$ is large, a contradiction. The conclusion now follows from compactness of $[0,1]$.
\end{proof}

As a consequence we can drop the dependence on $J$ of the local contact homology of the data $(\H,K,\gamma,J)$. It is easy to see that it is also independent of the small tubular neighborhood $K$ where $\gamma$ is the only closed Hamiltonian orbit in its free homotopy class (of loops in $K$). We will write $HC(\H,\gamma)$ for simplicity.

\section{Local contact homology of isolated prime Reeb orbits} \label{lch:isolated}

In this section we establish the relation between local contact homology of an isolated prime Reeb orbit and the associated Poincar\'e return map to a local cross section.

\begin{proposition}\label{comp_prop}
Let $\alpha$ be a contact form on a manifold $N$, and $\gamma$ be an isolated prime Reeb orbit. Let $\Sigma \subset N$ be an embedded hypersurface transverse to $\gamma$ at a point $p \in \gamma$, so that the local first return map $\varphi:(U,p) \to (\Sigma,p)$ is well-defined on a small neighborhood $U$ of $p$ in $\Sigma$. Then $HC(\alpha,\gamma)$ and $HF(\varphi,p)$ are isomorphic.
\end{proposition}

In the above statement we denote by $HF(\varphi,p)$ the local Floer homology at the isolated fixed point $p$ of the germ of symplectic diffeomorphism $\varphi$ of the symplectic manifold $(\Sigma,d\alpha|_\Sigma)$. The isomorphism in Proposition~\ref{comp_prop} is defined only up to an even shift in the grading, since the grading of local Floer homology of a germ of Hamiltonian diffeomorpohism near an isolated fixed point is only defined up to an even shift, see~\cite{GG}.

\subsection{Local models}

\begin{lemma}\label{local_coord}
Let $\alpha$ be a contact form on a $(2n-1)$-dimensional manifold, and $\gamma = (x,T)$ be a prime closed $\alpha$-Reeb orbit. Then there exists a tubular neighborhood $K\simeq \R/\Z\times \overline B$ of $x(\R)$, where $B\subset \R^{2n-2}$ is a small open ball centered at the origin, with coordinates $(t,q_1,\dots,q_{n-1},p_1,\dots,p_{n-1})$, such that $x(\R) \simeq \R/\Z\times 0$, $\alpha \simeq Hdt + \lambda_0$, where $H:K\to \R$ satisfies $H_t(0) = T$, $dH_t(0) = 0$, and $\lambda_0 = \frac{1}{2} \sum_{k=1}^{n-1} q_kdp_k - p_kdq_k$.
\end{lemma}

\begin{proof}
First, it is simple to get a tubular neighborhood $K\simeq \R/\Z\times \overline B$ such that $x(\R)=\R/\Z\times\{0\}$, $\alpha|_{\R/\Z\times 0} = Tdt$ and $d\alpha$ restricted to $0\times \R^{2n-2} \subset T_{(t,0)}(\R/\Z\times B)$ coincides with $\omega_0$, $\forall t\in\R/\Z$. By a parame\-trized version of Darboux's theorem for symplectic forms, we can change coordinates to obtain $d\alpha|_{T(t\times \overline B)} = \omega_0$, $\forall t$. 

Now, let the $\alpha$-Reeb flow be denoted by $\phi_t$. On a small neighborhood $U$ of $0\in \R^{2n-2}$ we find a smooth function $\tau : [0,1]\times U \to \R$ such that $\phi_{\tau(t,z)}(0,z) \in t\times \overline B$. The maps $\varphi_t$ defined by $(t,\varphi_t(z)) = \phi_{\tau(t,z)}(0,z)$ on $U$ are symplectic embeddings fixing the origin. Hence, we can find a smooth Hamitonian $H_t$ defined near $0$ such that $\varphi_t$ is its Hamiltonian flow. Moreover, $H_t$ can be arranged to be 1-periodic on $t$ since so is $\dot\varphi_t\circ\varphi_t^{-1}$ and, consequently, $H_t$ defines a smooth function near $\R/\Z\times 0$. There is no loss of generality to assume that $H_t(0) = T$. It must satisfy $dH_t(0) = 0$ since $0$ is left fixed. Consider the vector field $\widetilde X_H = \partial_t + X_{H_t}$, where $dH_t = i_{X_{H_t}}\omega_0$. By the definition of $\varphi_t$ we get $i_{\widetilde X_H}d\alpha = 0$. But our coordinates obtained so far guarantee that $d\alpha = \beta_t\wedge dt + \omega_0$, for some 1-periodic smooth family of 1-forms $\beta_t$ defined near $0\in\R^{2n-2}$. Consequently $$ 0 = i_{\widetilde X_H}d\alpha = (i_{X_{H_t}}\beta_t) dt - \beta_t + i_{X_{H_t}}\omega_0 = (i_{X_{H_t}}\beta_t) dt - \beta_t + dH_t $$ proving that $\beta_t = dH_t$. In other words, $d\alpha = dH_t \wedge dt + \omega_0$.

Let $\alpha_1 = Hdt + \lambda_0$, so that $d\alpha=d\alpha_1$. Moreover, $\int_{\R/\Z\times 0} \alpha-\alpha_1 = 0$ and, consequently, we find a smooth function $f$ defined near $\R/\Z\times 0$ such that $df = \alpha-\alpha_1$. After subtracting a constant we can assume $f=0$ on $\R/\Z\times 0$. Consider $\alpha_s = (1-s)\alpha + s\alpha_1$ and the vector field $Y_s = fX_{\alpha_s}$ where, for each $s\in[0,1]$, $X_{\alpha_s}$ is the Reeb vector of the contact form $\alpha_s$ (it is easy to see that $\alpha_s$ are contact forms near $\R/\Z\times0$). Denoting by $\psi_s$ the flow of $Y_s$ we get $$ \frac{d}{ds} \psi_s^*\alpha_s = \psi_s^*(i_{Y_s}d\alpha_s + d(i_{Y_s}\alpha_s) + \alpha_1-\alpha) = \psi_s^*(df + \alpha_1-\alpha) = 0. $$ Moreover, $\R/\Z\times 0$ is left fixed by $\psi_s$. Using $\psi_1$ we obtain the desired coordinates.
\end{proof}

\subsection{Proof of Proposition~\ref{comp_prop}}

Let $\gamma = (x,T)$ be a prime closed isolated Reeb orbit for a contact form $\alpha$, as in the statement of Proposition~\ref{comp_prop}. In view of Lemma~\ref{local_coord} we work on $K = \R/\Z\times \overline B$ with coordinates $(t,q_1,\dots,p_1,\dots)$ and assume that $\alpha = H_tdt+\lambda_0$, $H_t(0) = T$, $dH_t(0) = 0$, and $x(t) = (t/T,0)$. Also, we assume that $\gamma$ is the only closed $\alpha$-Reeb orbit which goes once around the tube. Thus we can take $(\Sigma,d\alpha) = (0\times B,\omega_0)$.

The 1-forms $\alpha_s = (1-s)\alpha + sdt$, $s\in[0,1]$, are contact forms on $K$ for $s<1$ if $B$ is small enough, and $d\alpha_s = (1-s)d\alpha = (1-s)\omega_H$, where $\omega_H = dH_t\wedge dt + \omega_0$. Consequently the Reeb vector fields $X_{\alpha_s}$, $s\in[0,1)$, are all positive multiples of each other, proving that $x(\R)$ is the only closed $\alpha_s$-Reeb orbit going once around the tube. Consider the family $$ \H_s = (\xi_s = \ker\alpha_s,X_{\alpha_s},d\alpha=\omega_H), \ \ s\in[0,1) $$ of stable Hamiltonian structures. It can be smoothly continued to $[0,1]$ by setting $$ \H_1 = (\xi_1=\ker dt, \widetilde X_H,\omega_H) $$ where $\widetilde X_H = \partial_t + X_{H_t}$. Since the 2-form is independent of $s$, the conditions of Lemma~\ref{invariance_theorem} are fulfilled, so that $HC_*(\H_s,\gamma)$ does not depend on $s\in[0,1]$. It is easy to check that $HC_*(\H_1,\gamma)$ coincides with the local Floer homology of the isolated 1-periodic orbit $0$ of the Hamiltonian $H_t$, up to an even shift in the grading since the homotopy class of $d\alpha$-symplectic trivializations along $\gamma$ induced by the choice of coordinates given by Lemma~\ref{local_coord} was not specified. This concludes the argument.

\section{Estimating local contact homology} \label{lch:iterated}

In this section we prove the following statement.

\begin{prop}\label{est_lch_prop}
Let $\alpha$ be a contact form on a manifold $N$ and $\gamma = (x,T=mT_0)$ be an isolated $\alpha$-Reeb orbit with multiplicity $m$ and minimal period $T_0>0$. Let $\Sigma \subset N$ be an embedded hypersurface transverse to $\gamma$ at $p_0 = x(0)$, so that the local first return map $\psi:(U,p_0) \to (\Sigma,p_0)$ is well-defined on a small neighborhood $U$ of $p_0$ in $\Sigma$. Then $\dim HC_*(\alpha,\gamma) \leq \dim HF_*(\psi^m,p_0)$, for every $* \in \Z$.
\end{prop}

The gradings in $HC_*(\alpha,\gamma)$ and in $HF_*(\psi^m,p_0)$ are given by the Conley-Zehnder indices computed with respect to homotopy classes of symplectic trivializations induced by a common homotopy class of $d\alpha$-symplectic trivializations of $\xi = \ker\alpha$ along $\gamma$, which we fix from now on.

\subsection{Geometric set-up and notation}

Let $n$ be defined by $\dim N = 2n-1$, denote the Reeb vector field of $\alpha$ by $R$ and fix $J \in \J(\alpha)$. Let $K \simeq \R/\Z\times \overline B$ be an isolating neighborhood for $\gamma$ equipped with coordinates $(t,z)$, $z=(q_1,\dots,p_1,\dots)$, such that $x(t) = (t/T_0,0)$, $\alpha$ coincides with $dt$ on $\R/\Z\times 0$, $d\alpha|_\xi$ coincides with $\omega_0 = \sum_i dq_i\wedge dp_i$ along $\R/\Z\times 0$, and $\inf_K i_Rdt > 0$. Here $B\subset \R^{2n-2}$ is a ball centered at the origin. This choice of coordinates induces a $d\alpha$-symplectic trivialization of $\xi$ along $\gamma$, which is assumed to be in the homotopy class previously chosen.

Consider a small nondegenerate perturbation $\alpha'$ of $\alpha$ on $K$, and $J' \in \J(\alpha')$ a small perturbation of $J$ which is regular for the data $(\alpha',K,\gamma)$ as explained in \S~\ref{local_complex}. We denote by $\P$ the set of closed  $\alpha'$-Reeb orbits in $K$ homotopic to $\gamma$, and by $\P_0\subset \P$ those which are good. Let $C_* = C_*(\alpha',K,\gamma)$ be the $\Q$-vector space freely generated by $\P_0$ graded by $|\cdot|=\mu_{CZ}+n-3$. Then $J'$ can be used to define a differential $d$ on $C_*$ and, by Corollary~\ref{local_independence}, if $(\alpha',J')$ is sufficiently close to $(\alpha,J)$ the homology of $(C_*,d)$ is the local contact homology $HC(\alpha,\gamma)$.

The natural $m:1$ covering
\begin{equation}\label{cov_proj}
\Pi : \widetilde K := \R/m\Z \times \overline B \to K = \R/\Z\times\overline B
\end{equation}
can be used to lift all the geometric data. $\Pi^*\alpha$ is a contact form on $\widetilde K$ and $\Pi^{-1}\gamma$ is an isolated $\Pi^*\alpha$-Reeb orbit, $\widetilde K$ is an isolating neighborhood for $\Pi^{-1}\gamma$, $\Pi^*\alpha'$ is a small nondegenerate perturbation of $\Pi^*\alpha$ and $(id_\R\times \Pi)^*J' \in \J(\Pi^*\alpha')$ is regular for the data $(\Pi^*\alpha',\widetilde K,\Pi^{-1}\gamma)$ and close to $(id_\R\times \Pi)^*J$. The covering group $\Z_m = \Z/m\Z$ of $\Pi$ acts on $\widetilde K$ with generator
\begin{equation}\label{Zm_action_space}
\begin{array}{ccc}
\sigma : \widetilde K \to \widetilde K, & & (t,z) \mapsto (t+1,z).
\end{array}
\end{equation}
The data $(\Pi^*\alpha',(id\times \Pi)^*J')$ is invariant under this action. The lifts of closed $\alpha'$-Reeb orbits homotopic to $\gamma$ are precisely the closed $\Pi^*\alpha'$-orbits which go once around the tube $\widetilde K$ and, consequently, they are all good. Moreover, their Conley-Zehnder indices coincide with the Conley-Zehnder indices of their projections.

Let $\widetilde\P$ be the set of closed $\Pi^*\alpha'$-Reeb orbits in $\widetilde K$ going once around the tube, which coincides precisely with the set of lifts of orbits in $\P$. The elements of $\widetilde P$ freely generate a $\Q$-vector space $\widetilde C_*$ graded by the Conley-Zehnder indices. Since $J'$ is assumed very close to $J$ in the $C^\infty$-strong topoloy, $(id\times \Pi)^*J'$ determines in the standard way described in Section~\ref{lch_section} a differential $\widetilde d$ on $\widetilde C^*$. According to Proposition~\ref{comp_prop}, the homology of $(\widetilde C_*,\widetilde d)$ coincides with the local Floer homology $HF_*(\psi^m,p)$.

Orbits in $\P$ have possibly many lifts to $\widetilde \P$, and the natural projection is still denoted $\Pi:\widetilde \P \to \P$. The generator $\sigma$ of the $\Z_m$-action~\eqref{Zm_action_space} induces an obvious action on $\widetilde \P$, and we choose a preferred lift for every element of $\P$. Our notation will be the following: if we write $\bar\varphi$ to denote an element in $\P$ then the chosen preferred lift is $\varphi$. Every orbit $\bar \varphi \in \P$ comes with a marked point $\pt_{\bar\varphi}$ assumed to lie on $0\times \overline B$. Its multiplicity $m_{\bar\varphi}$ divides $m$ and $\bar\varphi$ has precisely $p = m/m_{\bar\varphi}$ lifts which are orbits in $$ \mathcal O_{\bar\varphi} := \Pi^{-1}(\bar\varphi) = \{\varphi,\sigma\varphi,\dots,\sigma^{p-1}\varphi\}. $$ Note that $\sigma^{i+p}\varphi = \sigma^i\varphi$, $\forall i$. The marked point $\pt_{\varphi}$ is chosen in $0\times \overline B$ and we set $\pt_{\sigma^j\varphi} = \sigma^j(\pt_{\varphi})$, so that $\Pi(\pt_{\sigma^j\varphi}) = \pt_{\bar\varphi}$, for $j=0,\dots,p-1$. The elements of $\mathcal O_\varphi$ are simultaneously called good/bad if $\bar\varphi$ is good/bad. This terminology might be troublesome since all elements of $\widetilde \P$ are SFT-good (all such orbits are simple), but we will proceed without fear of ambiguity. The map $\Pi:\widetilde \P \to \P$ induces a linear map 
\begin{equation}\label{proj_chain_map}
\Pi_* : \widetilde C_* \to C_*
\end{equation}
by setting $\Pi_* = \Pi$ on good generators and $\Pi_* = 0$ on bad generators. Finally set
\begin{equation}
\delta_{\bar\varphi} = \left\{ \begin{aligned} & +1 \text{ if } \bar\varphi \text{ is good,} \\ & -1 \text{ if } \bar\varphi \text{ is bad.} \end{aligned} \right.
\end{equation}

\subsection{Finite-energy cylinders and their lifts}\label{cyl_lifts}

Given $\eta,\zeta \in \P$ we denote by $\M(\eta,\zeta)$ the moduli spaces of finite-energy $J'$-holomorphic cylinders in $\R\times K$ with a positive and a negative puncture, asymptotic to $\eta$ at its positive puncture and to $\zeta$ at its negative puncture, with asymptotic markers. Namely, an element is an equivalence class of triples $(t^+,t^-,F)$, where $t^\pm \in S^1$, $F = (a,f) : \R\times S^1 \to \R\times K$ is a non-constant finite-energy $J'$-holomorphic map with a positive puncture at $+\infty\times S^1$ where it is asymptotic to $\eta$, with a negative puncture at $-\infty\times S^1$ where it is asymptotic to $\zeta$, and satisfying $\lim_{s\to+\infty}f(s,t^+)=\pt_\eta$, $\lim_{s\to-\infty}f(s,t^-) \to \pt_\zeta$. The triple $(\theta^+,\theta^-,G)$ is equivalent to $(t^+,t^-,F)$ if one finds $c,\Delta s \in \R$ and $\Delta t \in S^1$ satisfying $F(s,t) = \tau_c \circ G(s+\Delta s, t+\Delta t)$ and $t^\pm +\Delta t = \theta^\pm$. Differently from the notation in Section~\ref{lch_section}, here we do quotient out by the $\R$-action on the target. The equivalence class of $(t^+,t^-,F)$ is denoted $[t^+,t^-,F]$. Moduli spaces $\M_0(\eta,\zeta)$ of cylinders in $\R\times K$ without asymptotic markers are defined as a set of equivalence classes of maps as above, where two maps $F,G$ are equivalent if there exist $c,\Delta s\in \R$ and $\Delta t \in S^1$ such that $F(s,t) = \tau_c\circ G(s+\Delta s,t+\Delta t)$. The class of such $F$ is denoted by $[F]$, and there is a natural surjective map 
\begin{equation}\label{forgetful_map}
\begin{array}{ccc} \Delta : \M(\eta,\zeta) \to \M_0(\eta,\zeta), & & [t^+,t^-,F] \mapsto [F] \end{array}.
\end{equation}

Let $F$ represent a given class $[F] \in \M_0(\eta,\zeta)$ where $\eta\neq\varphi$. Then the group ${\rm Iso}(F)$ of holomorphic self-diffeomorphisms $h$ of $\R\times S^1$ fixing the ends $\pm\infty\times S^1$ and satisfying $F\circ h = F$ can be identified with a subgroup of $S^1$ since such $h$ must have the form $h(s,t) = (s,t+\Delta t)$. Although ${\rm Iso}(F)$ depends on the representative $F$, its order $w[F] = \#{\rm Iso}(F)$ depends only on $[F]$. The choice of $F$ determines a subset of $S^1$ with $m_\eta$ elements, which are possible locations of asymptotic markers at the positive puncture. The group ${\rm Iso}(F)$ acts freely on this set and, consequently, $w[F]$ divides $m_\eta$. Analogously, $w[F]$ divides $m_\zeta$. Note that $\Z_{m_\eta}$ and $\Z_{m_\zeta}$ act on $\M(\eta,\zeta)$ by rotation of asymptotic markers, the generators are: $[t^+,t^-,F] \mapsto [t^+ + 1/m_{\eta},t^-,F]$ and $[t^+,t^-,F] \mapsto [t^+,t^- + 1/m_{\zeta},F]$. The set $\Delta^{-1}[F]$ has precisely $m_{\eta}m_{\zeta}/w[F]$ elements, $\forall [F] \in \M_0(\eta,\zeta)$. Note that both $m_\eta$ and $m_\varphi$ divide $\#\Delta^{-1}[F]$ since $w[F]$ is a common divisor of $m_\eta$ and $m_\varphi$.

There is a coherent system of orientations of the spaces $\M(\eta,\zeta)$, for all choices $\eta,\zeta \in \P$, compatible with glueing\footnote{The reader should note that, in view of Lemma~\ref{lemma_nearby_cylinders}, if $F=(a,f)$ is a cylinder representing an element of $\M(\eta,\zeta)$ for $\eta,\zeta\in\P$ then $\overline{f(\R\times S^1)} \subset {\rm int}(K)$ because $J'$ is assumed very close to some $J\in\J(\alpha)$. Hence, assuming regularity, such cylinders can be glued to obtain cylinders which again project into a compact subset of ${\rm int}(F)$.}. These are defined as in~\cite{be} even when $\eta$ or $\zeta$ is bad, and determine signs $\epsilon[t^+,t^-,F] = \pm 1$ when $|\eta|-|\zeta| = 1$. One has
\begin{equation}
\begin{aligned} & \epsilon[t^++1/m_{\eta},t^-,F] = \delta_{\eta} \epsilon[t^+,t^-,F] \\ & \epsilon[t^+,t^-+1/m_{\zeta},F] = \delta_{\zeta} \epsilon[t^+,t^-,F]. \end{aligned}
\end{equation}
In other words, the action of $\Z_{m_{\eta}}$ in $\M(\eta,\zeta)$ by rotating the asymptotic marker at the positive puncture is orientation preserving/reversing when $\eta$ is good/bad. The analogous statement holds for the action of $\Z_{m_{\zeta}}$ by rotations of the asymptotic marker at the negative puncture. Hence this signs descend to signs $\epsilon[F]$ on $\M_0(\eta,\zeta)$ only when both $\eta$ and $\zeta$ are good.

Moduli spaces of finite-energy $(id\times \Pi)^*J'$-holomorphic cylinders in $\R\times \widetilde K$ asymptotic to orbits in $\widetilde \P$ with or without marked points are defined in the same way. However, note that all such cylinders are somewhere injective and there are no non-trivial reparametrization groups.

Choose any $\bar\varphi,\bar\eta \in \P$ and set $p = m/m_{\bar\varphi}$, $q=m/m_{\bar\eta}$. There are well-defined projections 
\begin{equation}\label{proj_cyls}
\begin{array}{ccc} \Pi_*:\M(\sigma^i\varphi,\sigma^j\eta) \to \M(\bar\varphi,\bar\eta) & & [t^+,t^-,F] \mapsto [t^+,t^-,(id_\R\times\Pi)\circ F] \end{array}
\end{equation}
where $i=0,\dots,p-1$, $j=0,\dots,q-1$. For any $\zeta \in \widetilde \P$ denote $$ \M(\mathcal O_{\bar\varphi},\zeta) = \bigcup_{i=0}^{p-1} \M(\sigma^i\varphi,\zeta) \ \text{ and } \ \M(\zeta,\mathcal O_{\bar\eta}) = \bigcup_{j=0}^{q-1} \M(\zeta,\sigma^j\eta). $$ We define the space $\M(\mathcal O_{\bar\varphi},\mathcal O_{\bar\eta})$ analogously.

Any finite-energy $J'$-holomorphic cylinder $F=(a,f)$ representing an element in $\M_0(\bar\varphi,\bar\eta)$ can be lifted to (possibly many) finite-energy $(id\times \Pi)^*J'$-holomorphic cylinders, since the loops $t\mapsto f(s,t)$ go $m$ times around the tube $K$ and the projection~\eqref{cov_proj} is pseudo-holomorphic with respect to $J'$ and $(id\times \Pi)^*J'$. To be more precise, recall the forgetful map $\Delta$~\eqref{forgetful_map}, and for a fixed $[F]\in \M_0(\bar\varphi,\bar\eta)$ consider the bijection 
\begin{equation}
\begin{aligned} \left\{ t_0^+ + \frac{k}{m_{\bar\varphi}} : k = 0,\dots,\frac{m_{\bar\varphi}}{w[F]}-1 \right\} &\times \left\{ t_0^- + \frac{k}{m_{\bar\eta}} : k = 0,\dots,m_{\bar\eta}-1 \right\} \to \Delta^{-1}[F] \\ (t^+,t^-) &\mapsto [t^+,t^-,F] \end{aligned}
\end{equation}
For each fixed $i\in \{0,\dots,p-1\}$, a given choice of asymptotic marker $t^+$ at $+\infty\times S^1$ uniquely determines a lift $\widetilde F=(\widetilde a,\widetilde f)$ of $F$ to $\R\times \widetilde K$ asymptotic to the orbit $\sigma^i\varphi$ at the positive puncture and satisfying $\widetilde f(s,t^+) \to \pt_{\sigma^i\varphi}$ as $s\to+\infty$. After this is done there is no control at the negative puncture: the asymptotic limit $\sigma^j\eta$ is forced on us, together with the unique location of the asymptotic marker $t^-$ which satisfies $\widetilde f(s,t^-) \to \pt_{\sigma^j\eta}$ as $s\to-\infty$. One must have
\[
\lim_{s\to+\infty} \widetilde f\left(s,t^+ + \frac{k}{m_{\bar\varphi}}\right) = \sigma^{kp}(\pt_{\sigma^i\varphi}), \ \forall k\in\Z.
\]
Let us agree to say that $[t^+,t^-,F] \in \M(\bar\varphi,\bar\eta)$ lifts to $\M(\sigma^i\varphi,\mathcal O_{\bar\eta})$ when the unique lift $\widetilde F=(\widetilde a,\widetilde f)$ of $F$ satisfying $\lim_{s\to+\infty} \widetilde f(s,t^+) = \pt_{\sigma^i\varphi}$ also satisfies $\lim_{} \widetilde f(s,t^-) = \pt_{\sigma^j\eta}$ for the uniquely determined $\sigma^j\eta \in \mathcal O_{\bar\eta}$ that $\widetilde F$ is asymptotic to at the negative puncture. Hence, out of the $m_{\bar\varphi}m_{\bar\eta}/w[F]$ elements of $\Delta^{-1}[F]\subset \M(\bar\varphi,\bar\eta)$ only $m_{\bar\varphi}/w[F]$ of them lift to $\M(\sigma^i\varphi,\mathcal O_{\bar\eta})$, for any fixed choice of $i$. Let us denote by $\M^F$ the subset of $\M(\mathcal O_{\bar\varphi},\mathcal O_{\bar\eta})$ consisting of the $[t^+,t^-,\widetilde F]$ obtained lifting $F$ as above. We concluded that
\begin{equation}\label{exp_card}
\# \left( \M^F \cap \M(\sigma^i\varphi,\mathcal O_{\bar\eta}) \right) = \frac{m_{\bar\varphi}}{w[F]}, \ \forall i \in \{0,\dots,p-1\}.
\end{equation}
Obviously, all cylinders in $\M(\mathcal O_{\bar\varphi},\mathcal O_{\bar\eta})$ are obtained by this lifting procedure from some cylinder in $\M(\bar\varphi,\bar\eta)$.

The projection $\Pi$ can be used to pull-back the system of coherent orientations of moduli spaces of curves in $\R\times K$ to a system of coherent orientations on moduli spaces of curves in $\R\times \widetilde K$: $$ \epsilon[t^+,t^-,F] = \epsilon[t^+,t^-,(id_\R\times \Pi)\circ F]. $$ These are clearly compatible with glueing of curves on $\R\times \widetilde K$. The generator $\sigma$ of the covering group determines a bijection (again denoted by $\sigma$):
\begin{equation}\label{action_cyl}
\begin{aligned}
\sigma : \M(\sigma^i\varphi,\sigma^j\eta) &\to \M(\sigma^{i+1}\varphi,\sigma^{j+1}\eta) \\ [t^+,t^-,F] &\mapsto \left[t^+ - \frac{\delta_+(i)}{m_{\bar\varphi}},t^- -\frac{\delta_-(j)}{m_{\bar\eta}},(id_\R\times \sigma)\circ F \right] 
\end{aligned}
\end{equation}
where
\begin{equation}\label{delta_+-}
\begin{array}{ccc}
\delta_+(i) = \left\{ \begin{aligned} & 1 \text{ if } i+1 = p \\ & 0 \text{ if } i+1 < p \end{aligned} \right. & \text{and} & \delta_-(j) = \left\{ \begin{aligned} & 1 \text{ if } j+1 = q \\ & 0 \text{ if } j+1 < q \end{aligned} \right. .
\end{array}
\end{equation}
Here $(i,j) \in \{0,\dots,p-1\}\times\{0,\dots,q-1\}$. It follows that
\begin{equation}\label{sign_action_cyl}
\epsilon(\sigma[t^+,t^-,F]) = (\delta_{\bar\varphi})^{\delta_+(i)} (\delta_{\bar\eta})^{\delta_-(j)} \epsilon[t^+,t^-,F]
\end{equation}
for all $[t^+,t^-,F] \in \M(\sigma^i\varphi,\sigma^j\eta)$.

\begin{remark}\label{orbit_rmk}
Let $[F] \in \M_0(\bar\varphi,\bar\eta)$ be fixed, and let $[t^+,t^-,\widetilde F] \in \M(\sigma^i\varphi,\mathcal O_{\bar\eta})$ satisfying $(id_\R\times\Pi) \circ\widetilde F = F$ be fixed. Then
\[
\begin{array}{cc}
\M(\sigma^i\varphi,\mathcal O_{\bar\eta}) \cap \M^F = \{ \sigma^{kp}[t^+,t^-,\widetilde F] : k\geq 0 \}, & \M^F = \{ \sigma^{k}[t^+,t^-,\widetilde F] : k\geq 0 \}
\end{array}
\]
for all $i\in\{0,\dots,p-1\}$.
\end{remark}

\subsection{A $\Z_m$-action on $( \widetilde C_*,\widetilde d )$ by chain maps}\label{action_chain_maps_section}

Since $m_{\bar\varphi}$ is even when $\bar\varphi$ is bad, we can consider a linear $\Z_m$-action on $\widetilde C_*$ with generator $E$ defined by
\begin{equation}\label{map_E}
E(\varphi) = \sigma\varphi,\ E(\sigma\varphi) = \sigma^2\varphi, \ \dots, \ E(\sigma^{p-1}\varphi) = \delta_{\bar\varphi}\varphi
\end{equation}
on the generators of $\widetilde C_*$. Our goal here is to show

\begin{lemma}\label{action_chain_maps}
The map $E$ induces a $\Z_m$-action on $\widetilde C_*$ by chain maps.
\end{lemma}

The lemma can be restated by saying that
\begin{equation}\label{dE=Ed}
E\widetilde d = \widetilde d E
\end{equation}
so that we need to understand the differential $\widetilde d$ qualitatively. Of course, it suffices to prove~\eqref{dE=Ed} on the generators $\widetilde \P$.

Fix $\bar\varphi,\bar\eta \in \P$ satisfying $|\bar\varphi|-|\bar\eta|=1$, and denote $p = m/m_{\bar\varphi}$ and $q=m/m_{\bar\eta}$. Each cylinder $[F] \in \M_0(\bar\varphi,\bar\eta)$ reveals a distinct set 
\begin{equation}\label{coef_ij}
\widetilde d^F_{ij} \in \Z, \ \ \ i\in\{0,\dots,p-1\},\ j\in\{0,\dots,q-1\}
\end{equation}
of coefficients which, loosely speaking, is the contribution of the lifts of $F$ to cylinders in $\R\times \widetilde K$ (with all possible choices of asymptotic markers) connecting $\sigma^i\varphi$ to $\sigma^j\eta$ to the differential $\widetilde d$. To be more precise, recall the set $\M^F \subset \M(\mathcal O_{\bar\varphi},\mathcal O_{\bar\eta})$ discussed in~\ref{cyl_lifts} obtained by the lifts of $F$. We write
\begin{equation}
\M^F_{ij} = \M^F \cap \M(\sigma^i\varphi,\sigma^j\eta).
\end{equation}
The coefficients~\eqref{coef_ij} are defined as
\begin{equation}\label{exp_d_ij}
\widetilde d^F_{ij} = \sum_{[t^+,t^-,\widetilde F] \in \M^F_{ij}} \epsilon[t^+,t^-,\widetilde F].
\end{equation}
We get the formula
\begin{equation}\label{exp_tilde_d}
\widetilde d\sigma^i\varphi = \sum_{\{\bar\eta \in \P :|\bar\eta| = |\bar\varphi|-1\}} \sum_{[F]\in\M_0(\bar\varphi,\bar\eta)} \sum_{j=0}^{q-1} \widetilde d^F_{ij} \sigma^j\eta
\end{equation}
which implies
\begin{equation}\label{formula_d_qualit_upstairs}
\left< \widetilde d\sigma^i\varphi, \sigma^j\eta \right> = \sum_{[F]\in\M_0(\bar\varphi,\bar\eta)} \widetilde d^F_{ij}
\end{equation}
for all $(i,j) \in \{0,\dots,p-1\} \times \{0,\dots,q-1\}$.

From now on we view the indices $i,j$ as periodic: $i \in \Z_p$ and $j\in\Z_q$. Recall the functions $\delta_+:\Z_p \to \{0,1\}$, $\delta_-:\Z_q\to\{0,1\}$ from~\eqref{delta_+-}. With these agreements the map $E$~\eqref{map_E} acts on $\mathcal O_\varphi$ and $\mathcal O_\eta$ as $$ \begin{array}{ccc} E\sigma^i\varphi = (\delta_{\bar\varphi})^{\delta_+(i)} \sigma^{i+1}\varphi, & & E\sigma^j\eta = (\delta_{\bar\eta})^{\delta_-(j)} \sigma^{j+1}\eta. \end{array} $$ We have
\[
\left< E\widetilde d \sigma^i\varphi, \sigma^{j+1}\eta \right> = \sum_{[F]\in\M_0(\bar\varphi,\bar\eta)} \widetilde d^F_{ij} (\delta_{\bar\eta})^{\delta_-(j)}
\]
and
\[
\left< \widetilde d E\sigma^i\varphi, \sigma^{j+1}\eta \right> = \sum_{[F]\in\M_0(\bar\varphi,\bar\eta)} \widetilde d^F_{(i+1)(j+1)} (\delta_{\bar\varphi})^{\delta_+(i)}
\]
so to prove~\eqref{dE=Ed} it suffices to show that for any $[F] \in\M_0(\bar\varphi,\bar\eta)$ the following identity holds
\begin{equation}\label{miracle1}
\widetilde d^F_{ij} (\delta_{\bar\eta})^{\delta_-(j)} = \widetilde d^F_{(i+1)(j+1)} (\delta_{\bar\varphi})^{\delta_+(i)}.
\end{equation}
In fact, the map~\eqref{action_cyl} maps $\M^F_{ij}$ bijectively onto $\M^F_{(i+1)(j+1)}$. Thus
\[
\begin{aligned}
\widetilde d^F_{(i+1)(j+1)} & = \sum_{[\theta^+,\theta^-,G] \in \M^F_{(i+1)(j+1)}} \epsilon[\theta^+,\theta^-,G] \\
&= \sum_{[t^+,t^-,\widetilde F] \in \M^F_{ij}} \epsilon(\sigma[t^+,t^-,\widetilde F]) \\
&= \sum_{[t^+,t^-,\widetilde F] \in \M^F_{ij}} (\delta_{\bar\varphi})^{\delta_+(i)} (\delta_{\bar\eta})^{\delta_-(j)} \epsilon[t^+,t^-,\widetilde F] \\
&= (\delta_{\bar\varphi})^{\delta_+(i)} (\delta_{\bar\eta})^{\delta_-(j)} \widetilde d^F_{ij}
\end{aligned}
\]
which is another way of writing~\eqref{miracle1}. In the third equality we used~\eqref{sign_action_cyl}. This concludes the proof of Lemma~\ref{action_chain_maps}.

\subsection{$\Pi_*$ is a chain map}

\begin{lemma}\label{lemma_pi_chain_map}
The map $\Pi_*$ in~\eqref{proj_chain_map} satisfies $\Pi_* \widetilde d = d\Pi_*$.
\end{lemma}

To prove the above statement we first fix arbitrary orbits $\bar\varphi,\bar\eta \in \P$, denote $p = m/m_{\bar\varphi}$, $q=m/m_{\bar\eta}$ and split the argument in two cases.

\subsubsection{Case 1: $\bar\varphi,\bar\eta$ are good.}

Then for any $i\in\{0,\dots,p-1\}$ we have
\begin{equation}\label{good/good}
\begin{aligned}
\left< \Pi_*\widetilde d\sigma^i\varphi,\bar\eta \right> &= \left< \Pi_* \left( \sum_{[F]\in\M_0(\bar\varphi,\bar\eta)} \sum_{j=0}^{q-1} \widetilde d^F_{ij} \sigma^j\eta \right),\bar\eta \right> \\
&= \sum_{[F]\in\M_0(\bar\varphi,\bar\eta)} \sum_{j=0}^{q-1} \widetilde d^F_{ij} \\
&= \sum_{[F]\in\M_0(\bar\varphi,\bar\eta)} \sum_{j=0}^{q-1} \sum_{[t^+,t^-,\widetilde F] \in \M^F_{ij}} \epsilon[t^+,t^-,\widetilde F] \\
&= \sum_{[F]\in\M_0(\bar\varphi,\bar\eta)} \left( \sum_{[t^+,t^-,\widetilde F] \in \M^F \cap \M(\sigma^i\varphi,\mathcal O_{\bar\eta})} \epsilon[F] \right) \\
&= m_{\bar\varphi} \sum_{[F]\in\M_0(\bar\varphi,\bar\eta)} \frac{\epsilon[F]}{w[F]} \\
&= \frac{1}{m_{\bar\eta}} \sum_{[F]\in\M_0(\bar\varphi,\bar\eta)} \epsilon[F] \#\Delta^{-1}[F] \\
&= \frac{1}{m_{\bar\eta}} \sum_{[t^+,t^-,F]\in\M(\bar\varphi,\bar\eta)} \epsilon[t^+,t^-,F] \\
&= \left< d\bar\varphi,\bar\eta\right> = \left< d\Pi_*\sigma^i\varphi,\bar\eta\right>.
\end{aligned}
\end{equation}
In the second equality we used~\eqref{exp_tilde_d}, in the third equality we used~\eqref{exp_d_ij}, in the fourth equality we used that $\bar\varphi,\bar\eta$ are good, in the fifth equality we used~\eqref{exp_card}, in the seventh equality we used that $\bar\varphi,\bar\eta$ are good and that $\Delta$ is surjective.

\subsubsection{Case 2: $\bar\varphi$ is bad, $\bar\eta$ is good.}

Fix $i_0 \in \{0,\dots,p-1\}$. In this case clearly $$ \left< d\Pi_*\sigma^{i_0}\varphi,\bar\eta \right> = 0 $$ by the definition of $\Pi_*$ (since $\bar\varphi$ is bad). So the work reduces to showing 
\begin{equation}\label{goal_bad/good}
\left< \Pi_*\widetilde d\sigma^{i_0}\varphi,\bar\eta \right> = 0.
\end{equation}
For any $F$ representing some $[F] \in \M_0(\bar\varphi,\bar\eta)$ we have the $p\times q$ matrix of coefficients $\widetilde d^F = (\widetilde d^F_{ij})$ and, according to~\eqref{formula_d_qualit_upstairs},
\begin{equation}\label{ident_lines}
\left< \Pi_*\widetilde d\sigma^{i_0}\varphi,\bar\eta \right> = \sum_{[F]\in\M_0(\bar\varphi,\bar\eta)} \sum_{j=0}^{q-1} \widetilde d^F_{i_0j}
\end{equation}
is the sum over all possible such matrices of the sum of the elements of the $i_0$-th line. Fixing $F$, let $\Phi_0 = [t^+,t^-,\widetilde F] \in \M^F_{i_0j_0}$ be some reference element (as explained before, the lift $\widetilde F = (\widetilde a,\widetilde f)$ of $F$ is uniquely determined by asking $\lim_{s\to+\infty} \widetilde f(s,t^+) = \pt_{\sigma^{i_0}\varphi}$, the value of $j_0$ is forced on us). Recalling the action $\sigma$~\eqref{action_cyl} we have, by Remark~\ref{orbit_rmk}, that $\M^F = \{\sigma^k\Phi_0:k\in\Z\}$. From now we consider the variables $i$ and $j$ as periodic: $i\in \Z_p$, $j\in\Z_q$. Analogously, the indices of the matrix $\widetilde d^F$ will also be seen as periodic. Note that $\sigma^k\Phi_0 \in \M(\sigma^{i_0+k}\varphi,\sigma^{j_0+k}\eta)$, and each element $\sigma^k\Phi_0 \in \M^F$ contributes with $\epsilon(\sigma^k\Phi_0) = \pm1$ to the coefficient $d^F_{(i_0+k)(j_0+k)}$. Obviously
\[
\min\{k\geq 1 : \sigma^k\Phi_0 \in \M^F_{i_0j_0}\} = {\rm lcm}(p,q).
\]
However the set $\M^F$ might be larger than $\{\Phi_0,\sigma\Phi_0,\dots,\sigma^{{\rm lcm}(p,q)-1}\Phi_0\}$ since $\Phi_0$ need not be equal to $\sigma^{{\rm lcm}(p,q)}\Phi_0$. In fact, we have
\[
\M^F = \{\Phi_0,\dots,\sigma^{x{\rm lcm}(p,q)-1}\Phi_0\}
\]
where 
\[
x = \min \{k\in\{1,2,\dots\} \mid \sigma^{k{\rm lcm}(p,q)}\Phi_0 = \Phi_0 \}.
\]
By~\eqref{action_cyl}, each walk $$ \{\sigma^{k{\rm lcm}(p,q)}\Phi_0,\sigma^{k{\rm lcm}(p,q)+1}\Phi_0,\dots,\sigma^{(k+1){\rm lcm}(p,q)-1}\Phi_0\} $$ 
corresponds to ${\rm lcm}(p,q)/p$ rotations at the positive puncture, and to ${\rm lcm}(p,q)/q$ rotations of the negative puncture. In view of the definition of $x$ we have
\begin{equation}
[t^+,t^-,\widetilde F] = \left[ t^+ - \frac{x{\rm lcm}(p,q)}{pm_{\bar\varphi}}, t^- - \frac{x{\rm lcm}(p,q)}{qm_{\bar\eta}}, (id_\R\times \sigma^{x{\rm lcm}(p,q)})\circ \widetilde F \right]
\end{equation}
so after applying the projection~\eqref{proj_cyls} we conclude that
\[
[t^+,t^-,F] = \left[ t^+ - \frac{x{\rm lcm}(p,q)}{pm_{\bar\varphi}}, t^- - \frac{x{\rm lcm}(p,q)}{qm_{\bar\eta}}, F \right].
\]
Thus, $x$ can be computed by 
\[
x = \min \left\{ k\in\{1,2,\dots\} \text{ such that } \frac{x{\rm lcm}(p,q)}{pm_{\bar\varphi}} = \frac{x{\rm lcm}(p,q)}{qm_{\bar\eta}} \in {\rm Iso}(F) \right\}
\]
or, alternatively, by saying that $x$ is the minimal positive integer for which $\exists y\in\{1,2,\dots\}$ such that 
\[
\left\{ \begin{aligned} & x\frac{{\rm lcm}(p,q)}{p} = y \frac{m_{\bar\varphi}}{w[F]} \\ & x\frac{{\rm lcm}(p,q)}{q} = y \frac{m_{\bar\eta}}{w[F]} \end{aligned} \right. .
\]
Substituting $y=1$ above we would obtain $$ x = \frac{m}{{\rm lcm}(p,q)w[F]} $$ since $pm_{\bar\varphi} = qm_{\bar\eta} = m$. Note that $x$ given by this formula is an integer since $m/w[F]$ is a common multiple of $p$ and $q$ because $w[F]$ is a common divisor of $m_{\bar\varphi}$ and $m_{\bar\eta}$. Hence $x$ is actually given by this formula.

Observe that the path $\{\Phi_0,\dots,\sigma^{x{\rm lcm}(p,q)-1}\Phi_0\}$ visits the space $\M(\sigma^{i_0}\varphi,\mathcal O_{\bar\eta})$ exactly $x{\rm lcm}(p,q)/p = m_{\bar\varphi}/w[F]$ times, and each visit contributes with alternating signs to the $i_0$-th line of $\widetilde d^F$ because $\bar\varphi$ is bad. This follows from the formula~\eqref{action_cyl} for the action $\sigma$. Thus, in order to prove 
\begin{equation}\label{line_bad/good}
\sum_{j=0}^{q-1} \widetilde d^F_{i_0j} = 0
\end{equation}
it suffices to show that $m_{\bar\varphi}/w[F]$ is even. This follows easily from the fact that the $\Z_{m_{\bar\varphi}}$-action on $\M(\bar\varphi,\bar\eta)$ given by rotations of asymptotic markers at the positive puncture is orientation reversing. Thus~\eqref{goal_bad/good} follows from~\eqref{ident_lines} and~\eqref{line_bad/good}.

This concludes {\it Case 2} and the proof of Lemma~\ref{lemma_pi_chain_map}.

\subsection{Conclusion of the proof of Proposition~\ref{est_lch_prop}}

Consider the average operator
\begin{equation}
\begin{array}{ccc} A : \widetilde C_* \to \widetilde C_* & & A = \frac{1}{m}(I+E+\dots+E^{m-1}). \end{array}
\end{equation}
Since $E^m-I=0$ we have a decomposition
\[
\widetilde C_* = \ker (E-I) \oplus \ker A = {\rm im}A \oplus \ker A.
\]
By Lemma~\eqref{action_chain_maps} we have a subcomplex $({\rm im} A,\widetilde d) \subset (\widetilde C_*,\widetilde d)$. Let $Q:C_* \to \widetilde C_*$ be defined on generators $\bar\varphi \in \P_0$ by
\[
Q\bar\varphi = A(\text{some element in }\mathcal O_{\bar\varphi}).
\]
Then clearly $Q$ is injective and
\[
Q\Pi_* = A.
\]
Here we used that $A(\sigma^i\varphi) = 0$ whenever $\bar\varphi$ is a bad orbit, and that $C_*$ is generated by the good orbits. It follows this and the injectivity of $Q$ that $$ \ker \Pi_* = \ker A. $$ Here injectivity of $Q$ was used. Thus we get that
\[
\Pi_*: ({\rm im}A,\widetilde d) \to (C_*,d)
\]
is a linear isomorphism and a chain map. Consequently, the homology of $(C_*,d)$ is equal to the homology of the subcomplex $({\rm im}A,\widetilde d)$. To conclude the proof of Proposition~\ref{est_lch_prop} we must finally show that the inclusion $({\rm im}A,\widetilde d)\hookrightarrow(\widetilde C,\widetilde d)$, which is obviously a chain map, induces an injective map on the level of homology. In fact, let $\lambda \in {\rm im}A$ be a closed chain, and assume that there exists $\beta \in \widetilde C_*$ satisfying $\widetilde d\beta=\lambda$. Applying the chain map $A$ to this equation we get $\lambda = A\lambda = A\widetilde d\beta = \widetilde d A\beta$ with $A\beta \in {\rm im}A$, as desired.

\subsection{Proof of Theorem~\ref{main}} \label{proof:main}

Let $\gamma = (x,T)$ be an isolated periodic orbit with multiplicity $m$. Let $\Sigma \subset N$ be an embedded hypersurface transverse to $\gamma$ at $\pt = x(0)$, so that the local first return map $\psi:(U,\pt) \to (\Sigma,\pt)$ is well-defined on a small neighborhood $U$ of $\pt$ in $\Sigma$.  Following \cite{GG}, we say that a positive integer $j$ is admissible for $\gamma$ if $\lambda^j \neq 1$ for all eigenvalues $\lambda \neq 1$ of $d\psi^{m}(\pt)$. It follows from Proposition~\ref{est_lch_prop} and \cite[Theorem 1.1]{GG} that the total rank of $HC_*(\alpha,\gamma^j)$ is less or equal than the total rank of $HC_*(\psi^m,\pt)$ for every admissible $j$.

Now, suppose that $\gamma$ is simple and every iterate of $\gamma$ is isolated. In order to prove Theorem~\ref{main} it remains only to show that we can write the set of natural numbers as a finite union of admissible integers of iterates of $\gamma$. This is the content of the lemma below which is extracted from \cite[Lemma 2]{GM}. For the reader's convenience, we will reproduce its proof.

\begin{lemma}
There are positive integers $m_1,...,m_s$ and sequences $j^i_k$ of natural numbers with $i \in \N$ and $k=1,...,s$ such that the numbers $j^i_km_k$ are mutually distinct, $\cup_{i,k} \{j^i_km_k\} = \N$ and $j^i_k$ is admissible for $\gamma^{m_k}$ for every $i \in \N$ and $k=1,...,s$.
\end{lemma}

\begin{proof}
Assume $d\psi(\pt)$ has eigenvalues of the form $e^{i2\pi r}$, $r\in\Q$. Write the rational eigenvalues in the circle of $d\psi(\pt)$ in the form $p/q$ with $p$ and $q$ relatively prime (to be more precise, the corresponding eigenvalue is $e^{i2\pi p/q}$), and denote by $Q$ the set of denominators of these eigenvalues. For $\emptyset \neq A \subset Q$ let $m(A)$ denote the least common multiple of all elements in $A$. Choose distinct numbers $m_1,...,m_s$ such that $\{m_1,...,m_s\} = \{m(A);\ \emptyset \neq A \subset Q\} \cup \{1\}$. For each $k \in \{1,...,s\}$ consider the set $Q_k = \{q\in Q;\ q \text{ does not divide } m_k\}$. We list the elements of the set $$ \{j\in\N;\ q \text{ does not divide } jm_k, \ \forall q\in Q_k\} $$ in strictly increasing order $j_{k,1}<j_{k,2}<\dots$. Let us prove that $j_{k,i}$ is admissible for $\gamma^{m_k}$, $\forall i$. The eigenvalues in the circle for $d\psi^{m_k}(\pt)$ are of the form $\lambda^{m_k}\simeq m_kp/q$, for some eigenvalue $\lambda\simeq p/q$ for $d\psi(\pt)$ which lies in the circle. The conclusion follows since $\lambda^{m_k} \neq 1$ is equivalent to the condition that $q$ does not divide $m_k$, and the condition $\lambda^{j_{k,i}m_k} \neq 1$ is equivalent to the condition that  $q$ does not divide $j^i_km_k$. It remains only to show that $\cup_{i,k} \{j_{k,i}m_k\} = \N$ but this is easy and left to the reader.
\end{proof}

\section{Morse inequalities} \label{section:mi}

Let $(N^{2n-1},\xi)$ be a closed co-oriented contact manifold such that $c_1(\xi)$ vanishes. Let $\alpha$ be a fixed contact form for $\xi$ inducing the given co-orientation, and fix a homotopy class of augmentations $[\ep]$. The action spectrum of $\alpha$ is 
\[
\Sigma(\alpha) = \{ A(\gamma); \gamma\text{ is a periodic orbit of }\alpha\}
\]
where $A(\gamma) = \int_\gamma \alpha$ is the action of $\gamma$. The main goal of this section is to prove Proposition~\ref{mi}.

\subsection{Filtered linearized contact homology} \label{flch}

We quickly review a few definitions, see~\cite{Bo} or the original~\cite{EGH} for more details.

Let $\alpha'$ be a nondegenerate defining contact form for $\xi$, and assume the existence of $J' \in \J(\alpha')$ generic enough in order to get a well-defined DGA $(\A(\alpha'),d')$ whose homology is the full contact homology of $\xi$ with $\Q$-coefficients. As explained in~\cite{EGH}, $\A(\alpha')$ is the supercommutative unital algebra generated by the good closed $\alpha'$-Reeb orbits, with $\Q$-coefficients, graded by $|\cdot| = \cz(\cdot) + n-3$, and $d'$ is defined by the (algebraic) count of rigid punctured finite-energy spheres with one positive puncture in $(\R\times N,J')$. We can use $\Q$-coefficients since we assume $c_1(\xi)$ vanishes.
 
Let $\epsilon'$ be an augmentation for $(\A(\alpha'),d')$, {\it i.e.}, it is an algebra homomorphism $\epsilon' : \A(\alpha') \to \Q$ satisfying $\epsilon'(1)=1$, $\epsilon' \circ d' = 0$, and $\epsilon'(\gamma)=0$ if $|\gamma|\not=0$. Now let $C(\alpha')$ be the $\Q$-vector space freely generated by the good closed $\alpha'$-Reeb orbits, graded by $|\cdot|$. The algebra $\A(\alpha')$ can be decomposed according to word length $\A(\alpha') = \A_0 \oplus \A_1 \oplus \A_2 \oplus \dots$ where $\A_0=\Q$. If $\pi_1$ is projection onto $\A_1$ and $S^{\epsilon'} : \A(\alpha') \to \A(\alpha')$ is the algebra homomorphism determined by $S^{\epsilon'}(1)=1$ and $S^{\epsilon'}(\gamma) = \gamma + \epsilon'(\gamma)$, then the linearized differential is defined by
\begin{equation*}
\begin{array}{ccc}
d'_{\epsilon'} : C_*(\alpha') \to C_{*-1}(\alpha') & & d'_{\epsilon'} = \pi_1 \circ S^{\epsilon'} \circ d'.
\end{array}
\end{equation*}
Then $(C(\alpha'),d'_{\epsilon'})$ is a chain complex and its homology is the so-called linearized contact homology $HC^{[\epsilon']}(\xi)$, which turns out to depend only on $\xi$ and on the homotopy class\footnote{The definition of homotopy class of augmentations for the contact structure was given in the introduction.} of $\epsilon'$.

Since $d'$ decreases action it is trivial to see that so does $d'_{\epsilon'}$. Thus, if for a given $a \in [0,+\infty) \setminus \Sigma(\alpha')$ we define $C^a(\alpha')$ to be the subspace generated by the orbits with action strictly less than $a$ then $(C^a(\alpha')/C^b(\alpha'),d'_{\epsilon'})$ is a chain complex when $b<a$ and $a,b \in [0,+\infty) \setminus \Sigma(\alpha')$. Its homology is, by definition, linearized contact homology of $\alpha'$ and $\epsilon'$ filtered by the interval $[b,a)$ and will be denoted by $HC^{[b,a),\epsilon'}(\alpha')$. When $b=0$ we may simply write $HC^{a,\epsilon'}(\alpha')$. These vector spaces depend on $J'$, but this will not be explicit in the notation.

\begin{remark}
In the case $b=0$ one could alternatively define $HC^{a,\epsilon'}(\alpha')$ as follows. Let $\A^a(\alpha')$ be the subalgebra generated by the good orbits with action strictly less than $a$. Then $d'(\A^a(\alpha')) \subset \A^a(\alpha')$ and $\epsilon'$ is also an augmentation for $(\A^a(\alpha'),d')$. Linearizing $d'|_{\A^a(\alpha')}$ we get again the chain complex $(C^a(\alpha'),d'_{\ep'})$.
\end{remark}

\begin{proposition}\label{lch-ch}
Suppose that $\alpha$ has finitely many simple periodic orbits. For any given $a \in \Sigma(\alpha)$ and $J \in \J(\alpha)$ there exists $\delta > 0$, a $C^\infty$-neighborhood $\V$ of $\alpha$ and a $C^\infty$-strong neighborhood $\U$ of $J$ in the space of almost complex structures on $\R \times N$ such that the following holds: if $\alpha' \in \V$ is a nondegenerate contact form defining $\xi$, $J' \in \J(\alpha') \cap \U$ is regular so that the data $(\alpha',J')$ defines a DGA $(\A(\alpha'),d')$ whose homology is the full contact homology of $\xi$ with $\Q$-coefficients, then $$ HC_*^{[a-\delta,a+\delta),\ep'}(\alpha') \simeq \oplus_{\{\gamma : A(\gamma) = a\}} HC_*(\alpha,\gamma) $$ for every augmentation $\ep'$ of $(\A(\alpha'),d')$.
\end{proposition}

\begin{proof}
First we find $\sigma>0$ and a $C^\infty$-neighborhood $\V_0$ of $\alpha$ such that $\Sigma(\alpha') \subset [2\sigma,+\infty)$ for every contact form $\alpha'\in \V_0$ defining $\xi$. Consider $\delta>0$ such that $\delta < \sigma$ and $\delta < {\rm dist}(\{a\},\Sigma(\alpha)\setminus\{a\})$.

Let $\gamma_1,\dots,\gamma_m$ be the (not necessarily prime) $\alpha$-Reeb closed orbits with $\alpha$-action equal to $a$ and let $K_i$ be a small smooth compact tubular neighborhood of $\gamma_i$, $\forall i=1,\dots,m$. Define $K = \cup_i K_i$. We find $\V_1 \subset \V_0$ such that if $\alpha' \in \V_1$ defines $\xi$ then $a\pm\delta \not\in \Sigma(\alpha')$, all closed $\alpha'$-Reeb orbits with $\alpha'$-action in $[a-\delta,a+\delta]$ lie in the interior of $K$ and, moreover, those lying in ${\rm int}(K_i)$ are homotopic to $\gamma_i$ in $K_i$. In fact, all the closed $\alpha'$-Reeb orbits with action in $[a-\delta,a+\delta]$ are $C^\infty$-close to one of the $\gamma_i$ when $\alpha'$ is a sufficiently $C^\infty$-small perturbation of $\alpha$.

By the results of Section~\ref{lch_section} we find $\V \subset \V_1$ and a $C^\infty$-strong neighborhood $\U$ of $J$ in the space of almost complex structures on $\R \times N$ such that if $\alpha' \in \V$ is nondegenerate and $J' \in \J(\alpha') \cap \U$ is regular then the data $(\alpha',J')$ defines chain complexes $(C_*(\alpha',K_i),d_i)$ whose homologies are the local contact homologies $HC(\alpha,\gamma_i)$, for every $i=1,\dots,m$. The $\Q$-vector space $C_*(\alpha',K_i)$ is freely generated by the good closed $\alpha'$-Reeb orbits in $K_i$ which satisfy $\mu_{CZ} + n-3 = *$ and are homotopic to $\gamma_i$ in $K_i$, and the local differentials $d_i$ are defined by counting rigid finite-energy $J'$-holomorphic cylinders in $\R \times K_i$.

Let us assume that the data $(\alpha',J')$ defines a DGA $(\A(\alpha'),d')$ whose homology is the full contact homology of $\xi$ (with $\Q$-coefficients). For every $r>0$ denote by $C^r_*(\alpha')$ the $\Q$-vector space freely generated by all the good closed $\alpha'$-Reeb orbits with $\alpha'$-action less than $r$ and degree $*$. Write $C^{[r,s)}_*(\alpha')$ for $C^s_*(\alpha')/C^r_*(\alpha')$ when $r<s$. By construction we clearly have $C^{[a-\delta,a+\delta)}_*(\alpha') = \oplus_i C_*(\alpha',K_i)$.

The differential $d'$ can be linearized by $\epsilon'$ in order to define a differential $d'_{\ep'}$ on $C^{[a-\delta,a+\delta)}_*(\alpha')$. Let $F$ be a finite-energy $J'$-holomorphic sphere with one positive puncture where it is asymptotic to some closed $\alpha'$-Reeb orbit with $\alpha'$-action less than $a+\delta$. If $F$ has two or more negative punctures then its asymptotic limits at the negative punctures are closed $\alpha'$-Reeb orbits with $\alpha'$-action strictly less than $a-\delta$, in fact, they all must have $\alpha'$-action less than $a+\delta - 2\sigma < a+\delta - 2\delta = a-\delta$. Thus $d'_{\ep'} : C^{[a-\delta,a+\delta)}_*(\alpha') \to C^{[a-\delta,a+\delta)}_{*-1}(\alpha')$ is defined by counting cylinders, in particular, it does not depend on $\ep'$. However, one could imagine that such cylinders connect an orbit in $K_i$ with an orbit in $K_j$ with $i\neq j$. But an easy compactness argument using the results in Appendix~\ref{exist_orbits} shows that this does not happen. Now, the cylinders connecting orbits in a given $K_i$ must all lie in ${\rm int}(K_i)$ by an application of an immediate modification of the argument used to prove Lemma~\ref{lemma_nearby_cylinders}. It follows that $d'_{\ep'} = \oplus_i d_i$ on $C^{[a-\delta,a+\delta)}_*(\alpha') = \oplus_i C_*(\alpha',K_i)$.
\end{proof}

\subsection{Small cobordisms}\label{small_cobordisms}

Our goal is to prove Proposition~\ref{prop_inv_small} below. We start by collecting some geometric constructions necessary for its proof.

Let $L>0$, a compact set $\Lambda \subset [0,+\infty) \setminus \Sigma(\alpha)$ and $J \in \J(\alpha)$ be fixed arbitrarily. We also fix any $C^\infty$-strong neighborhood $\U_0$ of $J$ in the space of almost complex structures of $\R\times N$. We claim that
\begin{itemize}
\item [{\bf (S)}] There exists a $C^\infty$-neighborhood of $\V$ of $\alpha$ in the space of defining contact forms for $\xi$ with its co-orientation, a $C^\infty$-strong neighborhood $\U_1\subset \U_0$ of~$J$ and numbers $\delta_0,\delta_1>0$ such that
\begin{itemize}
\item[(i)] $\alpha',\alpha'' \in \V \Rightarrow \alpha'/\alpha'' \leq 1+\delta_1 \text{ holds pointwise}$;
\item[(ii)] $\alpha' \in \V \Rightarrow {\rm dist}(\Lambda,\Sigma(\alpha')) \geq \delta_0$;
\item[(iii)] $c\in\Lambda \Rightarrow c(1+\delta_1)^6 < c+\delta_0$;
\item[(iv)] for every $\alpha',\alpha'' \in \V$, $J'\in\J(\alpha') \cap \U_1$ and $J'' \in \J(\alpha'') \cap \U_1$ there exists $\jbar \in \J_L(J'',J') \cap \U_0$, and an exact symplectic form $\Omega$ on $[-L,L]\times N$ such that $\Omega$ tames $\jbar$ on $[-L,L]\times N$ and admits a primitive $\eta$ 
$$ \begin{array}{ccc} \eta|_{T(\{L\}\times N)} = (1+\delta_1)\alpha', & & \eta|_{T(\{-L\}\times N)} = (1+\delta_1)^{-1}\alpha''. \end{array} $$
\end{itemize}
Moreover, the numbers $\delta_0,\delta_1$ and the neighborhoods $\V$, $\U_1$ can be taken arbitrarily small.
\end{itemize}

In (i) we denoted by $\alpha'/\alpha''$ the unique function satisfying $\alpha' = (\alpha'/\alpha'')\alpha''$. In~(iv) we identified $T(\{\pm L\}\times N) \simeq TN$. The proof of the existence of $\V,\U_1,\delta_0,\delta_1$ is standard and not complicated. The construction of exact symplectic forms as in (iv) can be done explicitly.

Before continuing we introduce some notation: given $\alpha',\alpha'',\alpha'''$ defining contact forms for $\xi$, co-orientations considered, numbers $R>L>0$ and almost complex structures $J'\in\J(\alpha')$, $J'' \in\J(\alpha'')$, $J'''\in\J(\alpha''')$, $\jbar_0 \in \J_L(J'',J')$ and $\jbar_1 \in \J_L(J''',J'')$, we define $\jbar_0 \odot_R \jbar_1 \in \J_{R+L}(J''',J')$ by
\begin{equation}
\jbar_0 \odot_R \jbar_1 = \left\{ \begin{aligned} & (\tau_{-R})^*\jbar_0 \ \text{ on } \ [0,+\infty) \times N \\ & (\tau_R)^*\jbar_1 \ \text{ on } \ (-\infty,0] \times N \end{aligned} \right.
\end{equation}
where $\{\tau_c\}_{c\in\R}$ denotes the $(\R,+)$-action on $\R\times N$ by translations in the first coordinate. This clearly is a smooth almost complex structure since $R>L$.

Perhaps after making $\V$ and $\U_1$ smaller, we also claim that we can achieve the following property.
\begin{itemize}
\item[{\bf (H)}] Let $\alpha',\alpha'',\alpha''' \in \V$, $J'\in\J(\alpha') \cap \U_1$, $J'' \in \J(\alpha'') \cap \U_1$, $J''' \in \J(\alpha''') \cap \U_1$ be fixed. Consider $\bar J_0 \in \J_L(J'',J') \cap \U_0$, $\jbar_1 \in \J_L(J''',J'') \cap \U_0$ and $\bar J_2 \in \J_L(J''',J') \cap \U_0$ with associated exact symplectic forms on $[-L,L] \times N$ as in item~(iv) of claim~(S). Then $\forall R>L$ there exists a smooth family $\{\jtil_\tau\}_{\tau\in[0,1]} \subset \J_{R+L}(J''',J') \cap \U_0$ such that $\jtil_0 = \jbar_0 \odot_R \jbar_1$ and $\jtil_1 = \jbar_2$. Moreover, there is an exact symplectic form $\Omega$ on $[-R-L,R+L]\times N$ taming $\jtil_\tau$ on $[-R-L,R+L]\times N$ $\forall \tau$, that admits a primitive $\eta$ satisfying
\begin{equation}\label{nice_forms_splitting}
\begin{array}{ccc} 
\eta|_{T(\{R+L\}\times N)} = (1+\delta_1)^3\alpha' & \text{and} & \eta|_{T(\{-R-L\}\times N)} = (1+\delta_1)^{-3}\alpha'''. \end{array}
\end{equation}
\end{itemize}

Let us derive some consequences. With $L$, $\Lambda$, $J$ and $\U_0$ fixed as above, consider $\V$, $\U_1$, $\delta_0$, $\delta_1$ obtained from claim (S). Assume that $\alpha',\alpha''\in\V$ are nondegenerate and $J'\in\J(\alpha') \cap \U_1$, $J''\in\J(\alpha'')\cap\U_1$ are generic enough in order to have well-defined DGAs $(\A(\alpha'),d')$, $(\A(\alpha''),d'')$ whose homologies are equal to the full contact homology of~$\xi$. Let $\jbar_0 \in \J_L(J'',J') \cap \U_0$ be as in~(iv) of claim~(S) above, which we assume generic enough to define a chain map $$ \Psi : (\A(\alpha'),d') \to (\A(\alpha''),d'') $$ which is also an algebra homomorphism inducing an isomorphism in homology.

\begin{lemma}\label{lemma_action_1}
If $c\in\Lambda$ and $\A^c(\alpha') \subset \A(\alpha')$, $\A^c(\alpha'') \subset \A(\alpha'')$ denote the subDGAs generated by the orbits with action less than $c$, then $\Psi$ satisfies $\Psi(\A^c(\alpha')) \subset \A^c(\alpha'')$.
\end{lemma}

\begin{proof}
The map $\Psi$ is defined by counting rigid finite-energy spheres with one positive puncture in $(\R\times N,\jbar_0)$, where the energy is defined with the help of the exact symplectic form $\Omega$ on $[-L,L]\times N$ satisfying the properties listed in item (iv) of claim (S). Let $F = (a,f)$ be such a finite-energy sphere with $m$ negative punctures, asymptotic to the closed $\alpha'$-Reeb $\gamma$ at its positive puncture and to the closed $\alpha''$-Reeb orbits $\gamma_1,\dots,\gamma_m$ at its negative punctures. Using Stokes theorem we get
\[
\begin{aligned}
& (1+\delta_1) \int_\gamma \alpha' - (1+\delta_1)^{-1} \sum_{i=1}^m \int_{\gamma_i} \alpha'' \\
& = \int_{\{a\geq L\}} (1+\delta_1)f^*d\alpha' + \int_{\{-L\leq a\leq L\}} F^*\Omega + \int_{\{a\leq-L\}} (1+\delta_1)^{-1}f^*d\alpha'' \geq 0.
\end{aligned}
\]
Thus, assuming that the $\alpha'$-action of $\gamma$ is less than $c$ we obtain
\[
\sum_{i=1}^m \int_{\gamma_i} \alpha'' \leq (1+\delta_1)^2 \int_\gamma \alpha' < (1+\delta_1)^2c \leq c+\delta_0
\]
in view of (iii). Using (ii) we conclude that $$ \sum_{i=1}^m \int_{\gamma_i} \alpha'' < c, $$ in particular, each $\gamma_i$ has $\alpha''$-action less than $c$, as was to be proved.
\end{proof}

The next step is

\begin{lemma}\label{lemma_action_2}
If $\V$ and $\U_1$ are small enough and $\epsilon'':\A(\alpha'') \to \Q$ is any augmentation then, defining $\epsilon' := \Psi^*\epsilon''$, the linearized chain map $$ \Psi_{\epsilon''} : (C(\alpha'),d'_{\epsilon'}) \to (C(\alpha''),d''_{\epsilon''}) $$ satisfies $\Psi_{\epsilon''}(C^a_*(\alpha')) \subset C^a_*(\alpha'')$ and induces an isomorphism between filtered linearized homologies
\begin{equation*}
\Psi_{\epsilon''} : HC^{a,\epsilon'}_*(\alpha') \stackrel{\sim}{\to} HC^{a,\epsilon''}_*(\alpha'')
\end{equation*}
for every $a\in\Lambda \cup \{0,+\infty\}$.
\end{lemma}

\begin{proof}
We know from~\cite{EGH,Bo} that $\Psi_{\epsilon''}$ induces an isomorphism between the non-filtered linearized homologies. Consider also $\jbar_1 \in \J_L(J',J'') \cap \U_0$ as in~(iv) of claim~(S) assumed generic enough to give a chain map $\Phi : (\A(\alpha''),d'') \to (\A(\alpha'),d')$ which is also an algebra homomorphism inducing an isomorphism in homology. Then by Lemma~\ref{lemma_action_1} we have $\Psi(\A^c(\alpha')) \subset \A^c(\alpha'')$ and $\Phi(\A^c(\alpha'')) \subset \A^c(\alpha')$ for every $c\in\Lambda$. 

Using all the assumed regularity, we know from standard glueing-compactness analysis that if $R\gg L$ is large enough then the map $\Phi \circ \Psi$ is equal to the chain map $\A(\alpha') \to \A(\alpha')$ induced by $\jbar_0 \odot_R \jbar_1 \in \J_{R+L}(J',J')$. This chain map will be denoted as $\Phi \odot_R \Psi$. Moreover, perhaps after shrinking $\V$ and $\U_1$, we can assume claim~(H), {\it i.e.}, we find a homotopy $\{\jtil_\tau\}_{\tau\in[0,1]} \subset \J_{R+L}(J',J')$ from $\jtil_0 = \jbar_0 \odot_R \jbar_1$ to $\jtil_1 = J'$ which is uniformly tamed on $[-R-L,R+L]\times N$ by an exact symplectic form $\Omega$ satisfying~\eqref{nice_forms_splitting}. If we assume that $\{\jtil_\tau\}$ is generic enough then it will induce a degree $+1$ derivation $K$ on $\A(\alpha')$ given by counting index $-1$ finite-energy $\jtil_\tau$-holomorphic spheres on $\R\times N$ with one positive puncture, for all $\tau \in [0,1]$. As is explained in~\cite{Bo}, this derivation establishes a chain homotopy between $\Phi \odot_R \Psi$ and the identity map.

We claim that $K(\A^c(\alpha')) \subset \A^c(\alpha')$, $\forall c\in\Lambda$. Fixing $c\in\Lambda$, consider $\tau\in[0,1]$ and a finite-energy $\jtil_\tau$-holomorphic index $-1$ cylinder $F=(a,f)$ asymptotic to $\gamma$ at the positive puncture and to $\gamma_1,\dots,\gamma_m$ at the negative punctures. As in Lemma~\ref{lemma_action_1} we use Stokes theorem to estimate
\[
\begin{aligned}
& (1+\delta_1)^3 \int_\gamma \alpha' - (1+\delta_1)^{-3} \sum_{i=1}^m \int_{\gamma_i} \alpha' = \int_{\{a\geq R+L\}} (1+\delta_1)^3f^*d\alpha' + \\
& + \int_{\{-R-L\leq a\leq R+L\}} F^*\Omega + \int_{\{a\leq-R-L\}} (1+\delta_1)^{-3}f^*d\alpha' \geq 0.
\end{aligned}
\]
Thus, assuming that the $\alpha'$-action of $\gamma$ is less than $c$ we obtain
\[
\sum_{i=1}^m \int_{\gamma_i} \alpha' \leq (1+\delta_1)^6 \int_\gamma \alpha' < (1+\delta_1)^6c \leq c+\delta_0
\]
in view of (iii). We conclude that each $\gamma_i$ has $\alpha'$-action not larger than $c+\delta_0$ and, using (ii), that their $\alpha'$-actions must be less than $c$.

In other words, we proved that $K$ is a derivation on the subDGA $\A^c(\alpha')$, thus inducing a chain homotopy between $(\Phi \odot_R \Psi)|_{\A^c(\alpha')}$ and the identity on $\A^c(\alpha')$. Since $\Phi \odot_R \Psi = \Phi \circ \Psi$ at the chain level when $R$ is large, we have that $(\Phi \circ \Psi)|_{\A^c(\alpha')}$ induces the identity on the level of homology of $\A^c(\alpha')$. Now note that $\epsilon'=\Psi^*\epsilon''$ and $\epsilon''$ are also augmentations for the subDGAs $(\A^c(\alpha'),d')$ and $(\A^c(\alpha''),d'')$, respectively. It follows as in the non-filtered case that $\Psi_{\epsilon''}|_{C^c(\alpha')}$ induces an isomorphism $HC^{c,\epsilon'}(\alpha') \simeq HC^{c,\epsilon''}(\alpha'')$ for every $c\in\Lambda$, see~\cite{Bo}.
\end{proof}

\begin{lemma}\label{lemma_action_3}
If $\V,\U_1$ are as in Lemma~\ref{lemma_action_2} then for every $a,b\in\Lambda \cup \{0,+\infty\}$ satisfying $b<a$ the map $\Psi_{\epsilon''}$ induces an isomorphism $HC^{[b,a),\epsilon'}(\alpha') \simeq HC^{[b,a),\epsilon''}(\alpha'')$, where $\epsilon'=\Psi^*\epsilon''$.
\end{lemma}

\begin{proof}
Using Lemma~\ref{lemma_action_2} there is a commutative diagram
\begin{equation*}
\xymatrix{
0 \ar[r] & C^b(\alpha') \ar[d]^{\Psi_{\epsilon''}} \ar[r] & C^a(\alpha') \ar[d]^{\Psi_{\epsilon''}} \ar[r] & C^a(\alpha')/C^b(\alpha') \ar[d]^{\Psi_{\epsilon''}} \ar[r] & 0 \\
0 \ar[r] & C^b(\alpha'') \ar[r] & C^a(\alpha'') \ar[r] & C^a(\alpha'')/C^b(\alpha'') \ar[r] & 0
}
\end{equation*}
where the rows are obviously exact (projection followed by inclusion). Passing to long exact sequences in homology we obtain
\begin{equation*}
\xymatrixcolsep{1pc}\xymatrix{
HC^{b,\epsilon'}_*(\alpha') \ar[d]^{\tilde\Psi_{\epsilon''}} \ar[r] & HC^{a,\epsilon'}_*(\alpha') \ar[d]^{\tilde\Psi_{\epsilon''}} \ar[r] & HC^{[b,a),\epsilon'}_*(\alpha') \ar[d]^{\tilde\Psi_{\epsilon''}} \ar[r] & HC^{b,\epsilon'}_{*-1}(\alpha') \ar[d]^{\tilde\Psi_{\epsilon''}} \ar[r] & HC^{a,\epsilon'}_{*-1}(\alpha') \ar[d]^{\tilde\Psi_{\epsilon''}} \\
HC^{b,\epsilon''}_*(\alpha'') \ar[r] & HC^{a,\epsilon''}_*(\alpha'') \ar[r] & HC^{[b,a),\epsilon''}_*(\alpha'') \ar[r] & HC^{b,\epsilon''}_{*-1}(\alpha'') \ar[r] & HC^{a,\epsilon''}_{*-1}(\alpha'')
}
\end{equation*}
where $\tilde\Psi_{\epsilon'}$ denote the induced maps in homology. By Lemma~\ref{lemma_action_2} the first, second, fourth and fifth maps are isomorphisms. The five-lemma applies and we conclude that the third map is also an isomorphism.
\end{proof}

Summarizing this discussion, we have proved

\begin{proposition}\label{prop_inv_small}
Fix a compact set $\Lambda \subset [0,+\infty) \setminus \Sigma(\alpha)$, $L>0$, $J \in \J(\alpha)$ and any neighborhood $\U_0$ of $J$ in the space of almost complex structures of $\R \times N$ with respect to the $C^\infty$-strong topology. Assume that $\U_1\subseteq \U_0$, $\V$ and $\delta_0,\delta_1>0$ are so that claim~(S) holds, and assume, possibly after shrinking $\V$ and $\U_1$, that claim~(H) also holds. If 
\begin{itemize}
\item $\alpha',\alpha'' \in \V$ are nondegenerate, 
\item $J'\in\J(\alpha')\cap\U_1$ and $J''\in\J(\alpha'')\cap\U_1$ are generic enough to define DGAs $(\A(\alpha'),d')$, $(\A(\alpha''),d'')$ whose homologies coincide with the full contact homology of $\xi$ with $\Q$-coefficients, 
\item $\jbar \in \J_L(J'',J') \cap \U_0$ is tamed by an exact symplectic form as in (iv) of claim~(S), and is generic enough to define a chain map $$ \Psi:(\A(\alpha'),d') \to (\A(\alpha''),d''), $$ 
\item $\epsilon'' : (\A(\alpha''),d'') \to \Q$ is any augmentation,
\end{itemize}
then the linearized map $\Psi_{\epsilon''}$ induces an isomorphism of filtered linearized homologies $$ HC^{[b,a),\Psi^*\epsilon''}(\alpha')\simeq HC^{[b,a),\epsilon''}(\alpha'') $$ for every pair of numbers $b<a$ in $\Lambda \cup \{0,+\infty\}$.
\end{proposition}

\subsection{Morse inequalities}

Suppose now that $\alpha$ has finitely many simple periodic orbits $\gamma_1, \gamma_2,\dots, \gamma_r$. In particular, its action spectrum is a discrete subset. Let $b^{[\ep]}_i = \dim HC_i^{[\ep]}(\xi)$ denote the Betti numbers and
$$ c_i = \sum_{k=1}^r \sum_j \dim HC_i(\alpha,\gamma_k^j) $$
the Morse type numbers. The main result in this section provides versions of weak and strong Morse inequalities for contact homology suitable for our applications. Given $a \notin \Sigma(\alpha)$ the relative Morse type numbers
\begin{equation}\label{rel_morse-type-numbers}
c^a_i = \sum_{k=1}^r \sum_{j; A(\gamma_k^j) < a} \dim HC_i(\alpha,\gamma_k^j)
\end{equation}
will play an important role in the arguments.

\begin{proposition}\label{mi}
Under the assumption that $\alpha$ has finitely many simple periodic orbits, $b^{[\ep]}_i$ and $c_i$ are finite for every $i \geq 2n-3$ and satisfy the inequalities
\begin{equation}\label{mi:weak}
b^{[\ep]}_i \leq c_i
\end{equation}
\begin{equation}\label{mi:strong}
b^{[\ep]}_i - b^{[\ep]}_{i-1} + ... \pm b^{[\ep]}_{2n-3} - C \leq c_i - c_{i-1} + ... \pm c_{2n-3},
\end{equation}
for every $i \geq 2n-3$. Here $C \geq 0$ is a constant that does not depend on $i$.
\end{proposition}

As a first step we will prove the following statement.


\begin{lemma}\label{stabilization}
For every $i \geq 2n-3$ there exists $a \notin \Sigma(\alpha)$, a $C^\infty$-neighborhood $\V$ of $\alpha$ and a $C^\infty$-strong neighborhood $\U$ of $J$ such that if $\alpha_1 \in \V$ is a nondegenerate contact form defining $\xi$, and $J_1 \in \J(\alpha_1) \cap \U$ is regular so that the data $(\alpha_1,J_1)$ defines a DGA $(\A(\alpha_1),d_1)$ whose homology is the full contact homology of $\xi$ with $\Q$-coefficients, then $$ \begin{array}{ccc} HC_l^{a,\epsilon_1}(\alpha_1) \simeq HC_l^{[\epsilon_1]}(\xi) & \text{and} & c^a_l = c_l \end{array} $$ for every augmentation $\epsilon_1 : (\A(\alpha_1),d_1) \to \Q$ and every $2n-3 \leq l \leq i$.
\end{lemma}

\begin{proof}
Define $a \in (0,+\infty) \setminus \Sigma(\alpha)$ by requiring that if $\gamma$ is a closed $\alpha$-Reeb orbit with action $\geq a$ then, after a $C^\infty$-small nondegenerate perturbation of contact form, $\gamma$ splits into closed Reeb orbits with degree $\leq 2n-4$ or with degree $\geq i+2$.

One can argue the existence of such $a$ as follows. Let $\gamma$ be any closed $\alpha$-Reeb orbit with mean index $h$. Then, as is well-known, under a small nondegenerate perturbation of the contact form $\gamma$ splits into closed Reeb orbits satisfying
\begin{equation}\label{CZ_split}
\mu_{CZ} \in [h-n+1,h+n-1].
\end{equation}
Consider $\gamma_0$ the underlying simple orbit, so that $\gamma = \gamma_0^j$ for some $j\geq 1$. Let $h$ be the mean index of $\gamma_0$. Either $h\leq0$, in which case $\gamma$ splits into closed Reeb orbits with degree $\leq 2n-4$ under small perturbations of the contact form, or $h>0$, in which case $\gamma$ splits into closed Reeb orbits with degree $\geq jh - 2 \geq jh_{\rm min}-2$. Here we denoted by $h_{\rm min}$ the minimum among the positive mean indices of the simple closed $\alpha$-Reeb orbits. In the latter case, taking $a$ large forces $j$ to be large and, in particular, $jh_{\rm min}-2 \geq i+2$. Thus $c^a_l = c_l$ for every $l=2n-3,\dots,i+1$.

Before proceeding we introduce some terminology. Given $D>0$ and $L>0$, defining contact forms $\alpha',\alpha''$ for $\xi$ satisfying $\alpha''/\alpha' < e^D$ pointwise, and almost complex structures $J'\in\J(\alpha')$, $J'' \in\J(\alpha'')$, we say that $\jbar_0 \in \J_L(J'',J')$ is $D$-small if there exists an exact symplectic form $\Omega$ on $[-L,L]\times N$ taming $\jbar_0$ on $[-L,L]\times N$ which admits a primitive that coincides with $e^D\alpha'$ on $T(\{L\}\times N)$, and with $e^{-D}\alpha''$ on $T(\{-L\}\times N)$. We identified $TN\simeq T(\{\pm L\}\times N)$, as usual.

Fix $L>0$, $J \in \J(\alpha)$ and a small $C^\infty$-strong neighborhood $\U_0$ of $J$. Applying Proposition~\ref{prop_inv_small} with $\Lambda = \{a\}$, we find a $C^\infty$-neighborhood $\V$ of $\alpha$ and $\U_1 \subseteq \U_0$ such that the following holds.
\begin{itemize}
\item[a)] Let $\alpha',\alpha'' \in \V$ be nondegenerate defining contact forms for $\xi$, and let $J' \in \J(\alpha') \cap \U_1$, $J'' \in \J(\alpha'') \cap \U_1$ be regular so that $(\alpha',J')$ and $(\alpha'',J'')$ define DGAs $(\A(\alpha'),d')$ and $(\A(\alpha''),d'')$, respectively, whose homologies coincide with the full contact homology of $\xi$ (with $\Q$-coefficients). There exists $\jbar_0 \in \J_L(J'',J') \cap \U_0$ with the following property: if $\jbar_0$ is assumed regular enough to define a chain map $\Psi:(\A(\alpha'),d') \to (\A(\alpha''),d'')$ and $\epsilon'':(\A(\alpha''),d'') \to \Q$ is any augmentation then $\Psi(\A^c(\alpha')) \subset \A^c(\alpha'')$ and the linearized map $\Psi_{\ep''} : (C^{[c,d)}(\alpha'),d'_{\Psi^*\ep''}) \to (C^{[c,d)}(\alpha''),d''_{\ep''})$ induces an isomorphism in homology 
\begin{equation}\label{isom_local_eq}
\widetilde \Psi_{\ep''} : HC^{[c,d),\Psi^*\ep''}_*(\alpha') \simeq HC^{[c,d),\ep''}_*(\alpha'')
\end{equation}
for every $c<d$ in $\{0,a,+\infty\}$. Moreover, since such $\jbar_0$ is obtained by using claims~(S) and~(H) from \S~\ref{small_cobordisms}, $\jbar_0$ as above can be taken $D$-small for some constant $D<\log 2$.
\end{itemize}
By Stokes theorem the map~\eqref{isom_local_eq} $\Psi_{\ep''}$ induces maps
\begin{equation}\label{map_twice_action}
\widetilde \Psi_{\ep''} : HC^{[0,\sigma),\Psi^*\ep''}_*(\alpha') \to HC^{[0,4\sigma),\ep''}_*(\alpha'')
\end{equation}
for every $\sigma\in\R$.

After shrinking $\V$ and $\U_1$ we can use~(H) to further assume the following.

\begin{itemize}
\item[b)] If $\jbar_1 \in \J_L(J',J'') \cap \U_0$ is assumed regular enough to define a chain map $\Phi:(\A(\alpha''),d'') \to (\A(\alpha'),d')$ with the same properties of $\Psi$ as in a), then $\forall R>L$ there is a homotopy $\{\widetilde J_\tau\}_{\tau\in[0,1]} \subset \J_{R+L}(J'',J'') \cap \U_0$ from $\widetilde J_0 = \jbar_1 \odot_R \jbar_0$ to $J''$ uniformly tamed by a suitable symplectic form on $[-R-L,R+L]\times N$ which can be used to estimate actions, and to prove that when $\{\widetilde J_\tau\}$ is regular enough it induces a degree $+1$ map $K$ on $(\A(\alpha''),d'')$ satisfying $K(\A^c(\alpha'')) \subset \A^c(\alpha'')$, $\forall c \in \{0,a,+\infty\}$. 
\end{itemize}

It follows as in Lemma~\ref{lemma_action_2} that $K$ defines a chain homotopy between $\Psi \circ \Phi$ and $id$ restricted to the subDGA $(\A^c(\alpha''),d'')$, $\forall c\in\{0,a,+\infty\}$. This is so since, when $R$ is large enough, $\Psi\circ\Phi$ coincides at the chain level with the map induced by $\jbar_1 \odot_R \jbar_0$. Then an augmentation $\epsilon'' : (\A(\alpha''),d'') \to \Q$ is homotopic to $(\Psi \circ \Phi)^*\epsilon''$ as augmentations on the subDGA $(\A^c(\alpha''),d'')$. Hence for every $c\in \{0,a,+\infty\}$ there are isomorphisms
\begin{equation}\label{isom_def}
\theta : HC^{[0,c),\epsilon''}_*(\alpha'') \stackrel{\sim}{\to} HC^{[0,c),(\Psi\circ\Phi)^*\epsilon''}_*(\alpha'')
\end{equation}
satisfying $(\widetilde{\Psi\circ\Phi})_{\ep''} \circ \theta = id$.

Moreover, all $\jtil_\tau$ can be taken $D$-small for some $0<D\ll1$ so that we have maps
\begin{equation}\label{hom_local_eq}
\theta : HC^{\sigma,\ep''}(\alpha'') \to HC^{4\sigma,(\Psi\circ\Phi)^*\epsilon''}(\alpha'')
\end{equation}
making the following diagram
\begin{equation}\label{action_diagram}
\xymatrix{
HC^{[0,\sigma),\ep''}(\alpha'') \ar[r]^{\theta} \ar[d]^{\iota} & HC^{[0,4\sigma),(\Psi\circ\Phi)^*\ep''}(\alpha'') \ar[d]^{\widetilde\Phi_{\Psi^*\ep''}} \\
HC^{[0,64\sigma),\ep''}(\alpha'') & \ar[l]^{\widetilde \Psi_{\ep''}} HC^{[0,16\sigma),\Psi^*\ep''}(\alpha')
}
\end{equation}
commutative, for every $\sigma>0$. Here $\iota$ is induced by the inclusion.

Consider increasing sequences $\{\sigma_k\}_{k\geq1}$, $\{b_k\}_{k\geq0}$ such that $$ \begin{array}{cccc} b_0=a, & \sigma_k > 4b_{k-1} & \text{and} & b_k > 16 \sigma_k, \ \forall k\geq 1. \end{array} $$ By the properties of the number $a$ and of the neighborhood $\V$, there are nondegenerate defining contact forms $\alpha_k \in \V$ for $\xi$, $\alpha_k\to\alpha$ in $C^\infty$, such that there are no closed $\alpha_k$-Reeb orbits with action in $[a,b_k]$ and degree in $[2n-3,i+1]$. Choose $J_k \in \J(\alpha_k) \cap \U_1$ regular enough so that each $(\alpha_k,J_k)$ defines a DGA $(\A(\alpha_k),d_k)$ whose homology is the full contact homology with $\Q$-coefficients. Consider $\jbar^+_k \in \J_L(J_1,J_k) \cap \U_0$, $\jbar^-_k \in \J_L(J_k,J_1) \cap \U_0$ satisfying the properties described above. If the $\jbar^\pm_k$ are assumed regular enough then they define chain maps $$ \begin{array}{ccc} \Psi_k : (\A(\alpha_k),d_k) \to (\A(\alpha_1),d_1), & & \Phi_k : (\A(\alpha_1),d_1) \to (\A(\alpha_k),d_k) \end{array} $$ preserving the subDGAs filtered by any $c \in \{0,a,+\infty\}$. Fix an augmentation $\ep_1:(\A(\alpha_1),d_1) \to \Q$, define $\epsilon_k := \Psi_k^*\ep_1$ and denote $$ \begin{array}{ccc} H^k_* = HC^{[0,\sigma_k),\epsilon_1}_*(\alpha_1), & & \widehat H^k_* = HC^{[0,b_k),\epsilon_k}_*(\alpha_k). \end{array} $$ Consider the maps $\theta_k : H^k_*  \to HC^{[0,4\sigma_k),\Phi_k^*\ep_k}_*(\alpha_1)$ as in~\eqref{hom_local_eq}, and the maps induced by inclusions
\[
\begin{array}{ccc} \iota_k : HC^{[0,4b_k),\ep_1}_*(\alpha_1) \to H^{k+1}_*, & & \widehat{\iota}_k : HC^{[0,16\sigma_k),\ep_k}_*(\alpha_k) \to \widehat H^k_*. \end{array}
\]
Now define $f_k : H^k_* \to \widehat H^k_*$ by $f_k = \widehat \iota_k \circ (\Phi_k)_{\ep_k} \circ \theta_k$, and $g_k : \widehat H^k_* \to H^{k+1}_*$ by $\iota_k \circ (\Psi_k)_{\ep_1}$. By the diagram~\eqref{action_diagram}, $g_k \circ f_k$ is the same map $H^{k}_* \to H^{k+1}_*$ induced by the inclusion $C^{[0,\sigma_k)}_*(\alpha_1) \to C^{[0,\sigma_{k+1})}_*(\alpha_1)$.

We claim that $f_{k+1} \circ g_k : \widehat H^k_l \to \widehat H^{k+1}_l$ is an isomorphism for each $k\geq1$ and $l\in[2n-3,i+1]$. In fact, since $\alpha_k$ has no closed Reeb orbits with action in $[a,b_k]$ and degree $l\in[2n-3,i+1]$, the inclusion map $HC^{[0,a),\ep_k}_l(\alpha_k) \to \widehat H^k_l$ is an isomorphism when $2n-3 \leq l \leq i+1$. After these identifications we have that $f_{k+1}\circ g_k$ induces the map
\[
\begin{aligned} & HC^{[0,a),\ep_k}_l(\alpha_k) \stackrel{(\Psi_k)_{\ep_1}}{\longrightarrow} HC^{[0,a),\ep_1}_l(\alpha_1) \stackrel{\theta_k}{\to} \\ & \stackrel{\theta_k}{\to} HC^{[0,a),\Phi_{k+1}^*\ep_{k+1}}_l(\alpha_1) \stackrel{(\Phi_{k+1})_{\ep_{k+1}}}{\longrightarrow} HC^{[0,a),\ep_{k+1}}_l(\alpha_{k+1}) \end{aligned}
\]
which is an isomorphism when $2n-3\leq l \leq i+1$ in view of~\eqref{isom_local_eq},~\eqref{isom_def}.

The vector spaces $\{H^k_*\}_k$ and the linear maps $F_k = g_k \circ f_k : H^k_* \to H^{k+1}_*$ determine a direct system. Similarly, $\{\widehat H^k_*\}_k$ and $G_k = f_{k+1} \circ g_k : \widehat H^k_* \to \widehat H^{k+1}_*$ determine another direct system. The maps $f_k$ determine a map $\lim_k H^k_* \to \lim_k \widehat H^k_*$, and the $g_k$ determine a map $\lim_k \widehat H^k_* \to \lim_k H^k_*$. Since the $G_k$ are isomorphisms in degrees $l\in[2n-3,i+1]$, these maps determine an isomorphism $\lim_k H^k_l \sim \lim_k \widehat H^k_l$ for $l\in[2n-3,i+1]$. But $\lim_k H^k_l = HC^{[\ep]}_l(\xi)$ as is well-known, and $\lim_k \widehat H^k_l = HC^{[0,a),\ep_1}_l(\alpha_1)$ for $l\in[2n-3,i+1]$ since the $G_k$ are isomorphisms in these degrees.
\end{proof}

\begin{remark}\label{rmk_a_large}
The number $a$ in Lemma~\ref{stabilization} can be taken arbitrarily large.
\end{remark}

\begin{proof}[Proof of Proposition~\ref{mi}]
Let $J \in \J(\alpha)$ and $a \in (0,+\infty) \setminus \Sigma(\alpha)$ be arbitrary. Also, let $m < 0$ be an integer such that every periodic $\alpha$-Reeb orbit with mean index in $[m-2n+4,m+2]$ has action bigger than $a$. The existence of $m$ depending on $a$ follows from the assumption that there are finitely many simple periodic orbits. If $\alpha_1$ is a nondegenerate defining contact form for $\xi$ sufficiently $C^\infty$-close to $\alpha$, $J_1 \in \J(\alpha_1)$ is $C^\infty$-strong close to $J$ and regular enough to get the DGA $(\A(\alpha_1),d_1)$ of full contact homology with $\Q$-coefficients well-defined, and $\ep_1 : (\A(\alpha_1),d_1) \to \Q$ is any augmentation, then it follows from~\eqref{CZ_split} that $HC_m^{[a',a''),\ep_1}(\alpha_1) = 0$ for all numbers $0 \leq a^\prime < a^{\prime\prime} \leq a$ not in $\Sigma(\alpha_1)$. From the long exact sequence for filtered contact homology we get that the function
$$ \chi_l(a',a'') := \dim HC_l^{[a',a''),\ep_1}(\alpha_1) - \dim HC_{l-1}^{[a',a''),\ep_1}(\alpha_1) + ... \pm \dim HC_{m}^{[a',a''),\ep_1}(\alpha_1) $$
is subadditive for every $l \geq m$, {\it i.e.}, $\chi_l(a',a''') \leq  \chi_l(a',a'') + \chi_l(a'',a''')$ whenever $0 \leq a' < a'' < a''' \leq a$ do not belong to $\Sigma(\alpha_1)$. This implies in a standard way the strong Morse inequality
\begin{equation}\label{str_morse_filtered_1}
b^{a,\ep_1}_l(\alpha_1) - b^{a,\ep_1}_{l-1}(\alpha_1) + ... \pm b^{a,\ep_1}_{m}(\alpha_1) \leq c^a_l - c^a_{l-1} + ... \pm c^a_{m}
\end{equation}
where $b^{a,\ep_1}_*(\alpha_1) := \dim HC^{[0,a),\ep_1}_*(\alpha_1)$. Proposition~\ref{lch-ch} was used. Inequality~\eqref{str_morse_filtered_1} holds for every $l>m$. Notice that $b^{a,\ep_1}_l$ and $c^a_l$ are finite for every $l \in \Z$; this is trivial in view of action bounds since $\alpha_1$ is $C^\infty$-close to $\alpha$. By~\eqref{str_morse_filtered_1} with $l=j\geq0$ and $l=j+1$ we get
\begin{equation}\label{str_morse_filtered_2}
b^{a,\ep_1}_j(\alpha_1) \leq c_j^a, \ \ \forall j\geq0.
\end{equation}


Now fix any $i\geq 2n-3$ and consider $\V$, $\U$ and $a$ given by Lemma~\ref{stabilization}. Choose nondegenerate perturbation $\alpha_1 \in \V$ of $\alpha$, $J_1 \in \J(\alpha_1) \cap \U$ regular enough to get the DGA $(\A(\alpha_1),d_1)$ well-defined, and let $\ep_1 : (\A(\alpha_1),d_1) \to \Q$ be any augmentation in the class $[\ep]$. By our arguments so far we can assume that $(\alpha_1,J_1)$ is close enough to $(\alpha_,J)$ so that~\eqref{str_morse_filtered_2} holds for all $j\geq0$. In particular, taking $j=i$ we get
$$ b^{[\ep]}_i = b^{a,\ep_1}_i(\alpha_1) \leq c^a_i = c_i $$ from Lemma~\ref{stabilization}.

%
%

Inequality~\eqref{mi:strong} does not follow directly from the previous discussion due to the fact that $m$ depends on $a$. To circumvent this problem, the idea is to truncate the action filtration from below in a suitable way. The price that we have to pay is to add a correction term in the inequalities which in turn does not depend on $a$.

Take $\tilde a \notin \Sigma(\alpha)$, $\tilde a>0$, such that every periodic $\alpha$-Reeb orbit with positive mean index and action bigger than $\tilde a$ has mean index greater than or equal to $2n$. To find $\tilde a$ we need the assumption that there are finitely many simple periodic $\alpha$-Reeb orbits. Let $\Delta_{\tilde a} := \max \ \{ \Delta(\gamma): \int_\gamma \alpha \leq \tilde a\} \cup \{0\}$ and take $i>\Delta_{\tilde a}+2n$. In view of Remark~\ref{rmk_a_large} we find $a>\tilde a$ such that the conclusions of Lemma~\ref{stabilization} hold. Moreover, looking at the proof of Lemma~\ref{stabilization} we can take $a$ such that $c^a_l=c_l$, $\forall l\in[2n-3,i+1]$. By~\eqref{CZ_split} we conclude that $$ HC_{2n-3}^{[a',a''),\ep_1}(\alpha_1) = 0 \ \ \forall \tilde a \leq a' < a'' \leq a $$ when $\alpha_1 \in \V$ is a nondegenerate defining contact form for $\xi$ sufficiently $C^\infty$-close to $\alpha$ (choices of a regular $J_1 \in \J(\alpha_1)$ and of $\ep_1$ are implicit here). Hence
$$ \chi_l(a',a'') := \dim HC_l^{[a',a''),\ep_1}(\alpha_1) - \dim HC_{l-1}^{[a',a''),\ep_1}(\alpha_1) + ... \pm \dim HC_{2n-3}^{[a',a''),\ep_1}(\alpha_1) $$
is subadditive for every $l \geq 2n-3$, that is, given $\tilde a < a' < a'' < a^{\prime\prime\prime} \leq a$ then $\chi_l(a',a^{\prime\prime\prime}) \leq  \chi_l(a',a'') + \chi_l(a'',a^{\prime\prime\prime})$. This implies the inequality
\begin{equation}
\label{tr-mi}
b^{[\tilde a,a),\ep_1}_l - b^{[\tilde a,a),\ep_1}_{l-1} + ... \pm b^{[\tilde a,a),\ep_1}_{2n-3} \leq c^{[\tilde a,a)}_l - c^{[\tilde a,a)}_{l-1} + ... \pm c^{[\tilde a,a)}_{2n-3},
\end{equation}
where $b^{[\tilde a,a),\ep_1}_l$ denotes the rank of $HC_l^{[\tilde a,a),\ep_1}(\alpha)$ and
$$ c^{[\tilde a,a)}_l := \sum_{k=1}^r \sum_{j; \tilde a \leq A(\gamma_k^j) < a} \dim HC_l(\alpha,\gamma_k^j). $$
Here Proposition~\ref{lch-ch} is used.

A closed $\alpha_1$-Reeb orbit with degree in $[\Delta_{\tilde a} + 2n-3,i]$ and action $\leq a$ must also have action 
$\geq\tilde a$ in view of~\eqref{CZ_split} and the definition of $\Delta_{\tilde a}$. 
Thus 
$HC_l^{[\tilde a,a),\ep_1} (\alpha_1) \simeq HC_l^{[0,a),\ep_1} (\alpha_1) \simeq HC^{[\ep]}_l(\xi)$ for every $l \in [\Delta_{\tilde a} + 2n-3,i]$. Here Lemma~\ref{stabilization} was used again. Now it is easy to find $C \geq 0$ independent of $i$ such that
$$ b^{[\ep]}_i - b^{[\ep]}_{i-1} + ... \pm b^{[\ep]}_{2n-3} - C \leq c_i - c_{i-1} + ... \pm c_{2n-3} $$
for every $i \geq 2n-3$.
\end{proof}

\begin{remark}\label{mi-filtered}
One easily concludes the Morse inequalities for filtered contact homology from the beginning of the proof of Proposition~\ref{mi}. More precisely, fixing an arbitrary $a \in (0,+\infty) \setminus \Sigma(\alpha)$. In the notation of the above proof, 
\begin{equation}
b^{a,\ep_0}_i \leq c^a_i
\end{equation}
and
\begin{equation}
b^{a,\ep_0}_i - b^{a,\ep_0}_{i-1} + ... \pm b^{a,\ep_0}_{2n-3} - C \leq c^a_i - c^a_{i-1} + ... \pm c^a_{2n-3},
\end{equation}
for every $a \notin \Sigma(\alpha)$ and a constant $C\geq 0$ that does not depend on $i$, $a$ and $\ep_0$.
\end{remark}

\section{Proofs of the applications}\label{proof_appls}

\subsection{Proof of Theorem~\ref{inf_orbits}}
\label{inf_orbs}

An important ingredient in the proof of our applications is the following lemma that gives uniform bounds for the Morse type numbers of periodic orbits with mean index different from zero and follows easily from Theorem~\ref{main}.

\begin{lemma}\label{unif_bound}
Let $\gamma$ be an isolated periodic orbit of the Reeb flow of $\alpha$ with mean index different from zero such that $\gamma^j$ is isolated for every $j \in \N$. There exists a constant $B>0$ such that
$$ \sum_j \dim HC_i(\alpha,\gamma^j) < B $$
for every $i \in \Z$ .
\end{lemma}

\begin{proof}
Since $\gamma^j$ is isolated for every $j \in \N$, we conclude from Theorem~\ref{main} that there exists a constant $C>0$ such that
$$ \dim HC_i(\alpha,\gamma^j) < C $$
for every $i \in \Z$ and $j \in \N$. By~\eqref{CZ_split} we have that $HC_*(\alpha,\gamma^j)$ is supported in the interval $[j\Delta(\gamma) - 2,j\Delta(\gamma) + 2n - 4]$, {\it i.e.}, $HC_i(\alpha,\gamma^j) = 0$ if $i<j\Delta(\gamma) - 2$ or $i>j\Delta(\gamma) + 2n - 4$. The result follows.
\end{proof}

\begin{proof}[Proof of Theorem~\ref{inf_orbits}]
We will prove the result in the case that there exists a positive sequence of integers $l_i \to \infty$ such that $b_{l_i}^{[\ep]}(\xi) \to \infty$ since the negative case is analogous. Suppose that there exists a contact form for $\xi$ with finitely many simple closed orbits. By inequality~\eqref{CZ_split} only periodic orbits with positive mean index can contribute to $c_i$ for $i \geq 2n-3$. By Lemma~\ref{unif_bound} there exists a constant $B>0$ such that $c_i < B$ for every $i \geq 2n-3$. Hence by our assumption and inequality \eqref{mi:weak} we obtain a contradiction.
\end{proof}

\subsection{Invariance of the growth rate}
\label{inv:grate}

We will reproduce the argument of Seidel in \cite[Section 4a]{Se} that shows the invariance of the growth rate for symplectic cohomology under Liouville isomorphisms. However, the argument has to be adapted to our context, where we have to deal with augmentations. Let $\alpha_0$ and $\alpha_1$ be two non-degenerate contact forms for $\xi$ and choose an asymptotically cylindrical exact symplectic cobordism from the symplectization of $\alpha_0$ to that of $\alpha_1$. In other words, we have constants $b,L>0$, almost complex structures $J_0 \in \J(\alpha_0)$, $J_1 \in \J(\alpha_1)$, $\jbar \in \J_L(J_1,J_0)$ and an exact symplectic form on the neck $[-L,L] \times N$ taming $\jbar$ on this neck, which coincides with $e^b d\alpha_0$ on $T(\{L\} \times N) \simeq TN$, and with $e^{-b} d\alpha_1$ on $T(\{-L\} \times N) \simeq TN$. We assume regularity in the sense that $J_0$, $J_1$ define DGAs $(\A(\alpha_0),d_0)$, $(\A(\alpha_1),d_1)$ for the full contact homology of $\xi$ with rational coefficients, and $\jbar$ defines chain map $\Psi: (\A(\alpha_0),d_0) \to (\A(\alpha_1),d_1)$ between these DGAs. An augmentation $\ep$ for $(\A(\alpha_1),d_1)$ yields an augmentation $\Psi^*\ep$ for $(\A(\alpha_0),d_0)$ and $\Psi$ induces an isomorphism $\tilde\Psi_\ep: HC^{\Psi^*\ep}(\alpha_0) \to HC^{\ep}(\alpha_1)$. As in~\ref{flch}, given $a>0$ and a nondegenerate defining contact form $\alpha'$ for $\xi$, let $\A^a(\alpha')$ be the subalgebra generated by the good periodic $\alpha'$-Reeb orbits with action less than $a$. It turns out that there exists a constant $D_1>0$ depending on the geometric data of the cobordism such that $\Psi(\A^a(\alpha_0)) \subset \A^{D_1a}(\alpha_1)$ for every $a>0$. In fact, by an application of Stokes theorem, we can take $D_1 = e^{2b}$.

Exchanging the roles of $\alpha_0$ and $\alpha_1$ we get, as in the above discussion, a constant $D_2>0$ and a chain map $\Phi: (\A(\alpha_1),d_1) \to (\A(\alpha_0),d_0)$ between DGAs satisfying $\Phi(\A^a(\alpha_1)) \subset \A^{D_2a}(\alpha_0)$ for every $a>0$. The augmentations $\Phi^*\Psi^*\ep$ and $\ep$ are homotopic: there exists a degree $+1$ derivation $K$ such that $\Phi^*\Psi^*\ep = \ep\circ e^{d_1\circ K + K\circ d_1}$. In fact, the map $K$ is obtained by counting rigid holomorphic spheres with one positive puncture in a regular homotopy of cobordisms. One can check that such a family of cobordisms can be taken uniformly tamed (on a large neck) by a suitable exact symplectic form which provides a constant $D_3>0$ such that $e^{d_1\circ K + K\circ d}$ sends $\A^a(\alpha_1)$ to $\A^{D_3a}(\alpha_1)$, for every $a>0$.

Let $D = \max_i D_i$. It turns out that the induced maps on the homology fit into the commutative diagram
\begin{equation*}
\xymatrix @-1.1pc{
\cdots && \cdots \\
& HC_*^{D^8a,\Phi^*\Psi^*\ep}(\alpha_1) \ar[ul]
\\
\ar[uu]
HC_*^{D^6a,\Psi^*\ep}(\alpha_0) \ar[rr] && HC_*^{D^7a,\ep}(\alpha_1) \ar[ul] \ar[uu]
\\
& HC_*^{D^5a,\Phi^*\Psi^*\ep}(\alpha_1) \ar[ul]
\\
\ar[uu]
HC_*^{D^3a,\Psi^*\ep}(\alpha_0) \ar[rr] && HC_*^{D^4a,\ep}(\alpha_1) \ar[ul] \ar[uu]
\\
& HC_*^{D^2a,\Phi^*\Psi^*\ep}(\alpha_1) \ar[ul]
\\
\ar[uu]
HC_*^{a,\Psi^*\ep}(\alpha_0) \ar[rr] && HC_*^{Da,\ep}(\alpha_1) \ar[ul] \ar[uu]
}
\end{equation*}
where the maps in the vertical arrows are those induced by the inclusion.

Now, suppose that $HC^{\Psi^*\ep}(\alpha_0) \simeq HC^{\ep}(\alpha_1)$ is infinite-dimensional, since, otherwise, we would have $\Gamma^{\Psi^*\ep}(\alpha_0)  = \Gamma^{\ep}(\alpha_1) = 0$. Then,
\begin{align*}
\Gamma^{\Psi^*\ep}(\alpha_0)^{-1} & = \liminf_{a\to\infty} \frac{\log a}{\log r(\Psi^*\ep,\alpha_0,a)} = \liminf_{a\to\infty} \frac{\log Da}{\log r(\Psi^*\ep,\alpha_0,a)} \\
& \geq \liminf_{a\to\infty} \frac{\log Da}{\log r(\ep,\alpha_1,a)} = \Gamma^{\ep}(\alpha_1)^{-1},
\end{align*}
where $r(\ep,\alpha,a)$ is the rank of $\iota(HC_*^{a,\ep}(\alpha))$. Inverting the roles of $\alpha_0$ and $\alpha_1$ we conclude that the collection of all growth rates $\{\Gamma^{\ep}(\alpha)\}$ associated to triples $(\alpha_1,J_1,\ep)$ as above is an invariant of the contact structure.

\subsection{Proof of Theorem~\ref{grate}}\label{proof:grate}

Suppose that there is a (possibly degenerate) contact form $\alpha$ defining $\xi$ with finitely many simple periodic Reeb orbits $\ga_1,...,\ga_r$.

Let $\alpha'$ be an arbitrary nondegenerate contact form defining $\xi$, and assume the existence of $J' \in \J(\alpha)$ which is regular enough to get the DGA $(\A(\alpha'),d')$ well-defined. Fix $J \in \J(\alpha)$ and $L>0$ arbitrarily. We can find a $C^\infty$-neighborhood $\V$ of~$\alpha$, a $C^\infty_{\rm strong}$-neighborhood $\U$ of $J$ in the set of almost complex structures on $\R \times N$ and a constant $b>0$ such that the following holds:
\begin{itemize}
\item[{\bf (C)}] For every contact form $\alpha'' \in \V$ defining $\xi$ and every $J'' \in \J(\alpha'') \cap \U$ there exist almost complex structures $\jbar^+ \in \J_L(J'',J')$, $\jbar^- \in \J_L(J',J'')$ and exact symplectic forms $\Omega^\pm$ on $[-L,L] \times N$ such that $\Omega^\pm$ tames $\jbar^\pm$ on $[-L,L] \times N$, $$ \begin{array}{cc} \Omega^+|_{T(\{L\} \times N)} = e^bd\alpha', & \Omega^+|_{T(\{-L\} \times N)} = e^{-b}d\alpha'' \end{array} $$ and $$ \begin{array}{cc} \Omega^-|_{T(\{L\} \times N)} = e^bd\alpha'', & \Omega^-|_{T(\{-L\} \times N)} = e^{-b}d\alpha'. \end{array} $$ Moreover, $\forall R>L$ one finds a smooth family $\{\jtil_\tau\}_{\tau\in[0,1]}\subset \J_{R+L}(J',J')$ and an exact symplectic form $\Omega_R$ on $[-R-L,R+L] \times N$ such that $\Omega_R$ tames $\jtil_\tau$ on $[-R-L,R+L] \times N, \ \forall \tau \in [0,1]$, $$ \begin{array}{cccc} \jtil_0 = \jbar^+ \odot_R \jbar^-, & \jtil_1 = J' & \text{and} & \Omega_R|_{T(\{\pm L\} \times N)} = e^{\pm 2b}d\alpha'. \end{array} $$
\end{itemize}

The proof of~(C) is easy. Set $D = e^{4b}$ and choose $a_n \to +\infty$, $a_n \not\in \Sigma(\alpha)$ satisfying $a_{n+1} > D^3a_n, \ \forall n$.

\begin{lemma}\label{lemma_approx_proc}
There exists a sequence of nondegenerate contact forms $\alpha''_n$ defining $\xi$ and satisfying $\alpha''_n \to \alpha$ in $C^\infty$ such that the following holds for every $n$: if $J''_n \in \J(\alpha''_n)$ is regular enough to get the associated DGA $(\A(\alpha''_n),d''_n)$ of full contact homology well-defined and $\ep_n : \A(\alpha''_n) \to \Q$ is any augmentation then $\dim HC^{a,\ep_n}_l(\alpha''_n) \leq c^{a}_l$ holds for every $0\leq a < a_n$ and $l\in\Z$.
\end{lemma}

Here $c^a_l$ are the Morse-type numbers~\eqref{rel_morse-type-numbers}.

\begin{proof}
If $\delta_n>0$ is small enough then $[c-\delta_n,c+\delta_n] \cap \Sigma(\alpha) = \{c\}$ for every $c\in \Sigma(\alpha) \cap [0,a_n)$ and, in view of Proposition~\ref{lch-ch}, $$ HC^{[c-\delta_n,c+\delta_n),\ep_n}_*(\alpha''_n) \simeq \bigoplus_{\{\gamma : \int_\gamma\alpha = c\}} HC_*(\gamma,\alpha) $$ whenever $\alpha''_n$ is a nondegenerate defining contact form for $\xi$ sufficiently $C^\infty$-close to $\alpha$ and $J''_n,\ep_n$ are as in the statement. The direct sum above is taken over the closed $\alpha$-Reeb orbits with $\alpha$-action equal to~$c$. Clearly, perhaps after taking $\alpha''_n$ closer to $\alpha$, we can assume that if $c_0 < c_1$ belong to $\Sigma(\alpha) \cap [0,a_n)$ and satisfy $(c_0,c_1) \cap \Sigma(\alpha) = \emptyset$ then there are no closed $\alpha''_n$-Reeb orbits with $\alpha''_n$-action in $[c_0+\delta_n,c_1-\delta_n]$. The conclusion follows from the long exact sequence induced by the action filtration. 
\end{proof}

With $\V$ and $\U$ given by~(C), we can assume all the $\alpha''_n$ given as in Lemma~\ref{lemma_approx_proc} lie in $\V$, and we also assume that each $J''_n \in \J(\alpha''_n) \cap \U$ is regular enough to get the DGA $(\A(\alpha''_n),d''_n)$ of full contact homology with $\Q$-coefficients well-defined. According to~(C) we find $\jbar^+_n \in \J_L(J''_n,J')$, $\jbar^-_n \in \J_L(J',J''_n)$ tamed by suitable exact symplectic forms on the neck $[-L,L] \times N$ as described above. Hence, if $\jbar^\pm_n$ are regular enough they can be used to define chain maps $$ \begin{array}{cc} \Psi_n : (\A(\alpha''_n),d''_n) \to (\A(\alpha'),d'), & \Phi_n : (\A(\alpha'),d') \to (\A(\alpha''_n),d''_n) \end{array} $$ which are algebra homomorphisms given by the count of rigid holomorphic spheres with one positive puncture. By Stokes theorem these taming symplectic forms can be used to show that $\Psi_n(\A^c(\alpha''_n)) \subset \A^{Dc}(\alpha')$ and $\Phi_n(\A^c\alpha')) \subset \A^{Dc}(\alpha''_n)$ for every $c\in\R$. By the same token, if the homotopy $\{\jtil_\tau\}$ as given by~(C) is assumed regular enough then it defines a degree $+1$ derivation $K$ on $\A(\alpha')$ satisfying $K(\A^c(\alpha')) \subset \A^{Dc}(\alpha'), \ \forall c \in \R$. When $R$ is large enough $K$ determines a chain homotopy between $\Psi_n \circ \Phi_n$ and the identity. Thus, choosing any augmentation $\ep' : A(\alpha') \to \Q$, we have a commutative diagram
\begin{equation}\label{action_diagram_2}
\xymatrix{
HC^{[0,a_{n-1}),\ep'}(\alpha') \ar[r] \ar[d]^{\iota} & HC^{[0,Da_{n-1}),(\Psi_n\circ\Phi_n)^*\ep'}(\alpha') \ar[d]^{(\Phi_n)_{\Psi_n^*\ep'}} \\
HC^{[0,D^3a_{n-1}),\ep'}(\alpha') \ar[d]^{\iota} & \ar[l]^{(\Psi_n)_{\ep''}} HC^{[0,D^2a_{n-1}),\Psi_n^*\ep'}(\alpha''_n) \\
HC^{\ep'}(\alpha')
}
\end{equation}
as in~\eqref{action_diagram}, for every $c \in \R$. Here $\iota$ denotes maps on homology induced by inclusions of chain subcomplexes. Since $D^2a_{n-1} \leq a_n$ we can apply~\eqref{action_diagram_2} and Lemma~\ref{lemma_approx_proc} to estimate
\[
\dim \iota(HC^{[0,a_{n-1}),\ep'}_l(\alpha')) \leq \dim HC^{[0,D^2a_{n-1}),\Psi_n^*\ep'}_l(\alpha''_n) \leq c^{a_n}_l
\]
for every $n$ and $l$. By Theorem~\ref{main} we find $C>0$ such that $c^{a_n}_l \leq Ca_n$ for all $n$ and $l$, implying $\Gamma^{\ep'}(\alpha') \leq 1$. Theorem~\ref{grate} follows.


\subsection{Resonance relations}\label{mec}

\begin{proof}[Proof of Theorem~\ref{thm:mec}]
We will prove the result for the positive mean Euler characteristic since the negative case is analogous. Let $\alpha$ be a contact form with finitely many simple periodic orbits and denote by $\gamma_1, \gamma_2, ..., \gamma_r$ those with positive mean index. Let $B_m$ and $C_m$ be the left and right sides of~\eqref{mi:strong} respectively for $i=m$. Using \eqref{mi:strong} for $m+1$ and $m$ and Lemma~\ref{unif_bound} we conclude that there exist constants $B$ and $C$ such that
$$ |B_{m} - C_{m}| \leq C + c_{m+1} \leq C + B $$
for every $m\geq 2n-3$. Consequently,
\begin{equation}\label{global_mec_eq}
\begin{aligned}
\chi_+^{[\ep]}(\xi) & = \lim_{m\to\infty}\frac{(-1)^m}{m} B_m = \lim_{m\to\infty}\frac{(-1)^m}{m} C_m \\ 
& = \sum_{s=1}^r \lim_{m\to\infty}\frac{1}{m}\sum_j\sum_{i=2n-3}^m (-1)^i\dim HC_i(\alpha,\gamma_s^j) \\
& = \sum_{s=1}^r \lim_{m\to\infty}\frac{1}{m}\sum_j\sum_{i=0}^m (-1)^i\dim HC_i(\alpha,\gamma_s^j).
\end{aligned}
\end{equation}
For any periodic $\alpha$-Reeb orbit $\gamma$ with mean index $\Delta(\gamma)$ write  $l_-(\gamma) = \Delta(\gamma)-2$ and $l_+(\gamma) = \Delta(\gamma) + 2n - 4$. We know from~\eqref{CZ_split} that $HC_i(\alpha,\gamma) = 0$ if $i\not\in [l_-(\gamma),l_+(\gamma)]$. Consider $N(k,\gamma) = \# \{ j \in \Z \mid j \geq k+1, \ l_-(\gamma^j) \leq l_+(\gamma^k) \}$, $\Delta_{\rm min} = \min_s\Delta(\gamma_s)> 0$ and note that $N(k,\gamma_s) \leq 2(n-1)/\Delta_{\rm min}$ for every $k\geq1$ and $s=1,\dots,r$. Therefore,
\begin{align*}
& \sum_{s=1}^r \frac{\hat\chi_+(\alpha,\gamma_s)}{\Delta(\gamma_s)} \\
& = \sum_{s=1}^r \lim_{m\to\infty} \frac{1}{m\Delta(\gamma_s)} \sum_{j=1}^m \sum_{i\geq0} (-1)^i\dim HC_i(\alpha,\gamma_s^j) \\
& = \sum_{s=1}^r \lim_{m\to\infty} \frac{1}{m\Delta(\gamma_s)} \sum_{j=1}^m \sum_{i=0}^{\lf l_+(\gamma_s^m) \rf} (-1)^i\dim HC_i(\alpha,\gamma_s^j) \\
& \leq \sum_{s=1}^r \lim_{m\to\infty} \frac{1}{m\Delta(\gamma_s)} \left[ \left( \sum_j \sum_{i=0}^{\lf l_+(\gamma_s^m) \rf} (-1)^i\dim HC_i(\alpha,\gamma_s^j) \right) + N(m,\gamma_s)B \right] \\
& = \sum_{s=1}^r \lim_{m\to\infty} \frac{1}{\lf l_+(\gamma_s^m) \rf} \sum_j \sum_{i=0}^{\lf l_+(\gamma_s^m) \rf} (-1)^i\dim HC_i(\alpha,\gamma_s^j) = \chi^{[\ep]}_+(\xi),
\end{align*}
where the last identity follows from~\eqref{global_mec_eq} since $\lf l_+(\gamma_s^m) \rf/m\Delta(\gamma_s) \to 1$ as $m\to\infty$, $\forall s$. The other inequality leading to the identity in Theorem~\ref{thm:mec} for the positive mean Euler characteristic is proved analogously.
\end{proof}

\begin{proof}[Proof of Corollary~\ref{cor:mec}]
When $\gamma^j$ is non-degenerate for every $j$ it is easy to see that
$$ \chi_\pm(\alpha,\gamma) = \begin{cases}
(-1)^{|\gamma|} \ \text{ if } \gamma \text{ is good} \\
(-1)^{|\gamma|} /2 \ \text{ if } \gamma \text{ is bad} \\
\end{cases}. $$
\end{proof}

\subsection{Non-hyperbolic periodic orbits}
\label{hyperbolic}

Theorem~\ref{thm: hyperbolic1} and Theorem~\ref{thm: hyperbolic2} follow from the two theorems below.

\begin{theorem}\label{index zero}
Suppose that there are finitely many simple closed orbits, all of them hyperbolic. If $\dim HC_{n-3}(\xi) < \infty$ then there is no closed orbit with Conley-Zehnder index equal to zero.
\end{theorem}

\begin{proof}
Arguing indirectly, let $\gamma$ be a closed orbit of index zero. It is well known that the index of a hyperbolic periodic orbit $\psi$ satisfies
\begin{equation}
\label{indhyp}
\cz(\psi^k) = k\cz(\psi)
\end{equation}
for every $k \in \N$. In particular, we conclude that $\cz(\gamma^k) = k\cz(\gamma) = 0$ for every $k \in \N$. By equality \eqref{indhyp} and our assumption that there are finitely many simple closed orbits, we also conclude that there are finitely many periodic orbits of index $-1$ and $1$. Hence, since the differential decreases the action, we have that a chain generated by orbits of index zero and action big enough cannot be exact. In particular, there exists $k_0 > 0$ such that every chain given by a finite sum of the form $\sum_i a_i\gamma^{k_i}$ cannot be exact as long as $k_i>k_0$ for every $i$.

Let $\psi_1,\dots,\psi_N$ be the periodic orbits of index $-1$ (these are not necessarily simple). The set of chains of degree $n-4$ can be naturally identified with $\Q^N$:
$$ \sum_{i=1}^N a_i\psi_i \longleftrightarrow (a_1,\dots,a_N) \in \Q^N. $$
Therefore, given $k \in \N$ we can identify $\partial\gamma^k$ with a vector $v_k := (a^k_1,\dots,a^k_N)$ determined by the relation $\partial\gamma^k = \sum_{i=1}^N a^k_i\psi_i$. Consequently, given $k_0 < k_1 < \cdots < k_{N+1}$ we have that $v_{k_1}, \dots, v_{k_{N+1}}$ must be linearly dependent, that is, there exist rational numbers $p_1,\dots,p_{N+1}$ such that
$$ p_1v_{k_1} + \dots + p_{N+1}v_{k_{N+1}} = 0. $$
Thus the chain $p_1\gamma^{k_1} + \dots + p_{N+1}\gamma^{k_{N+1}}$ is closed and not exact. Since one can take $k_0$ arbitrarily large, we conclude that $\dim HC_{n-3}^{[\ep]}(\xi)  = \infty$.
\end{proof}

\begin{proof}[Proof of Theorem~\ref{thm: hyperbolic2}]
Arguing indirectly, suppose that every periodic orbit is hyperbolic, and therefore nondegenerate. The hypothesis that $\dim HC_{n-3}^{[\ep]}(\xi) > 0$ implies that there exists a closed orbit $\gamma$ of index zero. This contradicts Theorem~\ref{index zero}.
\end{proof}

\begin{theorem}
\label{non-hyperbolic}
Suppose that there is no periodic orbit with Conley-Zehnder index equal to zero and that there exists a positive integer $C$ such that
\begin{equation}\label{ineq_hyp_thm}
(-1)^{n}\sum_{i=n-3}^{mC+n-3} (-1)^ib_{i}^{[\ep]}(\xi) < (-1)^{n}mC\chi_+^{[\ep]}(\xi)
\end{equation}
for every $m \in \N$. Assume further that there are finitely many simple periodic orbits with positive mean index. Then there is a non-hyperbolic closed orbit.
\end{theorem}

\begin{proof}
Arguing indirectly, suppose that every closed orbit is hyperbolic. In particular, every periodic orbit is nondegenerate. Firstly, let us show that there is at least one periodic orbit with positive mean index. Indeed, since by~\eqref{indhyp} we know that $\mu_{CZ}(\gamma)=\Delta(\gamma)$ for every $\gamma$, if there is no such orbit we would have that $b^{[\ep]}_i=0$ for every $i>n-3$ and $\chi_+^{[\ep]}(\xi)=0$. Moreover, the hypothesis that there is no periodic orbit of index zero implies that $b^{[\ep]}_{n-3}=0$. This contradicts~\eqref{ineq_hyp_thm}.

Let $\gamma_1, ..., \gamma_r$ be the simple periodic orbits with positive mean index. Let $C^\prime = \text{lcm}\{C,2\prod_{k=1}^r \Delta(\gamma_k)\}$ and $c_i(\gamma_k) = \sum_j \dim HC_i(\alpha,\gamma_k^j)$. Note that each $\Delta(\gamma_k) = \mu_{CZ}(\gamma_k) \geq 1$ is an integer by~\eqref{indhyp}. Again by~\eqref{indhyp} we conclude that $$ \sum_{i=n-3}^{m\Delta(\gamma_k)+n-3} (-1)^i c_i(\gamma_k) = \ep_k(-1)^{\Delta(\gamma_k) + n-3}m $$ for every $1\leq k \leq r$ and $m \in 2\N$, where $\ep_k = 1$ if $\gamma_k^2$ is good and $\ep_k = 1/2$ if $\gamma_k^2$ is bad. Here we used that in our situation $\gamma_k^2$ is good if, and only if, $\Delta(\gamma_k)$ is even. From this, we arrive at $$ \sum_{i=n-3}^{mC^\prime+n-3} (-1)^i c_i(\gamma_k) = \ep_k mC^\prime\frac{(-1)^{\Delta(\gamma_k) + n-3}}{\Delta(\gamma_k)}, $$ since $mC^\prime = (mC^\prime/\Delta(\gamma))\Delta(\gamma)$ and $mC^\prime/\Delta(\gamma) \in 2\N$. Hence, by Corollary~\ref{cor:mec}, $$ \sum_{k=1}^r \sum_{i=n-3}^{mC^\prime+n-3} (-1)^i c_i(\gamma_k) = mC^\prime\chi_+^{[\ep]}(\xi). $$ Now, we need the following special version of the strong Morse inequality.

\begin{lemma}
Suppose that there is no periodic orbit of index zero and every periodic orbit is non-degenerate. Then
$$ b^{[\ep]}_i - b^{[\ep]}_{i-1} + ... \pm b^{[\ep]}_{n-3}  \leq c_i - c_{i-1} + ... \pm c_{n-3} $$
for every $i\geq n-3$.
\end{lemma}

\begin{proof}
Fix an augmentation $\ep_0$ with homotopy class $[\ep]$ for the DGA of full-contact homology with $\Q$-coefficients associated to $\alpha$ and to some regular almost complex structure in $\J(\alpha)$. Since every periodic is nondegenerate and there is no periodic orbit of index zero we have that $HC^{[a,b),\ep_0}_{n-3}(\alpha) = 0$ for every $0 \leq a < b \leq \infty$ not in $\Sigma(\alpha)$. In particular, $b^{[\ep]}_{n-3} = c_{n-3} = 0$. Moreover, by the long exact sequence for filtered contact homology, the function $$ \chi_l(a,b) = \dim HC_l^{[a,b),\ep_0}(\alpha) - \dim HC_{l-1}^{[a,b),\ep_0}(\alpha) + ... \pm \dim HC_{n-2}^{[a,b),\ep_0}(\alpha) $$ is subadditive for every $l\geq n-2$. This implies the inequality $$ b^{[\ep]}_i - b^{[\ep]}_{i-1} + ... \pm b^{[\ep]}_{n-3}  = b^{[\ep]}_i - b^{[\ep]}_{i-1} + ... \mp b^{[\ep]}_{n-2} \leq c_i - c_{i-1} + ... \mp c_{n-2} = c_i - c_{i-1} + ... \pm c_{n-3} $$ for every $i\geq n-3$, as desired.
\end{proof}

By the previous lemma we get
\begin{align*}
(-1)^{n-1}mC^\prime\chi_+^{[\ep]}(\xi) & = (-1)^{n-1}\sum_{k=1}^r \sum_{i=n-3}^{mC^\prime+n-3} (-1)^i c_i(\gamma_k) \\
& \geq (-1)^{n-1}\sum_{i=n-3}^{mC^\prime+n-3} (-1)^ib_i^{[\ep]}(\xi) \\
& > (-1)^{n-1}mC^\prime\chi_+^{[\ep]}(\xi),
\end{align*}
a contradiction. Here, we used~\eqref{ineq_hyp_thm} in the last inequality. Note that the term $(-1)^{n-1}$ comes from the fact that the term with exponent $mC^\prime+n-3$ must have positive sign (notice that $C^\prime$ is even).
\end{proof}

\begin{proof}[Proof of Theorem~\ref{thm: hyperbolic1}]
The result for the positive Euler characteristic is immediate by Theorems~\ref{index zero} and~\ref{non-hyperbolic}. The argument for the negative Euler characteristic is analogous.
\end{proof}

\appendix

\section{Finite-energy curves and periodic orbits}\label{exist_orbits}

Here we revisit basic facts about pseudo-holomorphic maps proved in~\cite{93} for the contact case. Consider a stable Hamiltonian structure $\H = (\xi,X,\omega)$ defined on a manifold $N$, a compact smooth domain $K\subset N$ (possibly with boundary), and some $J \in \J(\H)$.

We will use an $\R$-invariant Riemannian metric $g_0 = da\otimes da + g$ on $\R \times N$, where $g$ is a Riemannian metric on $N$ and $a$ denotes the $\R$-coordinate. Domains in $\C$ or $\R\times S^1$ are equipped with their standard Euclidean metric. Norms of maps are taken with respect to these metrics.

The point of proving the following lemmas is that the dynamics of $X$ may be very degenerate, hence the results from~\cite{sftcomp} are not available. However, the following arguments are contained in~\cite{93}.

\begin{lemma}\label{bubb_inv}
Let $F_n = (a_n,f_n) : \R\times S^1 \to \R\times K$ be smooth $J$-holomorphic maps. Viewing $F_n(s+it) \sim F_n(s,t)$ as $i$-periodic maps on $\C$, suppose $\exists \{z_n\} \subset \C$ such that $|dF_n(z_n)|\to\infty$. There are sequences $\{z'_j\} \subset \C$, $\{\delta_j\},\{R_j\},\{d_j\} \subset \R$, and a subsequence $F_{n_j}$ such that $|z_{n_j}-z'_j| \to 0$, $\delta_j\to 0^+$, $R_j \to+\infty$, $\delta_j R_j \to +\infty$ and that the maps
\[
  \begin{array}{ccc}
    \util_j : B_{\delta_j R_j}(0)  \to \R\times N, &  & \util_j(z) = \tau_{d_j} \circ F_{n_j} (z'_j + z/R_j)
  \end{array}
\]
converge in $C^\infty_{\rm loc}$ to a $J$-holomorphic map $\util:\C\to\R\times N$ satisfying $E(\util)>0$ and $\sup_{z\in\C} |d\util(z)| < \infty$. Moreover, $E(\util)\leq \sup_nE(F_n)$.
\end{lemma}

\begin{proof}
By Hofer's Lemma (Lemma 5.12 from~\cite{sftcomp}), there exists $z'_n\in \C$, $\delta_n\to0^+$ such that $|z'_n-z_n|\to0$ and if we set $R_n = |dF_n(z'_n)|$ then $R_n\to\infty$, $\delta_nR_n\to\infty$ and $|dF_n| \leq 2R_n$ on $B_{\delta_n}(z'_n)$. We set $d_n = -a_n(z'_n)$ and $\util_n = \tau_{d_n}\circ F_n(z'_n+z/R_n)$ on $B_{\delta_nR_n}(0) \subset \C$. Thus $\util_n(0) \in 0\times N$ and $|d\util_n|\leq 2$ on $B_{\delta_nR_n}(0)$, for all $n$. The $\util_n$ satisfy $\partial_s\util_n + J(\util_n)\partial_t\util_n = 0$. Elliptic estimates provide $C^\infty_{\rm loc}$-bounds for the sequence $\{\util_n\}$ so, up to selection of a subsequence, we may assume $\util_n\to \util$ in $C^\infty_{\rm loc}$, where $\util$ is $J$-holomorphic. Then $|d\util(0)|=1$ and $|d\util(z)|\leq 2 \ \forall z$ since the same holds for $\util_n$. The inequality $E(\util) \leq \sup_nE(F_n)$ is easy to check.
\end{proof}

\begin{remark}\label{replace_domain}
The proof of Lemma~\ref{bubb_inv} shows that 
we can replace the domain $\R\times S^1$ of the maps $F_n$ by $(\C,i)$, or by $(\D,i)$ and assume that $z_n \to 0$. The conclusion is exactly the same in both cases. One can also prove a version of Lemma~\ref{bubb_inv} when the domains of the maps $F_n$ form an exhausting and increasing sequence of open subsets of $\R\times S^1$ and the sequence $z_n$ is bounded. 
\end{remark}

\begin{lemma}\label{zero_omega}
Fix $J \in \J(\H)$ and let $F = (a,f) :\C\to\R\times K$ be a non-constant $J$-holomorphic map satisfying $\int_\C f^*\omega = 0$. Then there exists a (not necessarily periodic) trajectory $x$ of $X$ and an entire function $H:\C\to\C$ such that $F = Z^x \circ H$ where the $J$-holomorphic immersion $Z^x:\C \to \R\times N$ is defined by $Z^x(s+it) = (s,x(t))$. If, in addition, $|dF|$ is bounded then $H(z) = \alpha z + \beta$ with $\alpha \not=0$.
\end{lemma}

\begin{proof}
The identity $\bar\partial_J(F)=0$ tells us that $\int_\C f^*\omega = 0 \Rightarrow f^*\omega \equiv0$, so that $df$ takes values on $\R X\circ f$ and $da(z) = 0 \Leftrightarrow df(z) = 0 \Leftrightarrow dF(z) = 0$ for every $z$. Fix $z_1,z_0 \in \C$ and any smooth curve $z(t) : (-\epsilon,1+\epsilon) \to \C$ satisfying $z(0) = z_0$ and $z(1) = z_1$. Let $x:\R\to N$ be the trajectory of $X$ satisfying $x(0) = f(z_0)$. There is a unique function $g(t)$ defined by $df(z(t)) \cdot z'(t) = g(t) X\circ f\circ z(t)$. 
Then $Y(t,p) = g(t)X(p)$ defines a time-dependent vector field $(-\epsilon,1+\epsilon)\times N \to TN$. Consider $h(t) := \int_0^t g(\tau)d\tau$. Then $f\circ z$ and $x\circ h$ solve $\beta'(t) = Y(t,\beta(t))$ with the same initial condition, and hence are equal. Since $z_1$ was arbitrary it follows that $F(\C) \subset Z^x(\C)$.

Assume $x$ is not periodic. Then there exists a unique function $H:\C\to\C$ satisfying $F=Z^x \circ H$ because $Z^x$ is 1-1 in this case. Since $Z^x$ is a $J$-holomorphic immersion we conclude, using the similarity principle, that $H$ is holomorphic, see Lemma 2.4.3 from~\cite{mcdsal}. When $x$ is periodic, the map $Z^x$ descends to an injective map $\underline Z^x$ defined on $\R\times \R/T\Z$, where $T>0$ is the minimal period of $x$. As before there exists a unique holomorphic function $\underline H: \C\to\R\times \R/T\Z$ satisfying $F = \underline Z^x \circ \underline H$ since $\underline Z^x$ is 1-1. Clearly $\underline H$ can be lifted to a holomorphic map $H:\C \to \C$ satisfying $F=Z^x \circ H$.

If $w\in\C$ and $\zeta \in T_w\C$ then there is an estimate $|\zeta| \leq k |dZ^x(w) \cdot \zeta|$, the constant $k$ being independent of $w,\zeta$. This follows very easily from the particular form of the function $Z^x$ and the $\R$-invariance of the metric $g_0$. Thus $|dH|$ is bounded if so is $|dF|$, and the conclusion follows from Liouville's Theorem.
\end{proof}

\begin{lemma}\label{plane}
If $F: \C \to \R\times K$ satisfies $\bar\partial_J(F)=0$ for some 
$J\in\J(\H)$, $\sup_{z\in\C} |dF(z)| < \infty$ and $0<E(F)<\infty$ then $E_\omega(F) >0$.
\end{lemma}

\begin{proof}
If $E_\omega(F)=0$ then, by Lemma~\ref{zero_omega}, there exists a trajectory $x$ such that $F(z) = Z^x(\alpha z+\beta)$ for some $\alpha\not=0$. This implies $E(F) = \infty$, a contradiction.
\end{proof}

%
%
%


The next lemma is an important characterization of non-constant finite-energy cylinders in cylindrical cobordisms.

\begin{lemma}\label{cylinder}
Fix $J\in\J(\H)$ and let $F =(a,f) : \R \times \R/\Z \to \R\times N$ be a non-constant finite-energy $J$-holomorphic map satisfying $E_\omega(F) = 0$. Then there exists $\gamma = (x,T) \in \P(\H)$, constants $a_0,b_0 \in \R$, and a sign $\epsilon = \pm1$ such that $F(s,t) = (\epsilon Ts+a_0,x(\epsilon Tt+b_0))$.
\end{lemma}

\begin{proof}
We can think of $F(s+it)$ as defined on $\C$ and $i$-periodic. 
We claim that $|dF|$ is bounded. If not, let $z_n$ satisfy $|dF(z_n)| \to \infty$ and consider $F_n(z) := \tau_{c_n} \circ F(z+z_n)$ with $c_n = -a(z_n)$, which is again $J$-holomorphic. By Lemma~\ref{bubb_inv} applied to $F_n$ and Lemma~\ref{plane} we can assume, up to the choice of a subsequence, that there exists $r_n \to 0^+$ and $z_n' \in\C$ satisfying $|z'_n-z_n| \to 0$ and $$ \liminf_{n\to\infty} \int_{B_{r_n}(z'_n)} f^*\omega > 0 $$ contradicting our hypotheses. By Lemma~\ref{zero_omega} we find a trajectory $x$ and constants $\epsilon = \pm1$, $T>0$, $\alpha = \epsilon T+ib$, $\beta = a_0+ib_0$ satisfying 
\[
F(s+it) = Z^x(\alpha z+\beta) = (\epsilon Ts-bt+a_0,x(bs+\epsilon Tt+b_0)).
\]
Since $F$ is $i$-periodic we have that $b=0$ and that $x$ is $T$-periodic.
\end{proof}

Assume that $\omega$ has a primitive $\alpha$ on a neighborhood of $K$ such that $\inf_K i_X\alpha >0$. We wish to show that, as in the contact case~\cite{93}, for finite-energy curves with image in $K$ it is possible to distinguish between positive/negative punctures, and non-removable punctures give periodic orbits for $X$. Note that $\alpha$ is a contact form near $K$. In the language of~\cite{hofer_ias} $X$ is Reeb-like near $K$ since it is a positive multiple of the Reeb vector field of $\alpha$.

Following~\cite{93}, if $F=(a,f):[0,+\infty) \times\R/\Z \to \R\times K$ is a $J$-holomorphic finite energy half-cylinder, for some fixed $J\in\J(\H)$, then the limit $$ m := \lim_{s\to\infty} \int_{\R/\Z} f^*\alpha $$ exists since $E_\omega(F)<\infty$. Under these assumptions the following important result due to Hofer~\cite{93} holds.

\begin{lemma}\label{hofer}
Let $\alpha$ be a primitive of $\omega$ near $K$ satisfying $\inf_K i_X\alpha>0$. Assume also that $m\not=0$, and let $\epsilon = \pm1$ be its sign. Then $a(s,t) \to \epsilon\infty$ as $s\to\infty$, and $\forall s_n\to\infty$ there exist $n_j \to \infty$, a periodic orbit $\gamma = (x,T) \in \P(\H)$ and $c\in\R$ such that $f(s_{n_j},t) \to x(\epsilon Tt+c)$ in $C^\infty$ as $j \to \infty$.
\end{lemma}

We give a proof here since the statement above can not be explicitly found in the literature. The difference with the results from~\cite{93} is that $X$ is not the Reeb vector of $\alpha$ near $K$, and $\xi$ is not a contact structure (it might even be integrable). Note that the dynamics of $X$ is {\bf not} assumed to be transversely nondegenerate.

\begin{proof}
First we show $|dF|$ is bounded. If not let $(\rho_n,t_n)$ satisfy $|dF(\rho_n,t_n)| \to \infty$ and $\rho_n\to+\infty$. Define $F_n(s,t) := F(s+\rho_n,t)$ and write $F_n = (a_n,f_n)$. It follows from $E(F) < \infty$ that $\int_C f_n^*\omega \to 0$ for every compact $C\subset \R\times \R/\Z$. A combined application of Lemma~\ref{bubb_inv} and Lemma~\ref{plane} shows that $|dF_n|$ is bounded over compact sets, contradicting $|dF_n(0,t_n)|\to\infty$.

Suppose $m>0$. Define $F_n(s,t) := \tau_{c_n} \circ F(s+s_n,t)$ with $c_n = -a(s_n,0)$. Thus, by the above, $F_n$ is $C^1_{\rm loc}$-bounded and elliptic estimates tell us it is $C^\infty_{\rm loc}$-bounded. We find $n_j \to \infty$ and a smooth $J$-holomorphic map $u:\R\times \R/\Z \to \R\times N$ such that $F_{n_j} \to u$ in $C^\infty_{\rm loc}$ as $j\to\infty$. Clearly $E(u) \leq E(F)$, $E_\omega(u)=0$ and $\text{image}(u) \subset \R\times K$. Moreover $u$ is non-constant since $\int_{0\times \R/\Z} u^*\alpha = m > 0$. By Lemma~\ref{cylinder} we have $u(s,t) = (Ts+a_0,x(Tt + b_0))$, for some $(x,T) \in \P(\H)$ contained in $K$. Here we used the fact that $\inf_K i_X\alpha>0$ to conclude that the sign in Lemma~\ref{cylinder} is $+1$. 

We now prove $a(s,t) \to +\infty$ as $s\to\infty$. Consider the mean $\bar a(s) := \int_0^1 a(s,t) dt$ and $\sigma := \min \{ T>0 \mid \exists \ \gamma \in \P(\H) \text{ contained in } K \text{ with period } T \} > 0$. We claim that $\liminf_{s\to\infty} \bar a'(s) \geq \sigma$. If not let $s_n \to \infty$ satisfy $\sup_n \bar a'(s_n) \leq \sigma-\epsilon$. Arguing as above we can assume, up to the choice of a subsequence, that $f(s_n,t) \to x(Tt+c)$ in $C^\infty$, for some $\gamma=(x,T) \in \P(\H)$ contained in $K$, and $c\in\R$. Let $\lambda$ be the $1$-form defined by~\eqref{1form}. Then
\[
  \begin{aligned}
    \sigma-\epsilon &\geq \limsup_{n\to\infty} \bar a'(s_n) =  \limsup_{n\to\infty} \int_0^1 a_s(s_n,t)dt \\
    & = \limsup_{n\to\infty} \int_0^1 \lambda(f(s_n,t)) \cdot f_t(s_n,t)dt = T \geq \sigma.
  \end{aligned}
\]
This contradiction proves our claim. Thus $\bar a(s) \to +\infty$ as $s\to+\infty$. We conclude the case $m>0$ since $|a(s,t)-\bar a(s)|$ is uniformly bounded by $\sup |a_t|<\infty$. The case $m<0$ is treated similarly.
\end{proof}

The following statement, left with no proof, follows easily from the assumption that $\omega$ has a primitive $\alpha$ on $K$ satisfying $\inf_K i_X\alpha > 0$.

\begin{lemma}\label{std_props}
Suppose that $\omega$ has a primitive $\alpha$ on $K$ satisfying $\inf_K i_X\alpha > 0$. If $S$ is a closed Riemann surface and $F = (a,f) : S\setminus M \to \R\times K$ is a non-constant finite-energy $J$-holomorphic, where $M\subset S$ is finite, then $M\neq \emptyset$ and at least one point of $M$ is a non-removable positive puncture.
\end{lemma}

Finally, we prove a stronger version of Lemma~\ref{bubb_inv} 
needed in the arguments throughout the paper. Consider a stable Hamiltonian structure $\H = (\xi,X,\omega)$ defined on a manifold $N$ and a compact smooth domain $K\subset N$ with boundary. We fix $L>0$, $J \in \J(\H)$, and consider $\H_n^\pm = (\xi^\pm_n,X^\pm_n,\omega^\pm_n) \to \H$ in $C^\infty_{\rm loc}$, $J_n^\pm \in \J(\H_n^\pm)$ and $\jbar_n \in \J_L(J_n^-,J_n^+)$ such that 
\[
J_n^-,J_n^+,\jbar_n \to J
\]
in $C^\infty$ (weak or strong) as $n\to\infty$. 
Suppose $\Omega_n$ are symplectic forms on $[-L,L]\times K$ compatible with $\jbar_n$ which agree with $\omega^\pm_n$ on $T(\{\pm L\} \times K)$ up to constant positive multiples, and converge to a fixed symplectic form $\Omega$ on $[-L,L]\times K$ compatible with $J$, as $n \to \infty$. Below we use $\Omega_n$ to define the energy of $\jbar_n$-holomorphic maps as in~\eqref{energy_cobordism}.

\begin{lemma}\label{bubb}
Assume that $\omega$ has a primitive $\alpha$ near $K$ satisfying $\inf_K i_X\alpha>0$, and let $F_n = (a_n,f_n) : \R\times S^1 \to \R\times K$ be smooth $\jbar_n$-holomorphic cylinders satisfying $E(F_n)\leq C$, $\forall n$. Viewing $F_n(s,t) \sim F_n(s+it)$ as $i$-periodic maps on $\C$, suppose $\exists \{z_n\} \subset \C$ such that $|dF_n(z_n)|\to\infty$. Then one finds sequences $\{z'_j\} \subset \C$, $\{\delta_j\},\{R_j\},\{d_j\} \subset \R$, and a subsequence $F_{n_j}$ such that $|z_{n_j}-z'_j| \to 0$, $\delta_j\to 0^+$, $R_j \to+\infty$, $\delta_j R_j \to +\infty$, and that the maps
\[
  \begin{array}{ccc}
    \util_j : B_{\delta_j R_j}(0)  \to \R\times N, &  & \util_j(z) = \tau_{d_j} \circ F_{n_j} (z'_j + z/R_j)
  \end{array}
\]
$C^\infty_{\rm loc}$-converge to a $J$-holomorphic map $\util:\C\to\R\times N$ satisfying $0<E(\util) \leq M C$ and $\sup_{z\in\C} |d\util(z)| < \infty$, for some $M>0$ independent of $\{F_n\}$.
\end{lemma}

In the above statement $E(\util)$ is the energy~\eqref{energy_cylindrical}.

\begin{proof}
Arguing as in the proof of Lemma~\ref{bubb_inv}, using Hofer's lemma, we find $\util$ as the $C^\infty_{\rm loc}$-limit of a sequence of the $\jbar_n$-holomorphic maps $\util_j(z) = \tau_{d_j} \circ F_{n_j} (z'_j + z/R_j)$ defined on balls $B_{\delta_jR_j}(0)$, where the $z'_j$, $\delta_j$, $d_j$, $R_j$ satisfy the requirements in the statement. In particular, $\util$ is nonconstant and has bounded derivative.

To check the assertion about the energy, consider first the case $d_j$ is bounded. Hence, up to selection of a subsequence, $\exists d\in\R$ such that $0\leq d_j-d \to 0$ or $0\geq d_j-d \to 0$. We treat the first case, the other is analogous. In this case the characteristic function of $(-L+d_j,L+d_j]$ converges pointwise to that of $(-L+d,L+d]$. By Fatou's lemma
\begin{equation*}
\begin{aligned}
\int_{\util^{-1}((-L+d,L+d]\times N)} \util^*(\tau_{-d}^*\Omega) & \leq \liminf\int_{\util_j^{-1}((-L+d_j,L+d_j]\times N)} \util_j^*(\tau_{-d_j}^*\Omega_{n_j}) \\
&\leq \liminf \int_{F_{n_j}^{-1}((-L,L]\times N)} F_{n_j}^*\Omega_{n_j} \leq E(F_{n_j}) \leq C.
\end{aligned}
\end{equation*}
Analogously we can analyze the other relevant integrals to obtain an estimate for the following modified energy $E_{d,\Omega}(\util)$ of $\util$:
\begin{equation*}
\begin{aligned}
E_{d,\Omega}(\util) := \sup_{\phi\in\Lambda} & \int_{\util^{-1}((\R\setminus(-L+d,L+d])\times N)} \util^*(d\phi\wedge\lambda + \omega) \\
&+ \int_{\util^{-1}((-L+d,L+d]\times N)} \util^*(\tau_{-d}^*\Omega) \leq C.
\end{aligned}
\end{equation*}

\noindent {\bf Claim:} $|a(z)|\to\infty$ as $|z|\to\infty$.

\begin{proof}[Proof of Claim]
Using $E_{d,\Omega}(\util)<\infty$ we find $E_\omega(\util)<\infty$. 
Considering $\vtil(s,t) = \util(e^{2\pi(s+it)})$ defined on $\R\times S^1$, we must have that $d\vtil$ bounded. If not we get a contradiction in the following way. View $\vtil(s+it) \sim \vtil(s,t)$ as an $i$-periodic map on $\C$, and suppose that $\{\zeta_k\}\subset\C$ satisfies $|\zeta_k|\to\infty$ and $|d\vtil(\zeta_k)|\to\infty$. Then arguing as in Lemma~\ref{bubb_inv}, using Hofer's lemma, we find $\eta_k,\delta_k,R_k\in\R$ such that $\delta_k \to 0^+$, $R_k\to\infty$, $\delta_kR_k\to\infty$ and, perhaps after changing $\zeta_k$, the maps $\zeta \in B_{\delta_kR_k}(0) \mapsto \tau_{\eta_k}\circ \vtil(\zeta_k+\zeta/R_k)$ converge in $C^\infty_{\rm loc}$ to a nonconstant $J$-holomorphic plane $\widetilde w$ with $|d\widetilde w|$ bounded. It is easy to estimate $E_{d,\Omega}(\widetilde w)<\infty$ from $E_{d,\Omega}(\util)<\infty$, and we must have $E_\omega(\widetilde w)=0$ since 
$|\Re [\zeta_k]| \to \infty$ and $\delta_k\to0$. Then applying Lemma~\ref{zero_omega} we conclude that $\widetilde w(z) = Z^x(\alpha z+\beta)$ where $x$ is some $X$-trajectory. But this gives $E_{d,\Omega}(\widetilde w)=\infty$, a contradiction which shows that $d\vtil$ is bounded. Now we conclude using the Monotonicity Lemma in a standard fashion. We cannot have $a(z)$ bounded since, otherwise, the Removable Singularity Theorem would imply together with Stokes Theorem that $E_\omega(\util)=0$. Here we used that $\omega$ is exact by assumption. As explained above, since $d\util$ is bounded, Lemma~\ref{zero_omega} would give $E_{d,\Omega}(\util)=\infty$, an absurd. Thus, if $|a(z)|$ does not explode as $|z|\to\infty$ we find sequences sequences $(s_k,t_k)$ and $(s_k',t_k')$ such that $s_k<s'_k<s_{k+1}$ and, writing $\vtil=(b,v)$, we have $\sup_kb(s_k,t_k)<\infty$ and $b(s_k',t_k') \to\infty$. Since $d\vtil$ is bounded we have $\sup_k\sup_{t}b(s_k,t)<\infty$ and $\inf_t b(s_k',t) \to\infty$. Choose $\phi_0 \in \Lambda$ satisfying $\phi_0,\phi_0'>0$. Then $\Omega_0 = d(\phi_0\lambda) + A\omega$ is a symplectic form on $\R\times N$ when $A>0$ is large enough because $\ker\omega\subset\ker d\lambda$. Clearly $\Omega_0$ is compatible with $J$. Hence, applying the Monotonicity Lemma at a fixed compact piece of $\R\times N$, see~\cite{hummel}, we find smooth compact domains $S_k \subset \C$ such that $\inf\{|z|:z\in S_k\} \to \infty$ as $k\to\infty$, and that $\liminf_k \int_{S_k} \util^*\Omega_0 > 0$, a contradiction to $\int_\C\util^*\Omega_0<\infty$.
\end{proof}

For any open set $V\subset\C$ the integrals below (taken with respect to the standard orientation of $\C$) converge and we have estimates
\[
\left| \int_V u^*d\lambda \right| \leq C' \int_V u^*\omega, \ \ \ \left| \int_V \util^*(\phi d\lambda) \right| \leq C' \|\phi\|_\infty \int_V u^*\omega
\]
with $C'$ independent of $\util$. This is so because $\ker\omega\subset\ker d\lambda$, $\omega$ is symplectic in $\xi$ and $u$ is $J$-holomorphic. Let $x_k^+ < L+d$, $x_k^- > -L+d$ be regular values of $a(z)$ satisfying $x_k^\pm \to \pm L+d$. Set $U_k = \util^{-1}([x_k^-,x_k^+]\times N) = a^{-1}([x_k^-,x_k^+])$. Hence $U_k \subset \C$ are smooth closed sets, their boundaries split as $\partial U_k = \partial^+U_k \cup \partial^-U_k$, where $\partial^\pm U_k = a^{-1}(x_k^\pm)$ is oriented as the boundary of $U_k$. We have $\partial^\pm U_k = \pm \partial\{a\leq x_k^\pm\}$ with orientations, and also that $\cup_k \{a\leq x_k^+\}$ is a bounded set since $|a(z)|\to\infty$ as $|z|\to\infty$. Fixing $\phi\in\Lambda$ we compute
\[
\begin{aligned} 
\int_{U_k} \util^*(d\phi\wedge\lambda) &= \int_{U_k} \util^*d(\phi\lambda) - \int_{U_k} \util^*(\phi d\lambda) \\
&= \phi(x_k^+) \int_{\partial^+U_k} u^*\lambda + \phi(x_k^-) \int_{\partial^-U_k} u^*\lambda - \int_{U_k} \util^*(\phi d\lambda) \\
& \leq \left| \int_{\{a\leq x_k^+\}} u^*d\lambda \right| + \left| \int_{\{a\leq x_k^-\}} u^*d\lambda \right| + \left| \int_{U_k} \util^*(\phi d\lambda) \right| \\
&\leq 3C'E_\omega(\util).
\end{aligned}
\]
Letting $k\to\infty$ we get via the monotone convergence theorem
\[
\int_{\util^{-1}((-L+d,L+d)\times N)} \util^*(d\phi\wedge\lambda) \leq 3C' E_\omega(\util).
\]
Now note that
\[
\int_{\util^{-1}(\{L+d\}\times N)} \util^*(d\phi\wedge\lambda) = 0.
\]
Indeed, define the smooth function $h:\C\to\R$ by $\util^*(d\phi\wedge\lambda) = h(s+it)ds\wedge dt$. Consider the (Lebesgue) measurable set $E = \util^{-1}(\{L+d\}\times N)$. The Lebesgue measure of $E \cap \{h \neq 0\}$ is zero since it is a $1$-dimensional submanifold of $\C$. To see this note that $\util$ is transverse to $\{L+d\}\times N$ at a given point $z \in E$ where $h(z)\neq0$ because $d\phi\wedge\lambda$ vanishes on $T(\{L+d\}\times N)$. Hence $h$ integrates to zero over $E$, implying the above identity. It follows that
\[
\int_{\util^{-1}((-L+d,L+d]\times N)} \util^*(d\phi\wedge\lambda) \leq 3C' E_\omega(\util).
\]
This estimate holds independently of $\phi \in \Lambda$.

It is not hard to get an estimate $E_\omega(\util) \leq C'' E_{d,\Omega}(\util)$ for some $C''>0$ independent of $d$, $\Omega$ and $\util$. Thus, combining all these facts we get $M>0$ independent of $d$, $\Omega$ and $\util$ such that $E(\util) \leq M E_{d,\Omega}(\util)$. The case $d_n$ is unbounded is easier. 
\end{proof}

\end{document}